\documentclass[11pt,letterpaper]{amsart}
\usepackage{amsmath,amsfonts,amssymb,amsthm}
\usepackage[shortalphabetic,abbrev,nobysame]{amsrefs}
\usepackage{mathrsfs}
\usepackage[all,cmtip]{xy}
\usepackage{tikz-cd}
 \usepackage{relsize}
\usepackage{hyperref}
\usepackage{accents}
\usepackage{xcolor}
\usepackage{soul}\usepackage{calligra}
\usepackage{enumitem}

\allowdisplaybreaks

\DeclareMathAlphabet{\mathcalligra}{T1}{calligra}{m}{n}

\DeclareMathOperator{\Res}{Res}

\DeclareMathOperator{\Hom}{Hom}

\DeclareMathOperator{\Sing}{Sing}

\DeclareMathOperator{\Span}{Span}

\DeclareMathOperator{\End}{End}

\DeclareMathOperator{\diag}{diag}

\begin{document}

\newtheorem{thm}{Theorem}[section]
\newtheorem{prop}[thm]{Proposition}
\newtheorem{coro}[thm]{Corollary}
\newtheorem{conj}[thm]{Conjecture}
\newtheorem{example}[thm]{Example}
\newtheorem{lem}[thm]{Lemma}
\newtheorem{rem}[thm]{Remark}
\newtheorem{hy}[thm]{Hypothesis}
\newtheorem*{acks}{Acknowledgements}
\theoremstyle{definition}
\newtheorem{de}[thm]{Definition}
\newtheorem{ex}[thm]{Example}

\newtheorem{convention}[thm]{Convention}

\newtheorem{bfproof}[thm]{{\bf Proof}}

\newcommand{\C}{{\mathbb{C}}}
\newcommand{\Z}{{\mathbb{Z}}}
\newcommand{\N}{{\mathbb{N}}}
\newcommand{\Q}{{\mathbb{Q}}}
\newcommand{\te}[1]{\mbox{#1}}
\newcommand{\set}[2]{{
    \left.\left\{
        {#1}
    \,\right|\,
        {#2}
    \right\}
}}
\newcommand{\sett}[2]{{
    \left\{
        {#1}
    \,\left|\,
        {#2}
    \right\}\right.
}}

\newcommand{\choice}[2]{{
\left[
\begin{array}{c}
{#1}\\{#2}
\end{array}
\right]
}}
\def \<{{\langle}}
\def \>{{\rangle}}

\def\({\left(}

\def\){\right)}

\def \:{\mathopen{\overset{\circ}{
    \mathsmaller{\mathsmaller{\circ}}}
    }}
\def \;{\mathclose{\overset{\circ}{\mathsmaller{\mathsmaller{\circ}}}}}

\newcommand{\overit}[2]{{
    \mathop{{#1}}\limits^{{#2}}
}}
\newcommand{\belowit}[2]{{
    \mathop{{#1}}\limits_{{#2}}
}}

\newcommand{\wt}[1]{\widetilde{#1}}

\newcommand{\wh}[1]{\widehat{#1}}

\newcommand{\wck}[1]{\reallywidecheck{#1}}

\newcommand{\no}[1]{{
    \mathopen{\overset{\circ}{
    \mathsmaller{\mathsmaller{\circ}}}
    }{#1}\mathclose{\overset{\circ}{\mathsmaller{\mathsmaller{\circ}}}}
}}

\newlength{\dhatheight}
\newcommand{\dwidehat}[1]{%
    \settoheight{\dhatheight}{\ensuremath{\widehat{#1}}}%
    \addtolength{\dhatheight}{-0.45ex}%
    \widehat{\vphantom{\rule{1pt}{\dhatheight}}%
    \smash{\widehat{#1}}}}
\newcommand{\dhat}[1]{%
    \settoheight{\dhatheight}{\ensuremath{\hat{#1}}}%
    \addtolength{\dhatheight}{-0.35ex}%
    \hat{\vphantom{\rule{1pt}{\dhatheight}}%
    \smash{\hat{#1}}}}

\newcommand{\dwh}[1]{\dwidehat{#1}}

\newcommand{\ck}[1]{\check{#1}}

\newcommand{\dis}{\displaystyle}

\newcommand{\pd}[1]{\frac{\partial}{\partial {#1}}}

\newcommand{\pdiff}[2]{\frac{\partial^{#2}}{\partial #1^{#2}}}


\newcommand{\g}{{\frak g}}
\newcommand{\fg}{\g}
\newcommand{\ff}{{\frak f}}
\newcommand{\f}{\ff}
\newcommand{\gc}{{\bar{\g'}}}
\newcommand{\h}{{\frak h}}
\newcommand{\cent}{{\frak c}}
\newcommand{\notc}{{\not c}}
\newcommand{\Loop}{{\mathcal L}}
\newcommand{\G}{{\mathcal G}}
\newcommand{\D}{\mathcal D}
\newcommand{\T}{\mathcal T}
\newcommand{\Free}{\mathcal F}
\newcommand{\Cfk}{\mathcal C}
\newcommand{\nil}{\mathfrak n}
\newcommand{\al}{\alpha}
\newcommand{\alck}{\al^\vee}
\newcommand{\be}{\beta}
\newcommand{\beck}{\be^\vee}
\newcommand{\ssl}{{\mathfrak{sl}}}
\newcommand{\id}{\te{id}}
\newcommand{\rtu}{{\xi}}
\newcommand{\period}{{N}}
\newcommand{\half}{{\frac{1}{2}}}
\newcommand{\quar}{{\frac{1}{4}}}
\newcommand{\oct}{{\frac{1}{8}}}
\newcommand{\hex}{{\frac{1}{16}}}
\newcommand{\reciprocal}[1]{{\frac{1}{#1}}}
\newcommand{\inverse}{^{-1}}
\newcommand{\inv}{\inverse}
\newcommand{\SumInZm}[2]{\sum\limits_{{#1}\in\Z_{#2}}}
\newcommand{\uce}{{\mathfrak{uce}}}
\newcommand{\Rcat}{\mathcal R}
\newcommand{\cS}{{\mathcal{S}}}


\newcommand{\orb}[1]{|\mathcal{O}({#1})|}
\newcommand{\up}{_{(p)}}
\newcommand{\uq}{_{(q)}}
\newcommand{\upq}{_{(p+q)}}
\newcommand{\uz}{_{(0)}}
\newcommand{\uk}{_{(k)}}
\newcommand{\nsum}{\SumInZm{n}{\period}}
\newcommand{\ksum}{\SumInZm{k}{\period}}
\newcommand{\overN}{\reciprocal{\period}}
\newcommand{\df}{\delta\left( \frac{\xi^{k}w}{z} \right)}
\newcommand{\dfl}{\delta\left( \frac{\xi^{\ell}w}{z} \right)}
\newcommand{\ddf}{\left(D\delta\right)\left( \frac{\xi^{k}w}{z} \right)}

\newcommand{\ldfn}[1]{{\left( \frac{1+\xi^{#1}w/z}{1-{\xi^{#1}w}/{z}} \right)}}
\newcommand{\rdfn}[1]{{\left( \frac{{\xi^{#1}w}/{z}+1}{{\xi^{#1}w}/{z}-1} \right)}}
\newcommand{\ldf}{{\ldfn{k}}}
\newcommand{\rdf}{{\rdfn{k}}}
\newcommand{\ldfl}{{\ldfn{\ell}}}
\newcommand{\rdfl}{{\rdfn{\ell}}}

\newcommand{\kprod}{{\prod\limits_{k\in\Z_N}}}
\newcommand{\lprod}{{\prod\limits_{\ell\in\Z_N}}}
\newcommand{\E}{{\mathcal{E}}}
\newcommand{\F}{{\mathcal{F}}}

\newcommand{\Etopo}{{\mathcal{E}_{\te{topo}}}}

\newcommand{\Ye}{{\mathcal{Y}_\E}}

\newcommand{\rh}{{{\bf h}}}
\newcommand{\rp}{{{\bf p}}}
\newcommand{\rrho}{{{\pmb \varrho}}}
\newcommand{\ral}{{{\pmb \al}}}

\newcommand{\comp}{{\mathfrak{comp}}}
\newcommand{\ctimes}{{\widehat{\boxtimes}}}
\newcommand{\ptimes}{{\widehat{\otimes}}}
\newcommand{\ptimeslt}{{
{}_{\te{t}}\ptimes
}}
\newcommand{\ptimesrt}{{\ot_{\te{t}} }}
\newcommand{\ttp}[1]{{
    {}_{{#1}}\ptimes
}}
\newcommand{\bigptimes}{{\widehat{\bigotimes}}}
\newcommand{\bigptimeslt}{{
{}_{\te{t}}\bigptimes
}}
\newcommand{\bigptimesrt}{{\bigptimes_{\te{t}} }}
\newcommand{\bigttp}[1]{{
    {}_{{#1}}\bigptimes
}}

\newcommand{\ot}{\otimes}
\newcommand{\Ot}{\bigotimes}
\newcommand{\bt}{\boxtimes}

\newcommand{\affva}[1]{V_{\wh\g}\(#1,0\)}
\newcommand{\saffva}[1]{L_{\wh\g}\(#1,0\)}
\newcommand{\saffmod}[1]{L_{\wh\g}\(#1\)}

\newcommand{\otcopies}[2]{\belowit{\underbrace{{#1}\ot \cdots \ot {#1}}}{{#2}\te{-times}}}

\newcommand{\wtotcopies}[3]{\belowit{\underbrace{{#1}\wh\ot_{#2} \cdots \wh\ot_{#2} {#1}}}{{#3}\te{-times}}}


\newcommand{\tar}{{\mathcal{DY}}_0\(\mathfrak{gl}_{\ell+1}\)}
\newcommand{\U}{{\mathcal{U}}}
\newcommand{\htar}{\mathcal{DY}_\hbar\(A\)}
\newcommand{\hhtar}{\widetilde{\mathcal{DY}}_\hbar\(A\)}
\newcommand{\htarz}{\mathcal{DY}_0\(\mathfrak{gl}_{\ell+1}\)}
\newcommand{\hhtarz}{\widetilde{\mathcal{DY}}_0\(A\)}
\newcommand{\qhei}{\U_\hbar\left(\hat{\h}\right)}
\newcommand{\n}{{\mathfrak{n}}}
\newcommand{\vac}{{{\bf 1}}}
\newcommand{\vtar}{{{
    \mathcal{V}_{\hbar,\tau}\left(\ell,0\right)
}}}

\newcommand{\qtar}{
    \U_q\(\wh\g_\mu\)}
\newcommand{\rk}{{\bf k}}

\newcommand{\hctvs}[1]{Hausdorff complete linear topological vector space}
\newcommand{\hcta}[1]{Hausdorff complete linear topological algebra}
\newcommand{\ons}[1]{open neighborhood system}
\newcommand{\B}{\mathcal{B}}
\newcommand{\rx}{{\bf x}}
\newcommand{\re}{{\bf e}}
\newcommand{\rphi}{{\boldsymbol{ \phi}}}

\newcommand{\der}{\mathcal D}

\newcommand{\prodlim}{\mathop{\prod_{\longrightarrow}}\limits}

\newcommand{\lp}[1]{\mathcal L(f)}

\newcommand{\cha}{\check a}
\newcommand{\chh}{\check \h}


\makeatletter
\renewcommand{\BibLabel}{%
    \Hy@raisedlink{\hyper@anchorstart{cite.\CurrentBib}\hyper@anchorend}%
    [\thebib]%
}
\@addtoreset{equation}{section}
\def\theequation{\thesection.\arabic{equation}}
\makeatother \makeatletter


\title[Quantum affine VA]{Quantum affine vertex algebras associated to untwisted quantum affinization algebras}

\author{Fei Kong}
\address{Key Laboratory of Computing and Stochastic Mathematics (Ministry of Education), School of Mathematics and Statistics, Hunan Normal University, Changsha, China 410081} \email{kongmath@hunnu.edu.cn}

\subjclass[2020]{17B69, 17B37} \keywords{Quantum vertex algebra, affine vertex algebra, {$\phi$}-coordinated module, quantum affine algebra, quantum affinization}

\begin{abstract}
Let $\U_\hbar(\hat\fg)$ be the untwisted affinization of a symmetrizable quantum Kac-Moody algebra $\U_\hbar(\fg)$.
For $\ell\in\C$, we construct an $\hbar$-adic quantum vertex algebra $V_{\hat\fg,\hbar}(\ell,0)$, and establish a one-to-one correspondence between $\phi$-coordinated $V_{\hat\fg,\hbar}(\ell,0)$-modules and restricted $\U_\hbar(\hat\fg)$-modules of level $\ell$.
Suppose that $\ell$ is a positive integer. We construct a quotient $\hbar$-adic quantum vertex algebra $L_{\hat\fg,\hbar}(\ell,0)$ of $V_{\hat\fg,\hbar}(\ell,0)$, and establish a one-to-one correspondence between certain $\phi$-coordinated $L_{\hat\fg,\hbar}(\ell,0)$-modules and restricted integrable $\U_\hbar(\hat\fg)$-modules of level $\ell$.
Suppose further that $\fg$ is of finite type. We prove that $L_{\hat\fg,\hbar}(\ell,0)/\hbar L_{\hat\fg,\hbar}(\ell,0)$ is isomorphic to the simple affine vertex algebra $L_{\hat\fg}(\ell,0)$.
\end{abstract}
\maketitle

\section{Introduction}

Let $\fg$ be a finite dimensional simple Lie algebra, and let $\hat\fg=\fg\ot\C[t,t\inv]\oplus\C c$ be the corresponding affine Kac-Moody Lie algebra.
For any complex number $\ell\in\C$, we define the induced module
\begin{align*}
  V_{\hat\fg}(\ell,0)=\U(\hat\fg)\ot_{\U(\hat\fg_+)}\C_\ell,\quad\te{where}\,\,
  \hat\fg_+=\fg\ot\C[t]\oplus\C c
\end{align*}
and $\C_\ell=\C$ is a $\hat\fg_+$-module with $\fg\ot\C[t]$ acting trivially
and $c$ acting as scalar $\ell$.
The induced module $V_{\hat\fg}(\ell,0)$ carries a vertex algebra structure, and the category of $V_{\hat\fg}(\ell,0)$-modules is naturally isomorphic to the category of restricted $\hat\fg$-modules of level $\ell$ \cite{FZ,Li-local,LL}.
If $\ell$ is a positive integer, then the unique simple quotient $\hat\fg$-module $L_{\hat\fg}(\ell,0)$ of $V_{\hat\fg}(\ell,0)$ is integrable.
Moreover, $L_{\hat\fg}(\ell,0)$ carries a simple quotient vertex algebra structure, and the category of $L_{\hat\fg}(\ell,0)$-modules is naturally isomorphic to the category of restricted integrable $\hat\fg$-modules of level $\ell$ \cite{FZ,DL,MP1,MP2,DLM,LL}.

Denote by $\U_\hbar(\hat\fg)$ the (untwisted) quantum affine algebra associated to $\hat\fg$ (\cite{Dr-hopf-alg} and \cite{jim-start}).
Like the affinization realization of affine Kac-Moody Lie algebras, Drinfeld provided a quantum affinization realization of $\U_\hbar(\hat\fg)$ in \cite{Dr-new}.
Based on the Drinfeld presentation, Frenkel and Jing constructed vertex representations
for simply-laced untwisted quantum affine algebras in \cite{FJ-vr-qaffine}.
In the very paper, they formulated a fundamental problem of
developing certain ``quantum vertex algebra theory'' associated to quantum affine algebras, in parallel with the connection between affine Kac-Moody Lie algebras and vertex algebras.

As one of the fundamental works, Etingof and Kazhdan (\cite{EK-qva}) developed a theory of quantum vertex operator algebras in the sense of formal deformations of vertex algebras.
The most visible difference between these and vertex algebras is that the usual locality is replaced by the $S$-locality. Such $S$-locality is controlled by a rational quantum Yang-Baxter operator.
Partly motivated by the work of Etingof and Kazhdan, H. Li
conducted a series of studies.
While vertex algebras are analogues of commutative associative algebras, H. Li introduced the notion of
nonlocal vertex algebras \cite{Li-nonlocal}, which are analogues of noncommutative associative algebras. A nonlocal vertex algebra is a weak quantum vertex algebra \cite{Li-nonlocal} if it satisfies the $S$-locality. In addition, it becomes a quantum vertex algebra \cite{Li-nonlocal} if the $S$-locality is controlled by a rational quantum Yang-Baxter operator.
Mainly in order to associate quantum vertex algebras to quantum affine algebras, a theory of $\phi$-coordinated quasi modules was developed in \cite{Li-phi-coor, Li-G-phi}.
The $\hbar$-adic counterparts of these notions were introduced in \cite{Li-h-adic}.
In this framework, a quantum vertex operator algebra in sense of Etingof-Kazhdan is an $\hbar$-adic quantum vertex algebra whose classical limit is a vertex algebra.

In the pioneer work \cite{EK-qva}, Etingof and Kazhdan constructed quantum vertex operator algebras as formal deformations of $V_{\hat{\mathfrak gl}_n}(\ell,0)$ and $V_{\hat{\mathfrak sl}_n}(\ell,0)$, by using the $R$-matrix type relations given in \cite{RS-RTT}.
Recently, Butorac, Jing and Ko\v{z}i\'{c} (\cite{BJK-qva-BCD}) extended Etingof-Kazhdan's construction to type $B$, $C$ and $D$ rational $R$-matrices. The modules of these quantum vertex operator algebras are in one-to-one correspondence with restricted modules for the corresponding Yangian doubles (see \cite{K-qva-phi-mod-BCD}).
Based on the $R$-matrix presentation of quantum affine algebras (see \cite{DF-qaff-RTT-Dr,JLM-qaff-RTT-Dr-BD,JLM-qaff-RTT-Dr-C}),
Ko\v{z}i\'{c} constructed the quantum vertex operator algebras associated with trigonometric $R$-matrices of types $A$, $B$, $C$ and $D$ (\cite{Kozic-qva-tri-A, K-qva-phi-mod-BCD}), and established a one-to-one correspondence between $\phi$-coordinated modules and restricted modules for quantum affine algebras.

In \cite{JKLT-Defom-va}, N. Jing, H. Li, S. Tan and I developed a method to construct desired quantum vertex algebras.
Let $V$ be a nonlocal vertex algebra, and let $(H,\rho,\tau)$ be a triple consisting of a cocommutative nonlocal vertex bialgebra $H$, a right $H$-comodule nonlocal vertex algebra structure $\rho$ on $V$, and
a ``compatible'' $H$-module nonlocal vertex algebra structure $\tau$ on $V$.
Then a new nonlocal vertex algebra $\mathfrak D_\tau^\rho(V)$ with $V$ as the underlying space was obtained.
It was also proved that under certain conditions, $\mathfrak D_\tau^\rho(V)$ is a quantum vertex algebra. On the representation side, it was proved that for any $\phi$-coordinated $V$-module $W$ with a ``compatible'' $\phi$-coordinated $H$-module structure, there exists a $\phi$-coordinated $\mathfrak D_\tau^\rho(V)$-module structure on $W$.
Later in \cite{JKLT-Quantum-lattice-va}, we introduced an $\hbar$-adic version of this construction, and constructed a family of $\hbar$-adic quantum vertex algebras $V_L[[\hbar]]^\eta$ as formal deformations of the lattice vertex algebras $V_L$.
When $L$ is a root lattice of $A$, $D$, $E$ type, we established a natural connection between twisted quantum affine algebras and equivariant $\phi$-coordinated quasi modules for $V_L[[\hbar]]^\eta$ with suitable $\eta$.

It is remarkable that Drinfeld's quantum affinization process can be extended to general symmetrizable quantum Kac-Moody algebras (\cite{GKV,J-KM,Naka-quiver,CJKT-qeala-II-twisted-qaffinization}).
We denote by $\U_\hbar(\hat\fg)$ the untwisted quantum affinization of a symmetrizable quantum Kac-Moody algebra $\U_\hbar(\fg)$.
In this paper, we use the language of Drinfeld presentations to construct the $\hbar$-adic quantum vertex algebra $V_{\hat\fg,\hbar}(\ell,0)$ and establish a one-to-one correspondence between restricted $\U_\hbar(\fg)$-modules of level $\ell$ and $\phi$-coordinated $V_{\hat\fg,\hbar}(\ell,0)$-modules.
Assume further that $\ell$ is a positive integer. We will study a quotient $\hbar$-adic quantum vertex algebra $L_{\hat\fg,\hbar}(\ell,0)$ of $V_{\hat\fg,\hbar}(\ell,0)$,
and establish a one-to-one correspondence between restricted integrable $\U_\hbar(\hat\fg)$-modules of level $\ell$ and $\phi$-coordinated weight modules of $L_{\hat\fg,\hbar}(\ell,0)$.
Finally, when $\fg$ is of finite type, we show that $L_{\hat\fg,\hbar}(\ell,0)/\hbar L_{\hat\fg,\hbar}(\ell,0)\cong L_{\hat\fg}(\ell,0)$ as vertex algebras.

Now we provide some detailed information.
Let $A=(a_{ij})_{i,j\in I}$ be a symmetrizable generalized Cartan matrix (GCM).
Then there are unique relatively prime positive integers $r_i$ ($i\in I$) such that $DA$ is symmetric with $D=\diag\{r_i\}_{i\in I}$.
Let $\fg=[\fg(A),\fg(A)]$ be the derived subalgebra of the Kac-Moody Lie algebra $\fg(A)$,
and let $\hat\fg$ be an ``affinization Lie algebra'' associated to $\fg$ (see Definition \ref{de:MRY}).
Recall that the (untwisted) quantum affinization algebra (see \cite{CJKT-qeala-II-twisted-qaffinization}) $\U_\hbar(\hat\fg)$ is the $\C[[\hbar]]$-algebra topologically generated by
\begin{eqnarray}\label{eq:tqagenerators}
\set{ h_{i,q}(n),\  x^\pm_{i,q}(n),\ c}
{
   i\in I, n\in\Z
},
\end{eqnarray}
and subject to the relations in  terms of the generating functions
 \begin{eqnarray*}
 &&\phi_{i,q}^\pm(z)=q^{\pm h_{i,q}(0)}\, \te{exp}
    \left(
        \pm (q-q\inverse)\sum\limits_{\pm m> 0}h_{i,q}(m)z^{-m}
    \right),\\
 &&x^\pm_{i,q}(z)=\sum\limits_{m\in \Z} x^\pm_{i,q}(m)z^{-m},
\end{eqnarray*}
The relations are ($i,j\in I$):
\begin{align*}
&\te{(Q1)  }&&[c,\phi_{i,q}^\pm(z)]=0=[\phi_{i,q}^\pm(z),\phi_{j,q}^\pm(w)]=[c,x_{i,q}^\pm(z)],\\
&\te{(Q2) }&& \phi^+_{i,q}(z_1)\phi^-_{j,q}(z_2)=\phi^-_{j,q}(z_2)\phi^+_{i,q}(z_1)
    \iota_{z_1,z_2}g_{ij,q}(q^{rc} z_2/z_1)\inverse g_{ij,q}(q^{-rc}z_2/z_1),\\
&\te{(Q3) }&&\phi^+_{i,q}(z_1)x^\pm_{j,q}(z_2)=x^\pm_{j,q}(z_2)\phi^+_{i,q}(z_1)
    \iota_{z_1,z_2}g_{ij,q}(q^{\mp \half rc}z_2/z_1)^{\pm 1},\\
&\te{(Q4) }&& \phi^-_{i,q}(z_1)x^\pm_{j,q}(z_2)=x^\pm_{j,q}(z_2)\phi^-_{i,q}(z_1)
    \iota_{z_2,z_1}g_{ji,q}(q^{\mp \half rc}z_1/z_2)^{\mp 1},\\
&\te{(Q5)  }&&[x_{i,q}^+(z_1),x_{j,q}^-(z_2)]\\
&&&=\frac{\delta_{ij}}{q_i-q_i\inv}
    \Bigg(
        \phi_{i,q}^+(z_1 q^{-\half rc})\delta\left(
            \frac{z_2 q^{rc}}{z_1}
        \right)
        -
        \phi_{i,q}^-(z_1 q^{\half rc})\delta\left(
            \frac{z_2 q^{-rc} }{z_1}
        \right)
    \Bigg),\\
&\te{(Q6)  }&&F^\pm_{ij,q}(z_1,z_2)x^\pm_{i,q}(z_1)x^\pm_{j,q}(z_2)=
    G^\pm_{ij,q}(z_1,z_2)x^\pm_{j,q}(z_2)x^\pm_{i,q}(z_1),\\
&\te{(Q7)  }&&\sum_{\sigma\in S_{{m}_{ij}}}\sum_{k=0}^{{m}_{ij}}
    (-1)^k \binom{{m}_{ij}}{k}_{q_i}x_{i,q}^\pm(z_{\sigma(1)})\cdots x_{i,q}^\pm(z_{\sigma(k)}) x_{j,q}^\pm(w)\\
&&&\qquad\times
       x_{i,q}^\pm(z_{\sigma(k+1)})\cdots x_{i,q}^\pm(z_{\sigma({m}_{ij})})
\ =0,\quad
    \te{if}\ \ {a}_{ij}<0,
\end{align*}
where $q=\exp \hbar\in\C[[\hbar]]$, $q_i=q^{r_i}$, $r$ is the least common multiple of $\{r_i\}_{i\in I}$
and
\begin{align*} 
&F^\pm_{ij,q}(z_1,z_2)= z_1-q_i^{\pm a_{ij}}z_2, \quad
G^\pm_{ij,q}(z_1,z_2)= q_i^{\pm a_{ij}}z_1-z_2,\\
&g_{ij,q}(z)=\frac{G_{ij,q}^+(1,z)}{F_{ij,q}^+(1,z)}
\end{align*}
and the map $\iota_{z_1,z_2}$ is defined as in \cite[\S 3.1]{fhl}.
The classical limit $\U_\hbar(\hat\fg)/\hbar U_\hbar(\hat\fg)$ is isomorphic to the universal enveloping algebra of $\hat\fg$.

Let $V_{\hat\fg}(\ell,0)$ be the vertex algebra associated to $\hat\fg$ (see Definition \ref{de:aff-vas}).
Different from the construction of $V_L[[\hbar]]^\eta$ in \cite{JKLT-Quantum-lattice-va},
we are not able to find a suitable ``deforming triple'' for $V_{\hat\fg}(\ell,0)$, such that the corresponding $\hbar$-adic quantum vertex algebra can be associated to $\U_\hbar(\hat\fg)$.
Let $\U_\hbar^f(\hat\fg)$ be the unital associative algebra over $\C[[\hbar]]$ topologically generated by the elements \eqref{eq:tqagenerators}
subject to relations (Q1)-(Q4), (Q6) and
\begin{align*}
  &\te{(Q5.5)}&&(1-q^{r\ell} z_2/z_1)^{\delta_{ij}}(1-q^{-r\ell}z_2/z_1)^{\delta_{ij}}[x_{i,q}^+(z_1),x_{j,q}^-(z_2)]=0,\quad i,j\in I.
\end{align*}
It is easy to see that $\U_\hbar(\hat\fg)$ is a quotient algebra of $\U_\hbar^f(\hat\fg)$.
Using the theory of free vertex algebras developed by Roitman (\cite{R-free-conformal-free-va}),
we construct a vertex algebra $F(A,\ell)$ (see Definition \ref{de:aff-vas}).
Then by using the construction given in \cite{JKLT-Defom-va, JKLT-Quantum-lattice-va}, we get a family of $\hbar$-adic quantum vertex algebras $F_\tau(A,\ell)$ as formal deformations of $F(A,\ell)$ (see Theorem \ref{thm:S-tau}). The category of restricted $\U_\hbar^f(\hat\fg)$-modules is isomorphic to the category of $\phi$-coordinated $F_\tau(A,\ell)$-modules with suitable $\tau$.
For this suitable $\tau$, we construct the quotient $\hbar$-adic quantum vertex algebras $V_{\hat\fg,\hbar}(\ell,0)$ and $L_{\hat\fg,\hbar}(\ell,0)$ of $F_\tau(A,\ell)$.
Finally, based on the ``normal ordered'' type quantum Serre relations given in \cite{CJKT-qeala-II-twisted-qaffinization} and the integrable conditions developed in \cite{DM-int-rep-qaff}, we establish a natural connection between
restricted (resp. restricted integrable) $\U_\hbar(\hat\fg)$-modules of level $\ell$ and $\phi$-coordinated modules (resp. $\phi$-coordinated weighted modules) for $V_{\hat\fg,\hbar}(\ell,0)$ (resp. $L_{\hat\fg,\hbar}(\ell,0)$).

The paper is organized as follows. In Section \ref{sec:cla-va}, we recall the basics about vertex algebras and free vertex algebras. Define the vertex algebra $F(A,\ell)$ associated to a symmetrizable generalized Cartan matrix $A$ and a complex number $\ell$.
Then we realize $V_{\hat\fg}(\ell,0)$ and $L_{\hat\fg}(\ell,0)$ as quotient vertex algebras of $F(A,\ell)$.
In Section \ref{sec:qva}, we recall the notion of $\hbar$-adic quantum vertex algebras, and the theory of constructing $\hbar$-adic nonlocal vertex algebras introduced in \cite{Li-h-adic}.
In Section \ref{sec:F-tau}, we construct an $\hbar$-adic nonlocal vertex algebra $F_\tau(A,\ell)$. By using the deformation method introduced in \cite{JKLT-Defom-va, JKLT-Quantum-lattice-va} (see Section \ref{sec:deform-by-bialg}), we prove that $F_\tau(A,\ell)$ is an $\hbar$-adic quantum vertex algebra and $F_\tau(A,\ell)/\hbar F_\tau(A,\ell)\cong F(A,\ell)$ as vertex algebras.
In Section \ref{sec:V-tau-L-tau}, we construct two quotient $\hbar$-adic quantum vertex algebras $V_{\hat\fg,\hbar}(\ell,0)$ and $L_{\hat\fg,\hbar}(\ell,0)$ of $F_\tau(A,\ell)$.
In Section \ref{sec:phi-mods}, we recall the notion of $\phi$-coordinated modules of $\hbar$-adic nonlocal vertex algebras, and study the category of $\phi$-coordinated modules of $V_{\hat\fg,\hbar}(\ell,0)$.
In Section \ref{sec:qaff}, we establish the natural connections between certain $\U_\hbar(\hat\fg)$-modules and $\phi$-coordinated modules for
the two $\hbar$-adic quantum vertex algebras $V_{\hat\fg,\hbar}(\ell,0)$ and $L_{\hat\fg,\hbar}(\ell,0)$.

Throughout this paper, we denote the set of nonnegative integer and  positive integers by $\N$ and $\Z_+$, respectively.
For any ring $R$, we denote the set of invertible elements by $R^\times$.

\section{Vertex algebras and affine vertex algebras}\label{sec:cla-va}

In this section, we recall the notion of vertex algebras and define some vertex algebras associated to the symmetrizable GCM $A=(a_{ij})_{i,j\in I}$.

A \emph{vertex algebra} is a vector space $V$ equipped with a linear map
\begin{align}
  Y(\cdot,z):&V\to \E(V):=\Hom(V,V((z)));\quad
  v\mapsto Y(v,z)=\sum_{n\in\Z}v_nz^{-n-1}
\end{align}
and equipped with a distinguished vector $\vac$ of $V$,
called the \emph{vacuum vector}, such that
\begin{align}\label{eq:vacuum-property}
  &Y(\vac,z)v=v,\quad
  Y(v,z)\vac\in V[[z]],\quad \lim_{z\to 0}Y(v,z)\vac=v\quad\te{for }v\in V
\end{align}
and such that for $u,v,w\in V$, the following \emph{Jacobi identity} holds true
\begin{align}\label{eq:Jacobi}
z_0\inv\delta\(\frac{z_1-z_2}{z_0}\)&Y(u,z_1)Y(v,z_2)w
-z_0\inv\delta\(\frac{z_2-z_1}{-z_0}\)Y(v,z_2)Y(u,z_1)w
\nonumber\\&
=z_1\inv\delta\(\frac{z_2+z_0}{z_1}\)Y(Y(u,z_0)v,z_2)w.
\end{align}

A \emph{conformal Lie algebra} $(C,Y^+,T)$ (\cite{Kac-VA}), also known as a \emph{vertex Lie algebra} (\cite{DLM-vertex-Lie-vertex-Poisson-VA,Primc-VA-gen-by-Lie}),
is a vector space $C$ equipped with a linear operator $T$ and a linear map
 \begin{align}\label{Y-}
 Y^+(\cdot,z):\,\, C\to \mathrm{Hom}(C, z^{-1}C[z^{-1}]);\quad
  u\mapsto Y^+(u,z)=\sum_{n\ge 0}u_{n} z^{-n-1}
 \end{align}
 such that for any $u,v\in C$,
\begin{align*}
  &[T, Y^+(u,z)]=Y^+(Tu,z)=\frac{d}{d z}Y^+(u,z),\\
  &Y^+(u,z)v=\Sing_z\(e^{zT}Y^+(v,-z)u\),\\
  &[Y^+(u,z_1),Y^+(v,z_2)]=\Sing_{z_1}\Sing_{z_2}(Y^+(Y^+(u,z_1-z_2)v,z_2)),
\end{align*}
where
 $\Sing$ stands for the singular part.
Let $(C,Y^+,T)$, $(C_1,Y_1^+,T_1)$ be conformal Lie algebras.
A linear map $f:C\to C_1$ is called conformal Lie algebra homomorphism
if $f\big(Y^+(u,z)v\big)=Y^+_1(f(u),z)f(v)$ and
$f\big(Tu\big)=T_1f(u)$ for $u,v\in C$.

Let $C$ be an arbitrary conformal Lie algebra.
Recall the \emph{coefficient algebra} $\wh C$ (\cite{R-free-conformal-free-va}), also known as the \emph{local vertex Lie algebra} (\cite{DLM-vertex-Lie-vertex-Poisson-VA})
is the Lie algebra with the underlying vector space $\big(C\ot\C[t,t\inv]\big)/\te{Im}\big(T\ot 1+1\ot d/dt\big)$, and the Lie bracket
\begin{align}\label{eq:def-coeff-alg}
  [u(m),v(n)]=\sum_{k\ge0}\binom{m}{k}\big(u_kv\big)(m+n-k),
\end{align}
where $u,v\in C$, $m,n\in \Z$ and $u(m)$ is the image of $u\ot t^m$ in $\wh C$.
Set $\wh C^-=\Span_\C\set{u(m)}{u\in C,m<0}$ and $\wh C^+=\Span_\C\set{u(m)}{u\in C,m\ge0}$,
and define
\begin{align}\label{eq:def-V-C}
  V_C=\U(\wh C)\ot_{\U(\wh C^+)}\C,
\end{align}
where $\C$ is a trivial $\wh C^+$-module.
It is proved in \cite[Proposition 1.3]{R-free-conformal-free-va}
(see also \cite[Proposition 3.4]{DLM-vertex-Lie-vertex-Poisson-VA})
that $\wh C^\pm$ are Lie subalgebras of $\wh C$ and $\wh C=\wh C^-\oplus \wh C^+$.
Then $V_C\cong \U(\wh C^-)$ as vector spaces.
For convenience, we identify $u$ with $u(-1)\ot 1\in V_C$ for $u\in C$.
A $\wh C$-module $W$ is called a \emph{restricted} module if
\begin{align*}
  u(z)v=\sum_{m\in\Z}u(m)vz^{-m-1}\in W((z)),\quad\te{for }u\in C,\,\,v\in W.
\end{align*}
It is proved in \cite{DLM-vertex-Lie-vertex-Poisson-VA} that
\begin{thm}[{\cite[Theorem 4.8]{DLM-vertex-Lie-vertex-Poisson-VA}}]\label{thm:con-Lie=va}
There exists a vertex algebra structure on $V_C$ such that $\vac=1\ot 1\in V_C$ is the vacuum vector and
\begin{align*}
  Y(u,z)=\sum_{m\in\Z}u(m)z^{-m-1}\quad\te{for }u\in C.
\end{align*}
For any restricted $\wh C$-module $W$, there exists a $V_C$-module structure $Y_W$ on $W$ such that
\begin{align*}
  Y_W(u,z)=u(z)=\sum_{m\in\Z}u(m)z^{-m-1},\quad u\in C.
\end{align*}
On the other hand, for any $V_C$-module $W$, there exists a restricted $\wh C$-module structure on $W$ such that
\begin{align*}
  u(z).w=Y_W(u,z)w
\end{align*}
where $u\in C$ and $w\in W$.
\end{thm}

Let $\mathcal B$ be a set, and let $N:\mathcal B\times \mathcal B\to \N$ be a symmetric function called a \emph{locality function} (\cite{R-free-conformal-free-va}).
Denote by $\mathfrak{Conf}(\mathcal B,N)$ the category whose objects are conformal Lie algebras $C$ which contain $\mathcal B$ as a generating subset and
\begin{align*}
  a_nb=0\quad \te{for all } a,b\in\mathcal B,\,n\ge N(a,b).
\end{align*}
The following result is given in \cite[Proposition 3.1]{R-free-conformal-free-va}
\begin{prop}
There exists a conformal Lie algebra $C(\mathcal B,N)\in\mathfrak{Conf}(\mathcal B,N)$ such that for any $C\in\mathfrak{Conf}(\mathcal B,N)$, there is a unique conformal Lie algebra homomorphism $f$ from $C(\mathcal B,N)$ to $C$ such that
\begin{align*}
  f(a)=a\in \mathcal B\subset C\quad\te{for }a\in \mathcal B\subset C(\mathcal B,N).
\end{align*}
Moreover, the coefficient algebra $\wh C(\mathcal B,N)$ of $C(\mathcal B,N)$ is isomorphic to the Lie algebra generated by
\begin{align}
  \set{b(n)}{b\in\mathcal B,\,n\in\Z}
\end{align}
and subject to the relations
\begin{align}\label{eq:free-cla-def-rel}
  \sum_{i\ge 0}(-1)^i\binom{N(a,b)}{i}[a(n-i),b(m+i)]=0\quad\te{for } a,b\in\mathcal B,\,m,n\in\Z.
\end{align}
\end{prop}

The conformal Lie algebra $C(\mathcal B,N)$ is called the \emph{free conformal Lie algebra} in \cite{R-free-conformal-free-va}.
From Theorem \ref{thm:con-Lie=va}, we get a vertex algebra $V(\mathcal B,N)$ corresponding to $C(\mathcal B,N)$, which is called the \emph{free vertex algebra} in \cite{R-free-conformal-free-va}.

We note that the relation \eqref{eq:free-cla-def-rel} is equivalent to the following relation
\begin{align}
  (z_1-z_2)^{N(a,b)}[a(z_1),b(z_2)]=0\quad \te{for }a,b\in \mathcal B.
\end{align}
\begin{prop}\label{prop:universal-va}
Let $V$ be a vertex algebra, and let $f:\mathcal B\to V$ such that
\begin{align*}
  (z_1-z_2)^{N(a,b)}[Y(f(a),z_1),Y(f(b),z_2)]=0\quad \te{for }a,b\in \mathcal B.
\end{align*}
Then there exists a vertex algebra homomorphism $\wh f:V(\mathcal B,N)\to V$ such that
$\wh f|_{\mathcal B}=f$.
\end{prop}

We introduce the following Lie algebra.

\begin{de}\label{de:MRY}
Let $\hat\fg$ be the Lie algebra generated by $$\set{h_{i}(m),\,x_{i}^\pm(m)}{i\in I,\,m\in\Z}$$ and a central element $c$,
subject to the relations written in terms of generating functions in $z$:
\begin{align*}
  &h_i(z)=\sum_{m\in\Z}h_{i}(m)z^{-m-1},\quad
  x_i^\pm(z)=\sum_{m\in\Z}x_{i}^\pm(m) z^{-m-1},\quad i\in I.
\end{align*}
The relations are ($i,j\in I$)
\begin{align*}
  \te{(L1)}\quad& [h_i(z_1),h_j(z_2)]=r_ia_{ij}rc\pd{z_2}z_1\inv\delta\(\frac{z_2}{z_1}\),\\
  \te{(L2)}\quad& [h_i(z_1),x_j^\pm(z_2)]=\pm r_ia_{ij}x_j^\pm(z_2) z_1\inv\delta\(\frac{z_2}{z_1}\),\\
  \te{(L3)}\quad& [x_i^+(z_1),x_j^-(z_2)]=\frac{\delta_{ij}}{r_i}\(h_i(z_2)z_1\inv\delta\(\frac{z_2}{z_1}\)
  +rc\pd{z_2}z_1\inv\delta\(\frac{z_2}{z_1}\)\),\\
  \te{(L4)}\quad& (z_1-z_2)^{n_{ij}}[x_i^\pm(z_1),x_j^\pm(z_2)]=0,\\
  \te{(S)}\quad&[x_i^\pm(z_1),[x_i^\pm(z_2),\dots, [x_i^\pm(z_{m_{ij}}),x_j^\pm(z_0)]\cdots]]=0,\quad\te{if }a_{ij}<0,
\end{align*}
where $n_{ij}=1-\delta_{ij}$ for $i,j\in I$ and $m_{ij}=1-a_{ij}$ for $i,j\in I$ with $a_{ij}<0$.
\end{de}

\begin{rem}
{\em
Suppose that the GCM $A$ is of finite or affine type.
Recall that $\g=[\g(A),\g(A)]$ is the derived subalgebra of the Kac-Moody Lie algebra associated to the GCM $A$.
If $A$ is not of type $A_1^{(1)}$, then $\hat\fg$ is the universal central extension of $\g\ot\C[t,t\inv]$ (see \cite{Gar-loop-alg} when $A$ is of finite type and \cite{MRY} when $A$ is of affine type).
If $A$ is of type $A_1^{(1)}$, then $\hat\fg$ is the quotient algebra of the universal central extension of $\g\ot\C[t,t\inv]$ modulo the ideal generated by the following relations
\begin{align*}
  (z_1-z_2)[x_i^\pm(z_1),x_j^\pm(z_2)]=0,\quad \te{for }i\ne j.
\end{align*}
}
\end{rem}

Introduce a set $\mathcal B=\set{h_i,\,x_i^\pm}{i\in I}$, and define a function
$N:\mathcal B\times \mathcal B\to \N$ by
\begin{align*}
  N(h_i,h_j)=2,\quad N(h_i,x_j^\pm)=1,\quad N(x_i^\pm,x_j^\pm)=n_{ij},\quad N(x_i^+,x_j^-)=\delta_{ij}2.
\end{align*}

\begin{de}\label{de:aff-vas}
For $\ell\in\C$, we let $F(A,\ell)$ be the quotient vertex algebra of $V(\mathcal B,N)$ modulo the ideal generated by
\begin{align*}
  (h_i)_0(h_j),\quad (h_i)_1(h_j)-r_ia_{ij}r\ell\vac,\quad (h_i)_0(x_j^\pm)\mp r_ia_{ij}x_j^\pm\quad\te{for }i,j\in I.
\end{align*}
Furthermore, let $V_{\hat\fg}(\ell,0)$ be the quotient vertex algebra of $F(A,\ell)$ modulo the ideal generated by
\begin{align*}
  &(x_i^+)_0(x_j^-)-\frac{\delta_{ij}}{r_i}h_i,\quad (x_i^+)_1(x_j^-)-\frac{\delta_{ij}}{r_i}r\ell\vac\quad\te{for }i,j\in I,\\ &\((x_i^\pm)_0\)^{m_{ij}}(x_j^\pm)\quad \te{for }i,j\in I\,\te{with }a_{ij}<0.
\end{align*}
If $\ell\in\Z_+$, we define $L_{\hat\fg}(\ell,0)$ to be the quotient vertex algebra of $V_{\hat\fg}(\ell,0)$ modulo the ideal generated by
\begin{align*}
  \((x_i^\pm)_{-1}\)^{\ell r/r_i}x_i^\pm\quad \te{for }i\in I.
\end{align*}
\end{de}

\begin{rem}\label{rem:aff-Lie-alg}{\em
Set $\hat\fg_+=\fg\ot\C[t]\oplus\C c$. For $\ell\in\C$, we let $\C_\ell:=\C$ be a $\hat\fg_+$-module,
with $\fg\ot\C[t].\C_\ell=0$ and $c=\ell$.
Then there are $\hat\fg$-module structures on both $V_{\hat\fg}(\ell,0)$ and $L_{\hat\fg}(\ell,0)$, such that
\begin{align*}
  h_i(z)=Y(h_i,z),\quad x_i^\pm(z)=Y(x_i^\pm,z),\quad i\in I.
\end{align*}
In addition, as $\hat\fg$-modules, we have that
\begin{align*}
V_{\hat\fg}(\ell,0)\cong \U(\hat\fg)\ot_{\U(\hat\fg_+)}\C_\ell
\end{align*}
Suppose that the GCM $A$ is of finite type. Then $\fg$ is a finite dimensional simple Lie algebra.
And $L_{\hat\fg}(\ell,0)$ is the unique simple quotient $\hat\fg$-module of $V_{\hat\fg}(\ell,0)$, when
$\ell\in \Z_+$.}
\end{rem}

\section{Quantum vertex algebras}\label{sec:qva}

A $\C[[\hbar]]$-module $V$
is said to be \emph{torsion-free} if $\hbar v\ne 0$ for every $0\ne v\in V$,
and said to be \emph{separated} if $\cap_{n\ge1}\hbar^n V=0$.
For a $\C[[\hbar]]$-module $V$, using subsets $v+\hbar^nV$ for $v\in V$, $n\ge 1$
as the basis of open subsets one obtains a topology on $V$, which is called the
\emph{$\hbar$-adic topology}.
A $\C[[\hbar]]$-module $V$ is said to be \emph{$\hbar$-adically complete}
if every Cauchy sequence in $V$ with respect to this $\hbar$-adic topology has a limit in $V$.
A $\C[[\hbar]]$-module $V$ is \emph{topologically free} if $V=V_0[[\hbar]]$
for some vector space $V_0$ over $\C$.
It is known that a $\C[[\hbar]]$-module is topologically free if and only if it is torsion-free, separated, and $\hbar$-adically complete (\cite{Kassel-topologically-free}, see \cite{Li-h-adic}).
For another topologically free $\C[[\hbar]]$-module $U=U_0[[\hbar]]$, we recall the complete tensor
\begin{align*}
  U\wh\ot V=(U_0\ot V_0)[[\hbar]].
\end{align*}

We view a vector space as a $\C[[\hbar]]$-module by letting $\hbar=0$.
Fix a $\C[[\hbar]]$-module $W$. For $k\in\Z_+$,
and some formal variables $z_1,\dots,z_k$, we define
\begin{align}
  \E^{(k)}(W;z_1,\dots,z_k)=\Hom_{\C[[\hbar]]}(W,W((z_1,\dots,z_k))).
\end{align}
Recall from \cite[Definition 5.3]{li-g1} that an ordered sequence $(a_1(z),\dots,a_k(z))$ in $\E^{(1)}(W)$ is said to be \emph{compatible}
if there exists an $m\in\Z_+$, such that
\begin{align*}
  \(\prod_{1\le i<j\le k}(z_i-z_j)^m\)a_1(z_1)\cdots a_k(z_k)\in\E^{(k)}(W;z_1,\dots,z_k).
\end{align*}

Now, we assume $W=W_0[[\hbar]]$ for some vector space $W_0$. Then $W$ is topologically free.
Define
\begin{align}\label{eq:E-hbar}
  \E_\hbar^{(k)}(W;z_1,\dots,z_k)=\Hom_{\C[[\hbar]]}(W,W_0((z_1,\dots,z_k))[[\hbar]]).
\end{align}
Note that $\E_\hbar^{(k)}(W;z_1,\dots,z_k)=\E^{(k)}(W;z_1,\dots,z_k)[[\hbar]]$ is topologically free.
For convenience, we will also write $\E_\hbar^{(k)}(W)=\E_\hbar^{(k)}(W;z_1,\dots,z_k)$ and write $\E_\hbar(W)=\E_\hbar^{(1)}(W)$.
For $n,k\in\Z_+$, the quotient map from $W$ to $W/\hbar^nW$ induces the following $\C[[\hbar]]$-module map
\begin{align*}
  \wt\pi_n^{(k)}:(\End_{\C[[\hbar]]} (W))[[z_1^{\pm 1},\dots,z_k^{\pm 1}]]\to (\End_{\C[[\hbar]]}(W/\hbar^nW))[[z_1^{\pm 1},\dots,z_k^{\pm 1}]].
\end{align*}
For $A(z_1,z_2),B(z_1,z_2)\in\Hom_{\C[[\hbar]]}(W,W_0((z_1))((z_2))[[\hbar]])$, we write $A(z_1,z_2)\sim B(z_2,z_1)$ if for each $n\in\Z_+$ there is $k\in\N$ such that
\begin{align*}
  (z_1-z_2)^k\wt\pi_n^{(2)}(A(z_1,z_2))= (z_1-z_2)^k\wt\pi_n^{(2)}(B(z_2,z_1)).
\end{align*}
Let $Z(z_1,z_2):\E_\hbar(W)\wh \ot \E_\hbar(W)\wh\ot\C((z))[[\hbar]]\to\End_{\C[[\hbar]]}(W)[[z_1^{\pm 1},z_2^{\pm 1}]]$ be defined by
\begin{align*}
  Z(z_1,z_2)\(a(z)\ot b(z)\ot f(z)\)=\iota_{z_1,z_2}f(z_1-z_2)a(z_1)b(z_2).
\end{align*}
A subset $U$ of $\E_\hbar(W)$ is said to be \emph{$\hbar$-adically $S$-local} if for any $a(z),b(z)\in U$, there exists $A(z)\in \(\C U\ot \C U\ot \C((z))\)[[\hbar]]$ such that
\begin{align*}
  a(z_1)b(z_2)\sim Z(z_2,z_1)\(A(z)\),
\end{align*}
where $\C U$ denotes the subspace spanned by $U$.

Recall from \cite[Remark 4.7]{Li-h-adic} that $\wt \pi_n^{(k)}$ induces a $\C[[\hbar]]$-module map
$\pi_n^{(k)}:\E_\hbar^{(k)}(W)\to \E^{(k)}(W/\hbar^n W)$ with kernel $\hbar^n\E_\hbar^{(k)}(W)$.
And $\E_\hbar^{(k)}(W)$ is isomorphic to the inverse limit of the following inverse system
\begin{align*}
  \xymatrix{
    0 & \E^{(k)}(W/\hbar W)\ar[l]& \E^{(k)}(W/\hbar^2 W)\ar[l]& \E^{(k)}(W/\hbar^3 W)\ar[l]&\cdots\ar[l].
  }
\end{align*}
If $k=1$, we will also write $\pi_n=\pi_n^{(1)}$.
Then an ordered sequence $(a_1(z)\dots,a_r(z))$ in $\E_\hbar(W)$ is said to be \emph{$\hbar$-adically compatible} if for every $n\in\Z_+$,
the sequence $(\pi_n(a_1(z)),\dots,\pi_n(a_r(z)))$ in $\E(W/\hbar^nW)$ is compatible.
A subset $U$ of $\E_\hbar(W)$ is said to be \emph{$\hbar$-adically compatible} if every finite sequence in $U$ is $\hbar$-adically compatible.

Let $(a(z),b(z))$ in $\E_\hbar(W)$ be $\hbar$-adically compatible.
That is, for any $n\in\Z_+$, we have that
\begin{align*}
  (z_1-z_2)^{k_n}\pi_n\(a(z_1)\)\pi_n\(b(z_2)\)\in\E^{(2)}(W/\hbar^nW)\quad\te{for some }k_n\in\Z_+.
\end{align*}
We recall the following vertex operator introduced in \cite[Definition 4.11]{Li-h-adic}:
\begin{align}\label{eq:def-Y-E}
  &Y_\E\(a(z),z_0\)b(x)=\sum_{n\in\Z}a(z)_nb(z)z_0^{-n-1}\\
  =&\varprojlim_{n>0}z_0^{-k_n}\((z_1-z)^{k_n}\pi_n(a(z_1))\pi_n(b(z))\)|_{z_1=z+z_0}.\nonumber
\end{align}
An \emph{$\hbar$-adic nonlocal vertex algebra} \cite[Definition 2.9]{Li-h-adic}
is a topologically free $\C[[\hbar]]$-module $V$ equipped with a $\C[[\hbar]]$-module map
$Y(\cdot,z):V\to \E_\hbar(V)$ and a distinguished vacuum vector $\vac$ such that the vacuum property \eqref{eq:vacuum-property} holds,
and for $u,v\in V$, $(Y(u,z),Y(v,z))$ is an $\hbar$-adic compatible pair with
\begin{align}
  &Y_\E\(Y(u,z),z_0\)Y(v,z)=Y\(Y(u,z_0)v,z\)\label{eq:weak-asso}
\end{align}
Denote by $\partial$ the canonical derivation on $V$:
\begin{align}\label{eq:canonical-der}
  u\mapsto\partial u=\lim_{z\to 0}\frac{d}{dz}Y(u,z)\vac.
\end{align}

An \emph{$\hbar$-adic weak quantum vertex algebra} (\cite[Definition 2.9]{Li-h-adic}) is an $\hbar$-adic nonlocal vertex algebra $V$, such that $\set{Y(u,z)}{u\in V}$ is $\hbar$-adically $S$-local.
And an \emph{$\hbar$-adic quantum vertex algebra} (\cite[Definition 2.20]{Li-h-adic}) is an $\hbar$-adic weak quantum vertex algebra $V$ equipped with a $\C[[\hbar]]$-module map (called a \emph{quantum Yang-Baxter operator})
\begin{align}
  S(z):V\wh\ot V\to V\wh\ot V\wh\ot \C((z))[[\hbar]],
\end{align}
which satisfies the \emph{shift condition}:
\begin{align}\label{eq:qyb-shift}
  [\partial\ot 1,S(z)]=-\frac{d}{dz}S(z),
\end{align}
the \emph{quantum Yang-Baxter equation}:
\begin{align}\label{eq:qyb}
  S^{12}(z_1)S^{13}(z_1+z_2)S^{23}(z_2)=S^{23}(z_2)S^{13}(z_1+z_2)S^{12}(z_1),
\end{align}
and the \emph{unitarity condition}:
\begin{align}\label{eq:qyb-unitary}
  S^{21}(z)S(-z)=1,
\end{align}
satisfying the following conditions:

  (1) The \emph{vacuum property}:
  \begin{align}\label{eq:qyb-vac-id}
    S(z)(\vac\ot v)=\vac\ot v,\quad \te{for }v\in V.
  \end{align}

 (2) The \emph{$S$-locality}:
  For any $u,v\in V$, one has
  \begin{align}\label{eq:qyb-locality}
  Y(u,z_1)Y(v,z_2)\sim Y(z_2)(1\ot Y(z_1))S(z_2-z_1)(v\ot u).
  \end{align}

  (3) The \emph{hexagon identity}:
  \begin{align}\label{eq:qyb-hex-id}
    S(z_1)(Y(z_2)\ot 1)=(Y(z_2)\ot 1)S^{23}(z_1)S^{13}(z_1+z_2).
  \end{align}

We list some technical results that will be used later on.

\begin{lem}
Let $(V,S(z))$ be an $\hbar$-adic quantum vertex algebra.
Then
\begin{align}
  &S(z)(v\ot\vac)=v\ot\vac\quad \te{for }v\in V,\label{eq:qyb-vac-id-alt}\\
  &[1\ot\partial,S(z)]=\frac{d}{dz}S(z),\label{eq:qyb-shift-alt}\\
  &S(z_1)(1\ot Y(z_2))=(1\ot Y(z_2))S^{12}(z_1-z_2)S^{13}(z_1)\label{eq:qyb-hex-id-alt},\\
  &S(z)f(\partial\ot 1)=f\(\partial\ot 1+\pd{z}\)S(z)\quad \te{for }f(z)\in\C[z][[\hbar]],\label{eq:qyb-shift-total1}\\
  &S(z)f(1\ot \partial)=f\(1\ot \partial-\pd{z}\)S(z)\quad \te{for }f(z)\in\C[z][[\hbar]].\label{eq:qyb-shift-total2}
\end{align}
\end{lem}

\begin{lem}\label{lem:S-special-tech-gen2}
Let $(V,S(z))$ be an $\hbar$-adic quantum vertex algebra,
and let $u,v\in V$, $f(z)\in\C((z))[[\hbar]]$.
\begin{itemize}
  \item[(1)] If $S(z)(v\ot u)=v\ot u+\vac\ot\vac\ot f(z)$,
then
\begin{align*}
  &S(z)((v_{-1})^n\vac\ot u)\\
  =&(v_{-1})^n\vac\ot u+n(v_{-1})^{n-1}\vac\ot \vac\ot f(z),\\
  &S(z)(v\ot (u_{-1})^n\vac)\\
  =&v\ot (u_{-1})^n\vac+n\vac\ot (u_{-1})^{n-1}\vac\ot f(z),\\
  &S(z)\(\exp ( v_{-1})\vac\ot u\)\\
  =&\exp ( v_{-1})\vac\ot u+\exp (v_{-1})\vac\ot \vac\ot  f(z),\quad \te{if }v\in \hbar V,\\
  &S(z)\(v\ot \exp( u_{-1})\vac\)\\
  =&v\ot \exp( u_{-1})\vac+\vac\ot \exp( u_{-1})\vac\ot  f(z),\quad \te{if }u\in \hbar V.
\end{align*}

\item[(2)] If $S(z)(v\ot u)=v\ot u+\vac\ot u\ot f(z)$, then
\begin{align*}
  &S(z)((v_{-1})^n\vac\ot u)=\sum_{i=0}^n\binom{n}{i}(v_{-1})^i\vac\ot u\ot f(z)^{n-i},\\
  &S(z)\(\exp ( v_{-1})\vac\ot u\)=\exp ( v_{-1})\vac\ot u\ot\exp\( f(z)\),\\
  &\quad \te{if }v\in\hbar V, f(z)\in\hbar\C((z))[[\hbar]].
\end{align*}

\item[(3)] If $S(z)(v\ot u)=v\ot u+v\ot \vac\ot f(z)$, then
\begin{align*}
  &S(z)(v\ot (u_{-1})^n\vac)=\sum_{i=0}^n\binom{n}{i}v\ot (u_{-1})^i\vac\ot f(z)^{n-i},\\
  &S(z)\(v\ot \exp( u_{-1})\vac\)=v\ot \exp\( u_{-1}\)\vac\ot \exp( f(z)),\\
  &\quad \te{if }u\in\hbar V, f(z)\in\hbar\C((z))[[\hbar]].
\end{align*}
\end{itemize}
\end{lem}

For a topologically free $\C[[\hbar]]$-module $V$ and a submodule $M$, we recall the notation given in \cite[Definition 3.4]{Li-h-adic}:
\begin{align*}
  [M]=\set{v\in V}{\hbar^n V\in M\,\,\te{for some }n\in\N}.
\end{align*}
It is straightforward to verify that
\begin{lem}\label{lem:S-quotient-alg}
Let $(V,S(z))$ be an $\hbar$-adic quantum vertex algebra, and let $U\subset V$ be a generating subset of $V$.
Let $M\subset V$ be a closed ideal of $V$ such that $[M]=M$ and
\begin{align*}
  S(z)(M\ot U),\,S(z)(U\ot M)\subset M\wh\ot V\wh\ot \C((z))[[\hbar]]+V\wh\ot M\wh\ot \C((z))[[\hbar]].
\end{align*}
Then $V/M$ is an $\hbar$-adic quantum vertex algebra with quantum Yang-Baxter operator induced from $S(z)$.
\end{lem}

In the rest of this section, we recall the construction of $\hbar$-adic nonlocal vertex algebras and their modules introduced in \cite{Li-h-adic}.
Here, a \emph{module} (\cite[Definition 2.33]{Li-h-adic}) for an $\hbar$-adic nonlocal vertex algebra $V$ is a topologically free $\C[[\hbar]]$-module $W$, equipped with a $\C[[\hbar]]$-module map $Y_W(\cdot,z):V\to \E_\hbar(W)$,
satisfying the conditions that $Y_W(\vac,z)=1_W$ and that for $u,v\in V$,
$(Y_W(u,z),Y_W(v,z))$ is $\hbar$-adically compatible and
\begin{align*}
  Y_\E\(Y_W(u,z),z_0\)Y_W(v,z)=Y_W\(Y(u,z_0)v,z\).
\end{align*}

The following result is given in \cite{Li-h-adic}:

\begin{thm}\label{thm:construction}
Let $W$ be a topologically free $\C[[\hbar]]$-module, and let $U\subset \E_\hbar(W)$ be an $\hbar$-adically compatible subset. Then there is a minimal $\hbar$-adically $S$-local subset $\<U\>\subset \E_\hbar(W)$ containing $1_W$ and $U$, such that
\begin{itemize}
  \item[(1)] $\<U\>$ is topologically free and $[\<U\>]=\<U\>$.

  \item[(2)] $\<U\>$ is $Y_\E$ closed in the sense that for any $a(z),b(z)\in U$ and any $n\in\Z$, we have $a(z)_nb(z)\in \<U\>$.
\end{itemize}
Then $(\<U\>,Y_\E,1_W)$ carries the structure of an $\hbar$-adic nonlocal vertex algebra and $W$ is a faithful $\<U\>$-module with the module map $Y_W$ defined by $Y_W(a(z),z_0)=a(z_0)$ for $a(z)\in\<U\>$.
\end{thm}

\begin{rem}{\em
If $U$ is $\hbar$-adically $S$-local, then $\<U\>$ is an $\hbar$-adic weak quantum vertex algebra.
}\end{rem}

We recall the following result from \cite[Theorem 4.24]{Li-h-adic} for later use.

\begin{thm}\label{thm:gen-weak-qva}
Let $V$ be a topologically free $\C[[\hbar]]$-module, $U\subset V$, $\vac\in V$, and $Y^0$ a map
\begin{align*}
  Y^0:U\to \E_\hbar(V);\quad u\mapsto Y^0(u,z)=u(z)=\sum_{n\in\Z}u_nz^{-n-1}.
\end{align*}
Assume that all the following conditions hold:
\begin{align*}
  Y^0(u,z)\vac\in V[[z]]\quad\te{and}\quad \lim_{z\to 0}Y^0(u,z)\vac=u\quad \te{for }u\in U,
\end{align*}
$U(z)=\set{u(z)}{u\in U}$ is $\hbar$-adically $S$-local, and $V$ is $\hbar$-adically spanned by vectors
\begin{align}
  u_{m_1}^{(1)}\cdots u_{m_r}^{(r)}\vac
\end{align}
for $r\in\N$, $u^{(i)}\in U$, $m_i\in\Z$.
In addition we assume that there exists a $\C[[\hbar]]$-module map $\psi$ from $V$ to $\<U(z)\>\subset \E_\hbar(V)$ such that $\psi(\vac)=1_V$ and
\begin{align}
  \psi\(Y(u,z_0)v\)=Y_\E\(u(z),z_0\)\psi(v)\quad \te{for }u\in U,\,v\in V.
\end{align}
Then the map $Y^0$ extends uniquely to a $\C[[\hbar]]$-module map $Y$ from $V$ to $\E_\hbar(V)$ such that $(V,Y,\vac)$ carries the structure of an $\hbar$-adic weak quantum vertex algebra.
\end{thm}

\section{Deformation by vertex bialgebras}\label{sec:deform-by-bialg}

In this section, we recall the deformation of $\hbar$-adic nonlocal vertex algebras by using vertex bialgebras introduced in \cite{JKLT-Defom-va, JKLT-Quantum-lattice-va}.

We start by recalling the notion of vertex bialgebras and smash products given in \cite{Li-smash} (see also \cite{JKLT-Quantum-lattice-va}).
An \emph{$\hbar$-adic (nonlocal) vertex bialgebra} is an $\hbar$-adic (nonlocal) vertex algebra $V$ equipped with a classical coalgebra structure $(\Delta,\varepsilon)$ such that (the coproduct) $\Delta:V\to V\wh\ot V$ and (the counit) $\varepsilon:V\to\C[[\hbar]]$ are homomorphisms of $\hbar$-adic nonlocal vertex algebras.

\begin{rem}\label{rem:bialg-der}
{\em
Let $(H,\Delta,\varepsilon)$ be a bialgebra over $\C[[\hbar]]$ equipped with a derivation $\partial$.
Suppose that $H$ is topologically free.
Then $H$ is an $\hbar$-adic nonlocal vertex bialgebra with vacuum $1$ and vertex operator defined by
\begin{align*}
  Y(a,z)b=\(e^{z\partial}a\)b\quad\te{for } a,b\in H.
\end{align*}
We denote this $\hbar$-adic nonlocal vertex bialgebra by $(H,\partial,\Delta,\varepsilon)$.
}
\end{rem}

Let $(H,\Delta,\varepsilon)$ be an $\hbar$-adic nonlocal vertex bialgebra.
A (left) \emph{$H$-module (nonlocal) vertex algebra} (\cite{Li-smash}) is an $\hbar$-adic nonlocal vertex algebra $V$ equipped with a module structure $\tau$ on $V$ for $H$ viewed as an $\hbar$-adic nonlocal vertex algebra such that
\begin{align}
  &\tau(h,z)v\in V\wh\ot \C((z))[[\hbar]],\qquad
  \tau(h,z)\vac_V=\varepsilon(h)\vac_V,\label{eq:mod-va-for-vertex-bialg1-2}\\
  &\tau(h,z_1)Y(u,z_2)v=\sum Y(\tau(h_{(1)},z_1-z_2)u,z_2)\tau(h_{(2)},z_1)v
  \label{eq:mod-va-for-vertex-bialg3}
\end{align}
for $h\in H$, $u,v\in V$, where $\vac_V$ denotes the vacuum vector of $V$
and $\Delta(h)=\sum h_{(1)}\ot h_{(2)}$ is the coproduct in the Sweedler notation.

A \emph{(right) $H$-comodule nonlocal vertex algebra} (\cite{JKLT-Defom-va}) is a nonlocal vertex algebra $V$ equipped with a homomorphism
$\rho:V\to V\wh\ot H$ of $\hbar$-adic nonlocal vertex algebras such that
\begin{align}
  (\rho\ot 1)\rho=(1\ot \Delta)\rho,\quad (1\ot \varepsilon)\rho=\te{Id}_V.
\end{align}
$\rho$ is \emph{compatible} with  an $H$-module nonlocal vertex algebra structure $\tau$ (see \cite{JKLT-Defom-va}) if
\begin{align}
  \rho(\tau(h,z)v)=(\tau(h,z)\ot 1)\rho(v)\quad \te{for }h\in H,\,v\in V.
\end{align}
We introduce the following notion.

\begin{de}
Let $V$ be an $\hbar$-adic nonlocal vertex algebra.
A \emph{deforming triple} for $V$ is a triple $(H,\rho,\tau)$,
where $H$ is a cocommutative $\hbar$-adic nonlocal vertex bialgebra,
$(V,\rho)$ is a right $H$-comodule nonlocal vertex algebra and $(V,\tau)$ is an $H$-module nonlocal vertex algebra, such that $\rho$ and $\tau$ are compatible.
\end{de}

The following result is given in \cite{JKLT-Defom-va}.

\begin{thm}\label{thm:deform-va}
Let $V$ be an $\hbar$-adic nonlocal vertex algebra, and let $(H,\rho,\tau)$ be a deforming triple.
Set
\begin{align}
\mathfrak D_\tau^\rho (Y)(a,z)=\sum Y(a_{(1)},z)\tau(a_{(2)},z)\quad\te{for }a\in V,
\end{align}
where $\rho(a)=\sum a_{(1)}\ot a_{(2)}\in V\otimes H$.
Then $(V,\mathfrak D_\tau^\rho (Y),\vac)$ carries the structure of an $\hbar$-adic nonlocal vertex algebra.
Denote this $\hbar$-adic nonlocal vertex algebra by $\mathfrak D_\tau^\rho (V)$.
Moreover, $(\mathfrak D_\tau^\rho(V),\rho)$ is also a right $H$-comodule nonlocal vertex algebra.
\end{thm}

\begin{rem}\label{rem:trivial-deform}
{\em
Let $H$ be a cocommutative $\hbar$-adic nonlocal vertex bialgebra, and let $(V,\rho)$ be an $H$-comodule nonlocal vertex algebra.
We view the counit $\varepsilon$ as an $H$-module nonlocal vertex algebra structure on $V$ as follows
\begin{align}
  \varepsilon(h,z)v=\varepsilon(h)v\quad \te{for }h\in H,\,v\in V.
\end{align}
Then it is easy to verify that $\varepsilon$ and $\rho$ are compatible,
and $\mathfrak D_\varepsilon^\rho(V)=V$.
}
\end{rem}

\begin{rem}\label{rem:deformed-hom}
{\em
Let $V_1,V_2$ be two $\hbar$-adic nonlocal vertex algebras. Suppose that $(H,\rho_1,\tau_1)$ and $(H,\rho_2,\tau_2)$ are deforming triples of $V_1$ and $V_2$ respectively.
Suppose that $f:V_1\to V_2$ is an $\hbar$-adic nonlocal vertex algebra homomorphism such that
\begin{align*}
  \rho_2\circ f=(f\ot 1)\circ\rho_1,\quad\te{and}\quad f(\tau_1(h,z)v)=\tau_2(h,z)f(v)\quad\te{for } h\in H,\,v\in V_1.
\end{align*}
Then $f$ is an $\hbar$-adic nonlocal vertex algebra homomorphism from $\mathfrak D_{\tau_1}^{\rho_1}(V_1)$
to $\mathfrak D_{\tau_2}^{\rho_2}(V_2)$.
}
\end{rem}

Now, we fix a cocommutative $\hbar$-adic nonlocal vertex bialgebra $H$, and an $H$-comodule nonlocal vertex algebra $(V,\rho)$.
Note that $$\Hom\(H,\Hom(V,V\wh\ot\C((x))[[\hbar]])\)$$ is a unital associative algebra,
where the multiplication is defined by
\begin{align*}
  (f\ast g)(h,z)=\sum f(h_{(1)},z)g(h_{(2)},z)
\end{align*}
for $f,g\in\Hom\(H,\Hom(V,V\wh\ot\C((z))[[\hbar]])\)$,
and the unit $\varepsilon$.
For $f,g\in\Hom\(H,\Hom(V,V\wh\ot\C((z))[[\hbar]])\)$, we say $f$ and $g$ commute if
\begin{align*}
  [f(h,z_1),g(k,z_2)]=0\quad \te{for }h,k\in H.
\end{align*}
Let $\mathfrak L_H^\rho(V)$ be the set of $H$-module nonlocal vertex algebra structures on $V$ that are compatible with $\rho$.
The following is an immediate $\hbar$-adic analogue of \cite[Proposition 3.3 and Proposition 3.4]{JKLT-Defom-va}.

\begin{prop}\label{prop:L-H-rho-V-compostition}
Let $\tau$ and $\tau'$ are commuting elements in $\mathfrak L_H^\rho(V)$.
Then $\tau\ast\tau'\in\mathfrak L_H^\rho(V)$
and $\tau\ast\tau'=\tau'\ast\tau$.
Moreover, $\tau\in \mathfrak L_H^\rho\(\mathfrak D_{\tau'}^\rho(V)\)$ and
\begin{align}
  \mathfrak D_\tau^\rho\(\mathfrak D_{\tau'}^\rho(V)\)=\mathfrak D_{\tau\ast\tau'}^\rho(V).
\end{align}
\end{prop}

Recall that an element $\tau\in\mathfrak L_H^\rho(V)$ is said to be \emph{invertible} if there exists $\tau\inv \in \mathfrak L_H^\rho(V)$, such that $\tau$ and $\tau\inv$ are commute and $\tau\ast\tau\inv=\varepsilon$.
We have the immediate $\hbar$-adic analogue of \cite[Theorem 3.6]{JKLT-Defom-va}.

\begin{thm}\label{thm:qva}
Let $V_0$ be a vertex algebra, and let $V=V_0[[\hbar]]$ be the corresponding $\hbar$-adic vertex algebra.
Suppose that $(H,\rho,\tau)$ is a deforming triple of $V$, such that $\tau$ is invertible in $\mathfrak L_H^\rho(V)$.
Then $\mathfrak D_\tau^\rho(V)$ is an $\hbar$-adic quantum vertex algebra with quantum Yang-Baxter operator $S(z)$ defined by
\begin{align*}
  S(z)(v\ot u)=\sum \tau(u_{(2)},-z)v_{(1)}\ot \tau\inv(v_{(2)},z)u_{(1)}\quad\te{for } u,v\in V,
\end{align*}
where $\rho(u)=\sum u_{(1)}\ot u_{(2)}$ and $\rho(v)=\sum v_{(1)}\ot v_{(2)}$.
\end{thm}

\section{Construction of $F_\tau(A,\ell)$}\label{sec:F-tau}

In this section, we fix a GCM $A=(a_{ij})_{i,j\in I}$ and a complex number $\ell$.
Let $\mathfrak T$ be the set of tuples
\begin{align*}
\tau=\big(\tau_{ij}(z),\tau_{ij}^{1,\pm}(z),\tau_{ij}^{2,\pm}(z),\tau_{ij}^{\epsilon_1,\epsilon_2}(z)\big)_{i,j\in I}^{\epsilon_1,\epsilon_2=\pm},
\end{align*}
where $\tau_{ij}(z),\,\tau_{ij}^{1,\pm}(z),\,\tau_{ij}^{2,\pm}(z),\,\tau_{ij}^{\epsilon_1,\epsilon_2}(z)\in \C((z))[[\hbar]]$, such that
\begin{align*}
  &\lim_{\hbar\to 0}\tau_{ij}(z)=\lim_{\hbar\to 0}\tau_{ji}(-z),\quad \lim_{\hbar\to 0}\tau_{ij}^{1,\pm}(z)=-\lim_{\hbar\to 0}\tau_{ji}^{2,\pm}(-z),\\
  &\lim_{\hbar\to 0}\tau_{ij}^{\epsilon_1,\epsilon_2}(z)=\lim_{\hbar\to 0}\tau_{ji}^{\epsilon_2,\epsilon_1}(-z)\in\C[[z]]^\times.
\end{align*}

\begin{de}\label{de:cat-M-tau}
Let $\tau\in\mathfrak T$. Define $\mathcal M_\tau$ to be the category consisting of topologically free $\C[[\hbar]]$-modules $W$, equipped with fields $h_{i,\hbar}(z),x_{i,\hbar}^\pm(z)\in\E_\hbar(W)$ ($i\in I$) satisfying the following conditions
\begin{align}
  &[h_{i,\hbar}(z_1),h_{j,\hbar}(z_2)]\label{eq:fqva-rel1}\\
  &\quad=r_ia_{ij} r\ell\pd{z_2}z_1\inv\delta\(\frac{z_2}{z_1}\)+\tau_{ij}(z_1-z_2)-\tau_{ji}(z_2-z_1),\nonumber\\
  &[h_{i,\hbar}(z_1),x_{j,\hbar}^\pm(z_2)]\label{eq:fqva-rel2}\\
  &\quad=\pm x_{j,\hbar}^\pm(z_2)\(r_ia_{ij} z_1\inv\delta\(\frac{z_2}{z_1}\)+ \tau_{ij}^{1,\pm}(z_1-z_2)+\tau_{ji}^{2,\pm}(z_2-z_1)\),\nonumber\\
  &\tau_{ij}^{\pm,\pm}(z_1-z_2)(z_1-z_2)^{n_{ij}}x_{i,\hbar}^\pm(z_1)x_{j,\hbar}^\pm(z_2)\label{eq:fqva-rel3}\\
  &\quad=\tau_{ji}^{\pm,\pm}(z_2-z_1)(z_1-z_2)^{n_{ij}}x_{j,\hbar}^\pm(z_2)x_{i,\hbar}^\pm(z_1),\nonumber\\
  &\tau_{ij}^{\pm,\mp}(z_1-z_2)(z_1-z_2)^{2\delta_{ij}}x_{i,\hbar}^\pm(z_1)x_{j,\hbar}^\mp(z_2)\label{eq:fqva-rel4}\\
  &\quad=\tau_{ji}^{\mp,\pm}(z_2-z_1)(z_1-z_2)^{2\delta_{ij}}x_{j,\hbar}^\mp(z_2)x_{i,\hbar}^\pm(z_1).\nonumber
\end{align}
\end{de}
Let $\mathcal F_\tau$ be the forgetful functor from $\mathcal M_\tau$ to the category of topologically free $\C[[\hbar]]$-modules.
Define $\End_{\C[[\hbar]]}(\mathcal F_\tau)$ to be the algebra of endomorphisms of the functor $\mathcal F_\tau$.
For each $W\in\mathcal M_{\tau}$, $\End_{\C[[\hbar]]}(W)$ is a topological algebra over $\C[[\hbar]]$ such that
\begin{align*}
  \set{(K:\hbar^nW)}{K\subset W,\, |K|<\infty,\,n\in\Z_+}
\end{align*}
forms a local basis at $0$, where
\begin{align*}
  (K:\hbar^nW)=\set{\varphi\in\End_{\C[[\hbar]]}(W)}{\varphi(K)\subset\hbar^nW}.
\end{align*}
Equip $\End_{\C[[\hbar]]}(\mathcal F_{\tau})$ with the coarsest topology such that for any $W\in\mathcal M_{\tau}$, the canonical $\C[[\hbar]]$-algebra epimorphism from $\mathcal A$ to $\End_{\C[[\hbar]]}(W)$ is continuous.
It is easy to verify that both $\End_{\C[[\hbar]]}(W)$ and $\End_{\C[[\hbar]]}(\mathcal F_{ \tau})$ are topologically free.

For each $i\in I$, we define endomorphisms $h_{i,\hbar}(n)$ and $x_{i,\hbar}^\pm(n)$ ($n\in\Z$) of $\mathcal F_{\tau}$ as follows:
\begin{align*}
  \sum_{n\in\Z}h_{i,\hbar}(n).vz^{-n-1}=h_{i,\hbar}(z)v,\quad \sum_{n\in\Z}x_{i,\hbar}^\pm(n).vz^{-n-1}=x_{i,\hbar}^\pm(z)v.
\end{align*}
where $v\in W$ and $W\in\mathcal M_\tau$.
We denote by $\mathcal A$ the closed subalgebra of $\End_{\C[[\hbar]]}(\mathcal F_{\tau})$ generated by $\set{h_{i,\hbar}(n),x_{i,\hbar}^\pm(n)}{i\in I,n\in\Z}$.
Then $\mathcal A$ is topologically free.
Let $\mathcal A_+$ be the minimal closed left ideal of $\mathcal A$ containing $h_{i,\hbar}(n)$ and $x_{i,\hbar}^\pm(n)$ ($i\in I$, $n\ge 0$), such that $[\mathcal A_+]=\mathcal A_+$.
Define
\begin{align}
F_\tau(A,\ell)=\mathcal A/\mathcal A_+.
\end{align}
Set $\vac=1+\mathcal A_+\in F_\tau(A,\ell)$.
And for each $i\in I$, we identify $h_{i,\hbar}$ and $x_{i,\hbar}^\pm$ with $h_{i,\hbar}(-1)\vac$ and $x_{i,\hbar}^\pm(-1)\vac$ in $F_\tau(A,\ell)$, respectively.
It is straightforward to verify the following result.

\begin{lem}\label{lem:V-universal-prop}
The topologically free $\C[[\hbar]]$-module $F_\tau(A,\ell)$ equipped with fields
$h_{i,\hbar}(z),x_{i,\hbar}^\pm(z)$ ($i\in I$) is an object in $\mathcal M_{\tau}$.
Moreover, let $W$ be an object in $\mathcal M_\tau$, and let $v_+\in W$ be such that $h_{i,\hbar}(n).v_+=0=x_{i,\hbar}^\pm(n).v_+$ for $i\in I$ and $n\ge 0$.
Then there exists a unique $\mathcal A$-module homomorphism $\varphi:F_\tau(A,\ell)\to W$ such that $\varphi(\vac)=v_+$.
\end{lem}

\begin{prop}\label{prop:F-tau-nlVA}
There exists an $\hbar$-adic nonlocal vertex algebra structure on
$F_\tau(A,\ell)$ such that $\vac$ is the vacuum vector and the vertex operator map $Y_\tau$ is determined by the following conditions
\begin{align}
  Y_\tau(h_{i,\hbar},z)=h_{i,\hbar}(z),\quad Y_\tau(x_{i,\hbar}^\pm,z)=x_{i,\hbar}^\pm(z),
  \quad i\in I.
\end{align}
Moreover, there exists a natural isomorphism between the category of $F_\tau(A,\ell)$-modules and $\mathcal M_\tau$.
\end{prop}

\begin{proof}
Let $(W,h_{i,\hbar}(z),x_{i,\hbar}^\pm(z))$ be an object in the category $\mathcal M_\tau$, and let
\begin{align*}
  U_W=\set{h_{i,\hbar}(z),\,x_{i,\hbar}^\pm(z),\,c}{i\in I}\subset \E_\hbar(W).
\end{align*}
It follows from the relations \eqref{eq:fqva-rel1}, \eqref{eq:fqva-rel2}, \eqref{eq:fqva-rel3} and \eqref{eq:fqva-rel4} that $U_W$ is an $\hbar$-adically $S$-local subset.
From Theorem \ref{thm:construction}, we get an $\hbar$-adic weak quantum vertex algebra
$\<U_W\>\subset \E_\hbar(W)$ containing $U_W,1_W$, such that $1_W$ is the vacuum vector and $Y_\E$ is the vertex operator (see \eqref{eq:def-Y-E}).
Moreover, $W$ is a faithful $\<U_W\>$-module with $Y_W(a(z),z_0)=a(z_0)$ for $a(z)\in\<U_W\>$.

Using \cite[Lemma 4.21]{Li-h-adic}, we get that $\(\<U_W\>,Y_\E(h_{i,\hbar}(z_1),z),Y_\E(x_{i,\hbar}^\pm(z_1),z)\)$ is an object in $\mathcal M_{\tau}$.
From the vacuum property \eqref{eq:vacuum-property}, we have that
\begin{align*}
  &h_{i,\hbar}(z)1_W=Y_\E(h_{i,\hbar}(z_1),z)1_W,\\
  &x_{i,\hbar}^\pm(z)1_W
  =Y_\E(x_{i,\hbar}^\pm(z_1),z)1_W\in \<U_W\>[[z]]\quad\te{for } i\in I.
\end{align*}
Combining this with Lemma \ref{lem:V-universal-prop}, we get an $\mathcal A$-module map $\varphi_W:F_\tau(A,\ell)\to \<U_W\>$, such that $\varphi_W(\vac)=1_W$.

Note that $F_\tau(A,\ell)$ is an object in $\mathcal M_{\tau}$ (see Lemma \ref{lem:V-universal-prop}).
Since $\varphi_{F_\tau(A,\ell)}$ is an $\mathcal A$-module map, we have that
\begin{align*}
  &\varphi_{F_\tau(A,\ell)}(a(z_0)v)=a(z_0)\varphi_{F_\tau(A,\ell)}(v)=Y_\E(a(z),z_0)\varphi_{F_\tau(A,\ell)}(v)\\
  &\quad \te{for } a=h_{i,\hbar}\,\,\te{or}\,\,x_{i,\hbar}^\pm,\,\,v\in F_\tau(A,\ell).
\end{align*}
From the definition of $F_\tau(A,\ell)$ we have that $a(z)\vac\in F_\tau(A,\ell)[[z]]$ for $a=h_{i,\hbar},\,x_{i,\hbar}^\pm$.
Then by using Theorem \ref{thm:gen-weak-qva}, we get a unique $\C[[\hbar]]$-module map $Y_\tau$ such that
\begin{align*}
  Y_\tau(h_{i,\hbar},z)=h_{i,\hbar}(z),\quad Y_\tau(x_{i,\hbar}^\pm,z)=x_{i,\hbar}^\pm(z),\quad i\in I,
\end{align*}
and $(F_\tau(A,\ell),Y_\tau,\vac)$ carries the structure of an $\hbar$-adic weak quantum vertex algebra.

An immediate $\hbar$-adic analogue of \cite[Proposition 6.7]{Li-nonlocal} implies that every $F_\tau(A,\ell)$-module is an object in $\mathcal M_\tau$. On the other hand,
let $(W,h_{i,\hbar}(z),x_{i,\hbar}^\pm(z))$ be an object in the category $\mathcal M_\tau$.
Since $\varphi_W:F_\tau(A,\ell)\to \<U_W\>$ is an $\mathcal A$-module map,
we have that
\begin{align*}
  &\varphi_W(Y_\tau(a,z_0)v)=\varphi_W(a(z_0)v)=a(z_0)\varphi_W(v)=Y_\E(a(z),z_0)\varphi_W(v),
\end{align*}
for any $a=h_{i,\hbar},\,x_{i,\hbar}^\pm$ and $v\in F_\tau(A,\ell)$.
Hence, $\varphi_W$ is an $\hbar$-adic weak quantum vertex algebra homomorphism.
Since $W$ is a faithful $\<U_W\>$-module, we get that $W$ is a $F_\tau(A,\ell)$-module.
\end{proof}

The proof of Proposition \ref{prop:F-tau-nlVA} implies that $F_\tau(A,\ell)$ admits the following universal property.
\begin{prop}\label{prop:universal-nonlocal-va}
Let $V$ be an $\hbar$-adic nonlocal vertex algebra, and let $Y$ be the vertex operator map of $V$.
Suppose there exists $\bar h_{i,\hbar},\bar x_{i,\hbar}^\pm\in V$, such that $(V,Y(\bar h_{i,\hbar},z),Y(\bar x_{i,\hbar}^\pm,z))$ becomes an object of $\mathcal M_\tau$.
Then there exists a unique $\hbar$-adic nonlocal vertex algebra homomorphism $\varphi:F_\tau(A,\ell)\to V$, such that $\varphi(h_{i,\hbar})=\bar h_{i,\hbar}$ and $\varphi(x_{i,\hbar}^\pm)=\bar x_{i,\hbar}^\pm$ for $i\in I$.
\end{prop}

\begin{coro}\label{coro:universal-nonlocal-va}
Let $V$ be an $\hbar$-adic commutative vertex algebra, and let $\al_i,\,e_i^\pm\in V$ ($i\in I$).
Then there exists a unique $\hbar$-adic nonlocal vertex algebra homomorphism $\rho:F_\tau(A,\ell)\to F_\tau(A,\ell)\wh\ot V$, such that
\begin{align*}
  \rho(h_{i,\hbar})=h_{i,\hbar}\ot 1+\vac\ot \al_i,\quad \rho(x_{i,\hbar}^\pm)=x_{i,\hbar}^\pm\ot e_i^\pm,\quad i\in I.
\end{align*}
\end{coro}

Let $H_0$ be the symmetric algebra of the following vector space, and let $H=H_0[[\hbar]]$:
\begin{align*}
  \bigoplus_{i\in I}\(\C[\partial] \(\C\al_i\oplus\C e_i^+\oplus\C e_i^-\)\).
\end{align*}
Then $H$ is a commutative and cocommutative bialgebra with $\Delta$ and $\varepsilon$ uniquely determined by
($i\in I$, $n\in\N$):
\begin{align*}
  &\Delta(\partial^n \al_i)=\partial^n \al_i\ot 1+1\ot \partial^n \al_i,
  \quad\Delta(\partial^n e_i^\pm)=\sum_{k=0}^n\binom{n}{k} \partial^k e_i^\pm\ot \partial^{n-k} e_i^\pm,\\
  &\varepsilon(\partial^n \al_i)=0,
  \quad\varepsilon(\partial^n e_i^\pm)=\delta_{n,0}.
\end{align*}
Let $\partial$ be the derivation on $H$ such that
\begin{align*}
  \partial (\partial^n \al_i)=\partial^{n+1}\al_i,\quad \partial(\partial^n e_i^\pm)=\partial^{n+1}e_i^\pm\quad\te{for } i\in I,\,n\in\N.
\end{align*}
It is straightforward to see that $\Delta\circ\partial=(\partial\ot 1+1\ot\partial)\circ\Delta$ and $\varepsilon\circ \partial=0$.
From Remark \ref{rem:bialg-der} we have that $(H,\partial,\Delta,\varepsilon)$ carries an $\hbar$-adic vertex bialgebra structure.
It is immediate from Corollary \ref{coro:universal-nonlocal-va} that

\begin{lem}\label{lem:comod-nonlocalVA}
There is a unique $\hbar$-adic nonlocal vertex algebra map $\rho:F_\tau(A,\ell)\to F_\tau(A,\ell)\wh\ot H$, such that
\begin{align}
  \rho(h_{i,\hbar})=h_{i,\hbar}\ot 1+\vac\ot \al_i,\quad\rho(x_{i,\hbar}^\pm)=x_{i,\hbar}^\pm\ot e_i^\pm\quad\te{for } i\in I.
\end{align}
Moreover, $(F_\tau(A,\ell),\rho)$ is an $H$-comodule nonlocal vertex algebra.
\end{lem}

Recall from \cite{Li-smash} that a \emph{pseudo-endomorphism} of an $\hbar$-adic nonlocal vertex algebra $V$
is a $\C[[\hbar]]$-module map $A(z):V\to V\wh\ot\C((z))[[\hbar]]$, such that
\begin{align*}
  A(z)\vac=\vac\ot 1,\quad
  A(z_1)Y(u,z_2)v=Y(A(z_1-z_2)u,z_2)v\quad\te{for } u,v\in V.
\end{align*}
And a \emph{pseudo-derivation} of $V$ is a $\C[[\hbar]]$-module map $D(z):V\to V\wh\ot\C((z))[[\hbar]]$, such that
\begin{align*}
  [D(z_1),Y(u,z_2)]=Y(D(z_1-z_2)u,z_2)\quad\te{for } u\in V.
\end{align*}
As a straightforward $\hbar$-adic analogue of \cite[Proposition 2.11]{Li-pseudo}, we have that

\begin{prop}\label{prop:pseudo}
Let $V$ be an $\hbar$-adic nonlocal vertex algebra,
and view $\C((z))[[\hbar]]$ as an $\hbar$-adic vertex algebra with the vertex operator map
$$Y(f(z),z_0)g(z)=f(z-z_0)g(z).$$
Suppose that $A(z)$ is an $\hbar$-adic nonlocal vertex algebra homomorphism from $V$ to the tensor product $\hbar$-adic nonlocal vertex algebra $V\wh\ot \C((z))[[\hbar]]$.
Then $A(z)$ is a pseudo-endomorphism of $V$.
Moreover, let $$D(z):V\to V\wh\ot \C((z))[[\hbar]]$$ be a $\C[[\hbar]]$-module map.
We view $V[\delta]/\delta^2V[\delta]$ as an $\hbar$-adic nonlocal vertex algebra.
If $1+\delta D(z)$ is a $\C[\delta]$-linear pseudo-endomorphism of $V[\delta]/\delta^2V[\delta]$,
then $D(z)$ is a pseudo-derivation of $V$.
\end{prop}

Using Corollary \ref{coro:universal-nonlocal-va} and Proposition \ref{prop:pseudo}, we have:

\begin{lem}\label{lem:def-sigma}
Let $\sigma\in\mathfrak T$. Then for each $i\in I$ there exists a pseudo-derivation $\sigma_i(z)$ on $F_\tau(A,\ell)$ such that
\begin{align}\label{eq:def-sigma-i}
  \sigma_i(z)h_{j,\hbar}=\vac\ot\sigma_{ij}(z),\quad \sigma_i(z)x_{j,\hbar}^\pm=\pm x_{j,\hbar}^\pm\ot \sigma_{ij}^{1,\pm}(z).
\end{align}
And there exists a pseudo-endomorphism $\sigma_i^\pm(x)$ on $F_\tau(A,\ell)$ such that
\begin{align}\label{eq:def-sigma-pm-i}
  \sigma_i^\pm(z)h_{j,\hbar}=h_{j,\hbar}\ot 1\mp \vac\ot\sigma_{ij}^{2,\pm}(z),\quad
  \sigma_i^\pm(z)x_{j,\hbar}^\epsilon= x_{j,\hbar}^\epsilon\ot \sigma_{ij}^{\pm,\epsilon}(z)\inv,
\end{align}
where $\epsilon=\pm$ and $i,j\in I$.
\end{lem}

Let $V$ be an $\hbar$-adic nonlocal vertex algebra.
A subset $U$ of $$\Hom(V,V\wh\ot\C((z))[[\hbar]])$$ is said to be \emph{$\Delta$-closed} (see \cite{Li-smash})
if for any $a(z)\in U$, there exist two sequences $\{a_{(1n)}(z)\}_{n\ge1},\{a_{(2n)}(z)\}_{n\ge 1}\subset U\subset \Hom(V,V\wh\ot\C((z))[[\hbar]])$
that converge to $0$ under the $\hbar$-adic topology and
\begin{align*}
  a(z_1)Y(v,z_2)=\sum_{n\ge 1}Y(a_{(1n)}(z_1-z_2)v,z_2)a_{(2n)}(z_1)\quad\te{for }v\in V.
\end{align*}
Let $B(V)$ be the sum of all $\Delta$-closed subspaces $U$ of
$\Hom(V,V\wh\ot\C((z))[[\hbar]])$ such that
\begin{align}
a(z)\vac\in \C[[\hbar]]\vac\quad\te{for }a(z)\in U.
\end{align}
Note that $\Hom(V,V\wh\ot\C((z))[[\hbar]])\cong\te{End}_{\C((z))[[\hbar]]}(V\wh\ot \C((z))[[\hbar]])$
is an associative algebra.

The following is the immediate $\hbar$-adic analogue of \cite[Proposition 3.4]{Li-smash}.
\begin{prop}
For any $\hbar$-adic nonlocal vertex algebra $V$, $B(V)$ is a $\Delta$-closed associative subalgebra of $\Hom(V,V\wh\ot\C((z))[[\hbar]])$
and it is closed under the derivation $\partial=\frac{d}{dz}$.
Moreover, $B(V)$ is an $\hbar$-adic nonlocal vertex algebra with the vertex operator defined by
\begin{align*}
  Y(a(z),z_0)b(z)=a(z+z_0)b(z)\quad\te{for } a(z),\,b(z)\in B(V).
\end{align*}
Furthermore, $V$ is a $B(V)$-module with the module action $Y_V(a(z),z_0)=a(z_0)$ for $a(z)\in B(V)$.
\end{prop}

We have that

\begin{prop}\label{prop:deforming-triple}
For each $\sigma\in\mathfrak T$,
there exists an $H$-module structure $\sigma(\cdot,z)$ on $F_\tau(A,\ell)$ such that
\begin{align}\label{eq:def-sigma}
  \sigma(\al_i,z)=\sigma_i(z),\quad \sigma(e_i^\pm,z)=\sigma_i^\pm(z),\quad i\in I.
\end{align}
Moreover, $(H,\rho,\sigma)$ is a deforming triple for $F_\tau(A,\ell)$.
\end{prop}

\begin{proof}
Note that there is a $\C[[\hbar]]$-algebra homomorphism $\varphi:H\to B(F_\tau(A,\ell))$ defined by
\begin{align*}
  &\varphi(\partial^{n_1} \al_{i_1} \partial^{n_2}\al_{i_2}\cdots\partial^{n_r}\al_{i_r})
  =\frac{d^{n_1}\sigma_{i_1}(z)}{dz^{n_1}}\frac{d^{n_2}\sigma_{i_2}(z)}{dz^{n_2}}
  \cdots\frac{d^{n_r}\sigma_{i_r}(z)}{dz^{n_r}},\\
  &\varphi(\partial^{n_1} e_{i_1}^{\epsilon_1} \partial^{n_2}e_{i_2}^{\epsilon_2}\cdots\partial^{n_r}e_{i_r}^{\epsilon_r})
  =\frac{d^{n_1}\sigma_{i_1}^{\epsilon_1}(z)}{dz^{n_1}}\frac{d^{n_2}\sigma_{i_2}^{\epsilon_2}(z)}{dz^{n_2}}
  \cdots\frac{d^{n_r}\sigma_{i_r}^{\epsilon_r}(z)}{dz^{n_r}},
\end{align*}
where $i_1,i_2,\dots,i_r\in I$, $\epsilon_1,\epsilon_2,\dots,\epsilon_r=\pm$ and $0\le n_1,\dots,n_r\in\Z$.
It is easy to see that
\begin{align*}
  \varphi(\partial h)=\frac{d}{dz}\varphi(h)\quad\te{for } h\in H.
\end{align*}
Then $\varphi$ is an $\hbar$-adic vertex algebra homomorphism.
Hence, $F_\tau(A,\ell)$ is an $H$-module, and we denote this module action by $\sigma(\cdot,z)$.

Let $H'$ be the maximal $\C[[\hbar]]$-submodule of $H$ such that for any $h\in H'$,
one has $\sigma(h,z)\vac=\varepsilon(h)\vac$ and
\begin{align*}
  \sigma(h,z_1)Y_\tau(u,z_2)v=\sum Y_\tau(\sigma(h_{(1)},z_1-z_2)u,z_2)\sigma(h_{(2)},z_1)v,
\end{align*}
for all $u,v\in F_\tau(A,\ell)$,
where $\Delta(h)=\sum h_{(1)}\ot h_{(2)}$.
It is straightforward to check that $H'$ is an $\hbar$-adic vertex subalgebra of $H$.
Since $\set{\al_i,\,e_i^\pm}{i\in I}\subset H'$ and $H$ is generated by $\set{\al_i,\,e_i^\pm}{i\in I}$, we get that $H'=H$.
Hence, $(F_\tau(A,\ell),\sigma)$ is an $H$-module nonlocal vertex algebra.

Similarly, one can easily checks that
\begin{align}\label{eq:mod-comod-compatible}
  \rho(\sigma(h,z)v)=\(\sigma(h,z)\ot 1\)\circ\rho(v)\quad\te{for } h\in H,\,v\in F_\tau(A,\ell).
\end{align}
This proves that $(H,\rho,\sigma)$ is a deforming triple for $F_\tau(A,\ell)$.
\end{proof}

Let $\tau,\sigma\in\mathfrak T$.
Combining Theorem \ref{thm:deform-va} and Proposition \ref{prop:deforming-triple}, we have the $\hbar$-adic nonlocal vertex algebra
\begin{align}
  \mathfrak D_\sigma^\rho(F_\tau(A,\ell)).
\end{align}
Define $\tau\ast\sigma$ to be
\begin{align*}
  (\tau_{ij}(z)+\sigma_{ij}(z),\tau_{ij}^{1,\pm}(z)+\sigma_{ij}^{1,\pm}(z),\tau_{ij}^{2,\pm}(z)+\sigma_{ij}^{2,\pm}(z),
    \tau_{ij}^{\epsilon_1,\epsilon_2}(z)\sigma_{ij}^{\epsilon_1,\epsilon_2}(z))_{i,j\in I}^{\epsilon_1,\epsilon_2=\pm}.
\end{align*}
Then $\mathfrak T$ is an abelian group under $\ast$ with identity $\varepsilon$ defined by
\begin{align}\label{eq:identity-T}
  \varepsilon_{ij}(z)=0=\varepsilon_{ij}^{1,\pm}(z)=\varepsilon_{ij}^{2,\pm}(z),\quad \varepsilon_{ij}^{\epsilon_1,\epsilon_2}(z)=1,
\end{align}
where $i,j\in I$, $\epsilon_1,\epsilon_2=\pm$.

\begin{lem}\label{lem:mod-commute}
Let $\tau,\tau',\tau''\in\mathfrak T$.
Then the following relations hold true on $F_\tau(A,\ell)$
\begin{align*}
  [\tau'(h,z_1),\tau''(h',z_2)]=0\quad\te{for }h,h'\in H.
\end{align*}
\end{lem}

\begin{proof}
From \eqref{eq:def-sigma}, \eqref{eq:def-sigma-i} and \eqref{eq:def-sigma-pm-i}, we have that
\begin{align*}
  [\tau'(a,z_1),\tau''(b,z_2)]v=0.
\end{align*}
for $a,b\in \set{\al_i,e_i^\pm}{i\in I}$, $v\in\{h_{i,\hbar},x_{i,\hbar}^\pm\,\mid\,i\in I\}$.
Let $V$ be the $\C[[\hbar]]$-submodule of $F_\tau(A,\ell)$ consisting of elements $v$ such that
\begin{align*}
 [\tau'(a,z_1),\tau''(b,z_2)]v=0\quad \te{for }a,b=\al_i,e_i^\pm,\,i\in I.
\end{align*}
It is easy to verify that $V$ is a subalgebra of $F_\tau(A,\ell)$.
Since $F_\tau(A,\ell)$ is generated by $\{h_{i,\hbar},\,x_{i,\hbar}^\pm\,\mid\,i\in I\}$,
we get $V=F_\tau(A,\ell)$.
Since $H$ is a commutative $\hbar$-adic vertex algebra generated by $\set{\al_i,e_i^\pm}{i\in I}$, we have that
\begin{align*}
  [\tau'(h,z_1),\tau''(h',z_2)]v=0\quad\te{for }h,h'\in H,\,v\in F_\tau(A,\ell).
\end{align*}
Therefore, we complete the proof.
\end{proof}

It follows from Lemma \ref{lem:mod-commute} and Proposition \ref{prop:L-H-rho-V-compostition} that
\begin{align}\label{eq:deform-mult}
  \mathfrak D_{\tau'}^\rho\(\mathfrak D_{\tau''}^\rho(F_\tau(A,\ell))\)
  =\mathfrak D_{\tau'\ast\tau''}^\rho(F_\tau(A,\ell))\quad \te{for }\tau,\tau',\tau''\in\mathfrak T.
\end{align}
Moreover, we have that

\begin{prop}\label{prop:classical-limit}
Let $\tau,\sigma\in\mathfrak T$. Then
\begin{align*}
  \mathfrak D_\sigma^\rho(F_\tau(A,\ell))=F_{\tau\ast\sigma}(A,\ell).
\end{align*}
Moreover, $F_\tau(A,\ell)/\hbar F_\tau(A,\ell)\cong F(A,\ell)$.
\end{prop}

\begin{proof}
From Proposition \ref{prop:L-H-rho-V-compostition}, we have that
\begin{align*}
  \mathfrak D_\sigma^\rho(Y_\tau)(h_{i,\hbar},z)=Y_\tau(h_{i,\hbar},z)+\sigma_i(z),\quad
  \mathfrak D_\sigma^\rho(Y_\tau)(x_{i,\hbar}^\pm,z)=Y_\tau(x_{i,\hbar}^\pm,z)\sigma_i^\pm(z).
\end{align*}
Using Lemma \ref{lem:def-sigma}, one can easily check that $$(\mathfrak D_\sigma^\rho(F_\tau(A,\ell)),\mathfrak D_\sigma^\rho(Y_\tau)(h_{i,\hbar},z),\mathfrak D_\sigma^\rho(Y_\tau)(x_{i,\hbar}^\pm,z))$$ is an object of $\mathcal M_{\tau\ast\sigma}$.
Using Proposition \ref{prop:universal-nonlocal-va}, we get an $\hbar$-adic nonlocal vertex algebra homomorphism $f_{\tau,\sigma}$ from $F_{\tau\ast\sigma}(A,\ell)$ to $\mathfrak D_\sigma^\rho(F_\tau(A,\ell))$.
From Remark \ref{rem:deformed-hom} we see that $f_{\tau,\sigma}$ is also an $\hbar$-adic nonlocal vertex algebra homomorphism from $\mathfrak D_{\sigma\inv}^\rho(F_{\tau\ast\sigma}(A,\ell))$ to
$\mathfrak D_{\sigma\inv}^\rho\(\mathfrak D_\sigma^\rho(F_\tau(A,\ell))\)$.
It follows from Remark \ref{rem:trivial-deform} and \eqref{eq:deform-mult} that
$\mathfrak D_{\sigma\inv}^\rho\(\mathfrak D_\sigma^\rho(F_\tau(A,\ell))\)=F_\tau(A,\ell)$.
Then $f_{\tau,\sigma}$ becomes an $\hbar$-adic nonlocal vertex algebra homomorphism from $\mathfrak D_{\sigma\inv}^\rho(F_{\tau\ast\sigma}(A,\ell))$ to $F_\tau(A,\ell)$.
Replacing $\tau$ with $\tau\ast\sigma$ and replacing $\sigma$ with $\sigma\inv$, we get an $\hbar$-adic nonlocal vertex algebra homomorphism
\begin{align*}
  &f_{\tau\ast\sigma,\sigma\inv}:F_\tau(A,\ell)\to \mathfrak D_{\sigma\inv}^\rho(F_{\tau\ast\sigma}(A,\ell)).
\end{align*}
Hence, $f_{\tau,\sigma}\circ f_{\tau\ast\sigma,\sigma\inv}$ is an $\hbar$-adic nonlocal vertex algebra endomorphism on $F_\tau(A,\ell)$ which maps $a$ to $a$ for all $a\in\set{h_{i,\hbar},x_{i,\hbar}^\pm}{i\in I}$.
From Proposition \ref{prop:universal-nonlocal-va}, we get that
\begin{equation}\label{tb:classical-lim-temp1}
 \begin{tikzcd}
    F_\tau(A,\ell)\arrow[rr,"{f_{\tau\ast\sigma,\sigma\inv}}"]\arrow[rrd,equal]& &
        \mathfrak D_{\sigma\inv}^\rho(F_{\tau\ast\sigma}(A,\ell))\arrow[d,"{f_{\tau,\sigma}}"]\\
    &&F_\tau(A,\ell)
  \end{tikzcd}
\end{equation}
By replacing $\tau$ and $\sigma$ with $\tau\ast\sigma$ and $\sigma\inv$ in \eqref{tb:classical-lim-temp1}, we get that
\begin{equation}
 \begin{tikzcd}
    F_{\tau\ast\sigma}(A,\ell)\ar[rr,"{f_{\tau,\sigma}}"]\ar[rrd,equal]& &
        \mathfrak D_\sigma^\rho(F_{\tau}(A,\ell))\ar[d,"{f_{\tau\ast\sigma,\sigma\inv}}"]\\
    &&F_{\tau\ast\sigma}(A,\ell)
  \end{tikzcd}
\end{equation}
Notice that
\begin{align*}
  &\mathfrak D_{\sigma\inv}^\rho(F_{\tau\ast\sigma}(A,\ell))/\hbar \mathfrak D_{\sigma\inv}^\rho(F_{\tau\ast\sigma}(A,\ell))
  =F_{\tau\ast\sigma}(A,\ell)/\hbar F_{\tau\ast\sigma}(A,\ell),\\
  &\mathfrak D_\sigma^\rho(F_{\tau}(A,\ell))/\hbar \mathfrak D_\sigma^\rho(F_{\tau}(A,\ell))
  =F_{\tau}(A,\ell)/\hbar F_{\tau}(A,\ell).
\end{align*}
Then $f_{\tau,\sigma}$ and $f_{\tau\ast\sigma,\sigma\inv}$ induce the following $\C$-linear isomorphisms
\begin{equation}\label{eq:classical-limit-equal}
  \begin{tikzcd}
    F_{\tau\ast\sigma}(A,\ell)/\hbar F_{\tau\ast\sigma}(A,\ell)\arrow[r, rightharpoonup, shift left=0.25ex]\arrow[r, leftharpoondown, shift right=0.25ex]&F_{\tau}(A,\ell)/\hbar F_{\tau}(A,\ell)\ar[d,equal]\\
    &\mathfrak D_\sigma^\rho(F_{\tau}(A,\ell))/\hbar \mathfrak D_\sigma^\rho(F_{\tau}(A,\ell)).
  \end{tikzcd}
\end{equation}
Since both $F_{\tau\ast\sigma}(A,\ell)$ and $\mathfrak D_\sigma^\rho(F_{\tau}(A,\ell))$ are topologically free,
we get that $f_{\tau,\sigma}:F_{\tau\ast\sigma}(A,\ell)\to \mathfrak D_\sigma^\rho(F_{\tau}(A,\ell))$ is an $\hbar$-adic nonlocal vertex algebra isomorphism.

From \eqref{eq:identity-T}, we see that $F_\varepsilon(A,\ell)=F(A,\ell)[[\hbar]]$ as $\hbar$-adic vertex algebras.
Recall that $\varepsilon$ is the identity of $\mathcal T$.
Then by replacing $\tau$ and $\sigma$ with $\varepsilon$ and $\tau$ in \eqref{eq:classical-limit-equal}, we get that $F_\tau(A,\ell)/\hbar F_\tau(A,\ell)\cong F(A,\ell)$.
\end{proof}

For $\tau\in\mathfrak T$, we note that
\begin{align*}
  \tau\inv=(-\tau_{ij}(z),-\tau_{ij}^{1,\pm}(z),-\tau_{ij}^{2,\pm}(z),\tau_{ij}^{\epsilon_1,\epsilon_2}(z)\inv)_{i,j\in I}^{\epsilon_1,\epsilon_2=\pm}.
\end{align*}
It is immediate from Theorem \ref{thm:qva} that

\begin{thm}\label{thm:S-tau}
For any $\tau\in\mathfrak T$, $F_\tau(A,\ell)$ is an $\hbar$-adic quantum vertex algebra with the quantum Yang-Baxter operator $S_\tau(z)$ defined by
\begin{align*}
  S_\tau(z)(v\ot u)=\sum \tau(u_{(2)},-z)v_{(1)}\ot \tau\inv(v_{(2)},z)u_{(1)}\quad \te{for }u,v\in F_\tau(A,\ell).
\end{align*}
Moreover, for any $i,j\in I$ and $\epsilon_1,\epsilon_2=\pm$, we have that
\begin{align}
  &S_\tau(z)(h_{j,\hbar}\ot h_{i,\hbar})=h_{j,\hbar}\ot h_{i,\hbar}+\vac\ot\vac\ot \(\tau_{ij}(-z)-\tau_{ji}(z)\),\label{eq:S-tau-1}\\
  &S_\tau(z)(x_{j,\hbar}^\pm\ot h_{i,\hbar})=x_{j,\hbar}^\pm\ot h_{i,\hbar}\pm x_{j,\hbar}^\pm\ot \vac \ot \big(\tau_{ij}^{1,\pm}(-z)+\tau_{ji}^{2,\pm}(z)\big),\label{eq:S-tau-2}\\
  &S_\tau(z)(h_{j,\hbar}\ot x_{i,\hbar}^\pm)=h_{j,\hbar}\ot x_{i,\hbar}^\pm\mp\vac\ot x_{i,\hbar}^\pm
  \ot\big(\tau_{ij}^{2,\pm}(-z)+\tau_{ji}^{1,\pm}(z)\big),\label{eq:S-tau-3}\\
  &S_\tau(z)(x_{j,\hbar}^{\epsilon_1}\ot x_{i,\hbar}^{\epsilon_2})=x_{j,\hbar}^{\epsilon_1}\ot x_{i,\hbar}^{\epsilon_2}\ot\tau_{ji}^{\epsilon_1,\epsilon_2}(z)  \tau_{ij}^{\epsilon_2,\epsilon_1}(-z)\inv.\label{eq:S-tau-4}
\end{align}
\end{thm}

\section{Construction of quantum affine vertex algebras}\label{sec:V-tau-L-tau}

Let $\ell\in\C$.
In the rest of this paper, we fix a special $$\tau=(\tau_{ij}(z),\tau_{ij}^{1,\pm}(z),\tau_{ij}^{2,\pm}(z),\tau_{ij}^{\epsilon_1,\epsilon}(z))_{i,j\in I}^{\epsilon_1,\epsilon_2=\pm}\in\mathfrak T$$
defined as follows
\begin{align}
  &\tau_{ij}(z)=[r_ia_{ij}]_{q^{\pd{z}}}[r\ell]_{q^{\pd{z}}}q^{r\ell\pd{z}}\frac{e^{-z}}{(1-e^{-z})^2}
    -r_ia_{ij}r\ell z^{-2},\label{eq:tau-1}\\
  &\tau_{ij}^{1,\pm}(z)=\tau_{ij}^{2,\pm}(z)=[r_ia_{ij}]_{q^{\pd{z}}}q^{r\ell\pd{z}}\frac{1+e^{-z}}{2-2e^{-z}}
    -r_ia_{ij}z\inv,\label{eq:tau-2}\\
  &\tau_{ij}^{\pm,\pm}(z)=\begin{cases}
  \(e^{z/2}-e^{-z/2}\)\inv\(q_i\inv e^{z/2}-q_i e^{-z/2}\),&\mbox{if }a_{ij}> 0,\\
  z\inv\(q_i^{-a_{ij}/2} e^{z/2}-q_i^{a_{ij}/2}e^{-z/2}\),&\mbox{if }a_{ij}\le 0,
  \end{cases}
  \label{eq:tau-3}\\
  &\tau_{ij}^{+,-}(z)=z^{-\delta_{ij}}(z+2r\ell\hbar)^{\delta_{ij}},\label{eq:tau-4}\\
  &\tau_{ij}^{-,+}(z)=z^{-\delta_{ij}}(z-2r\ell\hbar)^{\delta_{ij}}
  g_{ij,\hbar}(z)\inv,\label{eq:tau-5}
\end{align}
where $g_{ij,\hbar}(z)=(1-q_i^{a_{ij}}e^{-z})/(q_i^{a_{ij}}-e^{-z})$.
Then the equations \eqref{eq:fqva-rel1}, \eqref{eq:fqva-rel2}, \eqref{eq:fqva-rel3}, \eqref{eq:fqva-rel4} become:
\begin{align}
  &[h_{i,\hbar}(z_1),h_{j,\hbar}(z_2)]\label{eq:local-h-1}\\
  &\quad=[r_ia_{ij}]_{q^{\pd{z_2}}}[r\ell]_{q^{\pd{z_2}}}
  \(\iota_{z_1,z_2}q^{-r\ell\pd{z_2}}-\iota_{z_2,z_1}q^{r\ell\pd{z_2}}\)\frac{e^{-z_1+z_2}}{(1-e^{-z_1+z_2})^2},\nonumber\\
  &[h_{i,\hbar}(z_1),x_{j,\hbar}^\pm(z_2)]\label{eq:local-h-2}\\
  &\quad=\pm x_{j,\hbar}^\pm(z_2)[r_ia_{ij}]_{q^{\pd{z_2}}}
  \(\iota_{z_1,z_2}q^{-r\ell\pd{z_2}}-\iota_{z_2,z_1}q^{r\ell\pd{z_2}}\)\frac{1+e^{-z_1+z_2}}{2-2e^{-z_1+z_2}},\nonumber\\
  &\iota_{z_1,z_2}\frac{1-q_i^{a_{ij}}e^{-z_1+z_2}}{(1-e^{-z_1+z_2})^{\delta_{ij}}} x_{i,\hbar}^\pm(z_1)x_{j,\hbar}^\pm(z_2)\label{eq:local-h-3-pre}\\
  &\quad  =\iota_{z_2,z_1}\frac{q_i^{a_{ij}}-e^{-z_1+z_2}}{(1-e^{-z_1+z_2})^{\delta_{ij}}} x_{j,\hbar}^\pm(z_2)x_{i,\hbar}^\pm(z_1),\nonumber\\
  &(z_1-z_2)^{\delta_{ij}}(z_1-z_2+2r\ell\hbar)^{\delta_{ij}}\label{eq:local-h-4}\\
  &\quad\times\(x_{i,\hbar}^+(z_1)x_{j,\hbar}^-(z_2)
    -\iota_{z_2,z_1}g_{ij,\hbar}(z_1-z_2)
  x_{j,\hbar}^-(z_2)x_{i,\hbar}^+(z_1)\)=0.\nonumber
\end{align}
For $i,j\in I$, we set
\begin{align}
  &f_{ij,\hbar}(z)=z^{n_{ij}}\tau_{ij}^{+,+}(z)=
  \frac{(q_i^{-a_{ij}/2} e^{z/2}-q_i^{a_{ij}/2}e^{-z/2})}
  {(e^{z/2}-e^{-z/2})^{\delta_{ij}}}.
\end{align}
Then the relation \eqref{eq:local-h-3-pre} is equivalent to the following equation:
\begin{align}
  &\iota_{z_1,z_2}f_{ij}(z_1-z_2) x_{i,\hbar}^\pm(z_1)x_{j,\hbar}^\pm(z_2)\label{eq:local-h-3}\\
    =&-(-1)^{\delta_{ij}}\iota_{z_2,z_1}f_{ji}(z_2-z_1) x_{j,\hbar}^\pm(z_2)x_{i,\hbar}^\pm(z_1).\nonumber
\end{align}

It is straightforward to verify that
\begin{lem}\label{lem:special-tau-tech0}
Let $i,j\in I$. Then
\begin{align}
  &\tau_{ij}(z)-\tau_{ji}(-z)=[r_ia_{ij}]_{q^{\pd{z}}}[r\ell]_{q^{\pd{z}}}
  \(q^{r\ell\pd{z}}-q^{-r\ell\pd{z}}\)\frac{e^{-z}}{(1-e^{-z})^2},
    \label{eq:special-tau-tech0-1}\\
  &\tau_{ij}^{1,+}(z)+\tau_{ji}^{2,+}(-z)
  =[r_ia_{ij}]_{q^{\pd{z}}}\(q^{r\ell\pd{z}}-q^{-r\ell\pd{z}}\)\frac{1+e^{-z}}{2-2e^{-z}}\label{eq:special-tau-tech0-2}\\
  &\quad=\tau_{ij}^{1,-}(z)+\tau_{ji}^{2,-}(-z),
    \nonumber\\
  &\tau_{ij}^{\pm,\pm}(z)\tau_{ji}^{\pm,\pm}(-z)\inv=
  g_{ij,\hbar}(z)
  =\tau_{ij}^{\pm,\mp}(z)\inv\tau_{ji}^{\mp,\pm}(-z).
  \label{eq:special-tau-tech0-3-4}
\end{align}
\end{lem}

Moreover, we have that

\begin{lem}\label{lem:special-tau-tech1}
Let
\begin{align}\label{eq:def-op-F}
F(z)=(q^z-q^{-z})/z\in \hbar\C[z^2][[\hbar]].
\end{align}
Then for any $i,j\in I$, we have that
\begin{align}
  &F\(\pd{z}\)\(\tau_{ij}(z)-\tau_{ji}(-z)\)\label{eq:special-tau-tech1-3}\\
  &\quad=\( q^{-r\ell\pd{z}}-q^{r\ell\pd{z}} \)\(\tau_{ij}^{1,+}(z)+\tau_{ji}^{2,+}(-z)\),\nonumber\\
  &F\(\pd{z}\)\(\tau_{ij}^{1,\pm}(z)+\tau_{ji}^{2,\pm}(-z)\)
  =\mp\(q^{r\ell\pd{z}}-q^{-r\ell\pd{z}}\)\log\frac{\tau_{ij}^{+,\pm}(z)}{\tau_{ji}^{\pm,+}(-z)}.\label{eq:special-tau-tech1-4}
\end{align}
\end{lem}

\begin{proof}
From \eqref{eq:special-tau-tech0-2}, we have that
\begin{align*}
  &\(q^{-r\ell\pd{z}}-q^{r\ell\pd{z}}\)(\tau_{ij}^{1,+}(z)+\tau_{ji}^{2,+}(-z))\\
  =&\(q^{-\pd{z}}-q^{\pd{z}}\)[r\ell]_{q^{\pd{z}}}(\tau_{ij}^{1,+}(z)+\tau_{ji}^{2,+}(-z))\\
  =&F\(\pd{z}\)\pd{z}[r\ell]_{q^{\pd{z}}}(\tau_{ij}^{1,+}(z)+\tau_{ji}^{2,+}(-z))\\
  =&F\(\pd{z}\)[r_ia_{ij}]_{q^{\pd{z}}}[r\ell]_{q^{\pd{z}}}
  \(q^{r\ell\pd{z}}-q^{-r\ell\pd{z}}\)\frac{e^{-z}}{(1-e^{-z})^2}.
\end{align*}
Combining this with \eqref{eq:special-tau-tech0-1}, we prove the equation \eqref{eq:special-tau-tech1-3}.
The proof of the equation \eqref{eq:special-tau-tech1-4} is similar.
\end{proof}

\begin{de}\label{de:V-tau}
For $\ell\in\C$, we let $R_1^\ell$ be the minimal closed ideal of $F_\tau(A,\ell)$ such that $[R_1^\ell]=R_1^\ell$ and contains the following elements
\begin{align}
  &\(x_{i,\hbar}^+\)_0x_{i,\hbar}^--(q_i-q_i\inv)\inv
    \(\vac-E(h_{i,\hbar})\)\quad\te{for } i\in I,\label{eq:x+0x-}\\
  &\(x_{i,\hbar}^+\)_1x_{i,\hbar}^-+2r\ell\hbar(q_i-q_i\inv)\inv
  E(h_{i,\hbar})\quad\te{for } i\in I,\label{eq:x+1x-}\\
  &\(x_{i,\hbar}^\pm\)_0^{m_{ij}}x_{j,\hbar}^\pm\quad\te{for } i,j\in I\,\te{with }a_{ij}<0,\label{eq:serre}
\end{align}
where
\begin{align}\label{eq:def-E}
  E(h_{i,\hbar})=\(\frac{F(r_i+r\ell)}{F(r_i-r\ell)}\)^\half
    \exp\(\(-q^{-r\ell\partial} F\(\partial\)h_{i,\hbar}\)_{-1}\)\vac.
\end{align}
Define
\begin{align}
  V_{\hat\fg,\hbar}(\ell,0)=F_\tau(A,\ell)/R_1^\ell.
\end{align}
\end{de}

\begin{de}\label{de:L-tau}
Let $\ell\in \Z_+$.
We let $R_2^\ell$ be the minimal closed ideal of $V_{\hat\fg,\hbar}(\ell,0)$ such that $[R_2^\ell]=R_2^\ell$ and contains the following elements
\begin{align}
  & \(x_{i,\hbar}^\pm\)_{-1}^{r\ell/r_i}x_{i,\hbar}^\pm\quad\te{for } i\in I.\label{eq:integrable}
\end{align}
Define
\begin{align}
  L_{\hat\fg,\hbar}(\ell,0)=V_{\hat\fg,\hbar}(\ell,0)/R_2^\ell.
\end{align}
\end{de}

It is immediate from Propositions \ref{prop:universal-va}, \ref{prop:classical-limit} and Definitions \ref{de:aff-vas}, \ref{de:V-tau}, \ref{de:L-tau} that
\begin{prop}\label{prop:classical-limit-V-L-pre}
For $\ell\in\C$, there is a surjective vertex algebra homomorphism from $V_{\hat\fg}(\ell,0)$ to $V_{\hat\fg,\hbar}(\ell,0)/\hbar V_{\hat\fg,\hbar}(\ell,0)$ such that
\begin{align}\label{eq:classical-limit-V-L-pre}
  h_i\mapsto h_{i,\hbar},\quad x_i^\pm\mapsto x_{i,\hbar}^\pm\quad\te{for } i\in I.
\end{align}
If $\ell\in \Z_+$,
then \eqref{eq:classical-limit-V-L-pre} defines a surjective vertex algebra homomorphism from $L_{\hat\fg}(\ell,0)$ to $L_{\hat\fg,\hbar}(\ell,0)/\hbar L_{\hat\fg,\hbar}(\ell,0)$.
\end{prop}

The main purpose of this section is to prove the following results.
\begin{thm}\label{thm:quotient-algs}
Let $\ell\in\C$. Then $V_{\hat\fg,\hbar}(\ell,0)$ is an $\hbar$-adic quantum vertex algebra.
Moreover, if $\ell\in \Z_+$, then $L_{\hat\fg,\hbar}(\ell,0)$ is also an $\hbar$-adic quantum vertex algebra.
Furthermore, the quantum Yang-Baxter operators of both $V_{\hat\fg,\hbar}(\ell,0)$ and $L_{\hat\fg,\hbar}(\ell,0)$ satisfy the relations \eqref{eq:S-tau-1}, \eqref{eq:S-tau-2}, \eqref{eq:S-tau-3} and \eqref{eq:S-tau-4}.
\end{thm}

We start with some technical results.

\begin{lem}\label{lem:tech-2}
Let $W$ be a topologically free $\C[[\hbar]]$-module, $$\beta(z)=\sum_{m\in\Z}\beta_mz^{-m-1}\in W((z))$$ and $f(z)=\sum_{i=0}^n a_iz^i\in\C[z]$, such that $a_n=1$. Suppose that
\begin{align*}
  \Sing_z z^k \hbar^nf(z/\hbar)\beta(z)=0\quad \mbox{for some }k\in\Z.
\end{align*}
Then $\beta_m\in\Span_{\C[[\hbar]]}\set{\beta_i}{k\le i<k+n}$ for all $m\ge k$.
Moreover, if $\beta_i=0$ for all $k\le i<k+n$, then $\Sing_zz^k\beta(z)=0$.
\end{lem}

\begin{proof}
Notice that
\begin{align*}
  0=\Sing_z z^k\hbar^nf(z/\hbar)\beta(z)=\sum_{m\ge0}z^{-m-1}\sum_{i=0}^na_i\hbar^{n-i}\beta_{m+k+i}.
\end{align*}
It follows that
\begin{align*}
  \beta_{m+n}=-\sum_{i=0}^{n-1}a_i\hbar^{n-i}\beta_{m+i}\quad\te{for }m\ge k.
\end{align*}
By using induction on $m$, we complete the proof.
\end{proof}

\begin{lem}\label{lem:tech-3}
Let $V$ be an $\hbar$-adic nonlocal vertex algebra, and let $u,v\in V$,
$f(z)=(\sum_{i=0}^na_iz^i)/(\sum_{j=0}^m b_jz^j)\in\C(z)$.
Suppose
\begin{align}\label{eq:tech-3-1}
  \iota_{z_1,z_2,\hbar} \hbar^{n-m}f\big((z_1-z_2)/\hbar\big)Y(u,z_1)Y(v,z_2)\in\E_\hbar^{(2)}(V).
\end{align}
Then we have that
\begin{align*}
  \Sing_z \iota_{z,\hbar}\hbar^{n-m}f(z/\hbar)Y(u,z)v=0.
\end{align*}
\end{lem}

\begin{proof}
Let $U=\set{Y(u,z)}{u\in V}\subset \E_\hbar(V)$.
From \eqref{eq:weak-asso}, we get an $\hbar$-adic nonlocal vertex algebra homomorphism $V\to \<U\>$ defined by $u\mapsto Y(u,z)$.
Then from \eqref{eq:tech-3-1}, we have that
\begin{align*}
  &\hbar^{n-m}f(z/\hbar)Y(Y(u,z)v,z_2)\\
  =&\hbar^{n-m}f\big((z_1-z_2)/\hbar\big)Y_\E\(Y(u,z_2),z\)Y(v,z_2)\\
  =&\(\hbar^{n-m}f\big((z_1-z_2)/\hbar\big)Y(u,z_1)Y(v,z_2)\)|_{z_1=z_2+z}.
\end{align*}
Notice that
\begin{align*}
  &\(\hbar^{n-m}f\big((z_1-z_2)/\hbar\big)Y(u,z_1)Y(v,z_2)\)|_{z_1=z_2+z}\\
  &\quad\in \Hom(V,V((z_2))[[z]])+\hbar^n\End V[[z_2^{\pm 1},z]]\quad \te{for all }n\ge0,
\end{align*}
and that $f(z,\hbar)Y(Y(u,z)v,z_2)\vac\in V[[z,z\inv,z_2]]$.
It follows that
\begin{align*}
  &\hbar^{n-m}f(z/\hbar)Y(Y(u,z)v,z_2)\vac\\
  =&\(\hbar^{n-m}f\big((z_1-z_2)/\hbar\big)Y(u,z_1)Y(v,z_2)\vac\)|_{z_1=z_2+z}\in V[[z_2,z]].
\end{align*}
Setting $z_2=0$, we get that
\begin{align*}
  \hbar^{n-m}f(z/\hbar)Y(u,z)v=\(\hbar^{n-m}f(z/\hbar)Y(Y(u,z)v,z_2)\vac\)|_{z_2=0}\in V[[z]].
\end{align*}
Therefore, $\Sing_z\iota_{z,\hbar}\hbar^{n-m}f(z/\hbar)Y(u,z)v=0$.
\end{proof}

Combining Lemmas \ref{lem:tech-3} and equations \eqref{eq:local-h-3}, \eqref{eq:local-h-4}, we have that
\begin{lem}\label{lem:qSing}
In the $\hbar$-adic quantum vertex algebra $F_\tau(A,\ell)$, we have that
\begin{align}
  &\Sing_z(z-r_ia_{ij}\hbar)Y_\tau(x_{i,\hbar}^\pm,z)x_{j,\hbar}^\pm=0,\quad\mbox{if }a_{ij}<0,\label{eq:qSing1<=0}\\
  &\Sing_zz\inv(z-r_ia_{ii}\hbar)Y_\tau(x_{i,\hbar}^\pm,z)x_{i,\hbar}^\pm=0,\label{eq:qSing1>0}\\
  &\Sing_zY_\tau(x_{i,\hbar}^+,z)x_{j,\hbar}^-=0,\quad\mbox{if }i\ne j,\label{eq:qSing2neq}\\
  &\Sing_zz(z+2r\ell\hbar)Y_\tau(x_{i,\hbar}^+,z)x_{i,\hbar}^-=0.\label{eq:qSing2eq}
\end{align}
\end{lem}

We have that

\begin{prop}\label{prop:qCartan}
For $i\in I$, we define $A_i(z)$ to be
\begin{align*}
  &\Sing_zY_\tau(x_{i,\hbar}^+,z)x_{i,\hbar}^--(q_i-q_i\inv)\inv \(\vac z\inv-
  E(h_{i,\hbar}) (z+2r\ell\hbar)\inv\).
\end{align*}
Then $A_i(z)=0$ in $V_{\hat\fg,\hbar}(\ell,0)$.
Moreover, we have that
\begin{align}
  &Y_\tau(x_{i,\hbar}^+,z_1)Y_\tau(x_{j,\hbar}^-,z_2)-
  \iota_{z_2,z_1}g_{ij,\hbar}(z_1-z_2)
  Y_\tau(x_{j,\hbar}^-,z_2)Y_\tau(x_{i,\hbar}^+,z_1)\nonumber \\
  =&\frac{\delta_{ij}}{q_i-q_i\inv} \(z_1\inv\delta\(\frac{z_2}{z_1}\)
  -Y_\tau\(E(h_{i,\hbar}),z_2\)z_1\inv\delta\(\frac{z_2-2r\ell\hbar}{z_1}\)\).\label{eq:local-h-5}
\end{align}
\end{prop}

\begin{proof}
Let
\begin{align*}
  \bar A_i(z)=Y_\tau(x_{i,\hbar}^+,z)x_{i,\hbar}^--(q_i-q_i\inv)\inv \(\vac z\inv-
  E(h_{i,\hbar}) (z+2r\ell\hbar)\inv\).
\end{align*}
It is easy to see that $A_i(z)=\Sing_z \bar A_i(z)$.
From \eqref{eq:qSing2eq}, we have that $\Sing_zz(z+2r\ell\hbar)\bar A_i(z)=0$.
Then it follows from Lemma \ref{lem:tech-2} that
\begin{align*}
  A_i(z)=\Sing_z\bar A_i(z)\in\bigoplus_{n=0}^1\C[z\inv][[\hbar]]\( \Res_z z^n\bar A_i(z)\).
\end{align*}
Recall from \eqref{eq:x+0x-} and \eqref{eq:x+1x-} that $\Res_z z^n\bar A_i(z)=0$ for $n=0,1$.
Therefore, $A_i(z)=0$ in $V_{\hat\fg,\hbar}(\ell,0)$.

From \cite[(2.25)]{Li-h-adic}, we have that
\begin{align*}
  &z_0\inv\delta\(\frac{z_1-z_2}{z_0}\)Y_\tau(x_{i,\hbar}^+,z_1)Y_\tau(x_{j,\hbar}^-,z_2)\\
  &-z_0\inv\delta\(\frac{z_2-z_1}{-z_0}\)\iota_{z_2,z_1}g_{ij,\hbar}(z_1-z_2)
  Y_\tau(x_{j,\hbar}^-,z_2)Y_\tau(x_{i,\hbar}^+,z_1)\\
  =&z_1\inv\delta\(\frac{z_2+z_0}{z_1}\)Y_\tau\(Y_\tau(x_{i,\hbar}^+,z_0)x_{j,\hbar}^-,z_2\).
\end{align*}
Taking $\Res_{z_0}$, we get that
\begin{align*}
  &Y_\tau(x_{i,\hbar}^+,z_1)Y_\tau(x_{j,\hbar}^-,z_2)-
  \iota_{z_2,z_1}g_{ij,\hbar}(z_1-z_2)
  Y_\tau(x_{j,\hbar}^-,z_2)Y_\tau(x_{i,\hbar}^+,z_1)\\
  =&\Res_{z_0}z_1\inv\delta\(\frac{z_2+z_0}{z_1}\)Y_\tau\(Y_\tau(x_{i,\hbar}^+,z_0)x_{j,\hbar}^-,z_2\)\\
  =&\Res_{z_0}z_1\inv\delta\(\frac{z_2+z_0}{z_1}\)Y_\tau\(\Sing_{z_0}Y_\tau(x_{i,\hbar}^+,z_0)x_{j,\hbar}^-,z_2\)\\
  =&\Res_{z_0}z_1\inv\delta\(\frac{z_2+z_0}{z_1}\)(q_i-q_i\inv)\inv\(
  z_0\inv-
  Y_\tau\(E(h_{i,\hbar}),z_2\)(z_0+2r\ell\hbar)\inv
  \)\\
  =&\delta_{ij}(q_i-q_i\inv)\inv
  \(z_1\inv\delta\(\frac{z_2}{z_1}\)-
  Y_\tau\(E(h_{i,\hbar}),z_2\)z_1\inv\delta\(\frac{z_2-2r\ell\hbar}{z_1}\)\),
\end{align*}
where the third equation follows from the first statement.
\end{proof}

For any positive integer $k$ and $i_1,\dots,i_k\in I$, we set
\begin{align}\label{eq:def-normal-ordering}
  &Y_{i_1,\dots,i_k}^\pm(z_1,\dots,z_k)\\
  =&\(\prod_{1\le a<b\le k}f_{i_a,i_b,\hbar}(z_a-z_b)\)Y_\tau(x_{i_1,\hbar}^\pm,z_1)Y_\tau(x_{i_2,\hbar}^\pm,z_2)\cdots Y_\tau(x_{i_k,\hbar}^\pm,z_k).\nonumber
\end{align}
It is immediate from equations \eqref{eq:local-h-3} and \eqref{eq:local-h-4} that
\begin{lem}
For any positive integer $k$, $i_1,\dots,i_k\in I$ and $\sigma\in S_k$, we have that
\begin{align}
  &Y_{i_{\sigma(1)},\dots,i_{\sigma(k)}}^\pm(z_{\sigma(1)},\dots,z_{\sigma(k)})
  =\(\prod_{\substack{1\le a<b\le k\\ \sigma(a)>\sigma(b)}}C_{i_a,i_b}\)
    Y_{i_1,\dots,i_k}^\pm(z_1,\dots,z_k),
\end{align}
where
\begin{align}\label{eq:def-C}
  C_{ij}=-(-1)^{\delta_{ij}}.
\end{align}
Moreover, $Y_{i_1,\dots,i_k}^\pm(z_1,\dots,z_k)\in\E_\hbar^{(k)}(F_\tau(A,\ell))$.
\end{lem}

\begin{lem}\label{lem:qSerre}
Let $i,j\in I$ with $a_{ij}<0$ and $k\in\N$. Then in $F_\tau(A,\ell)$, we have that
\begin{align}\label{eq:qSerre-1}
&\Sing_{z_1,\dots,z_k}Y_\tau(x_{i,\hbar}^\pm,z_1)\cdots Y_\tau(x_{i,\hbar}^\pm,z_k)x_{j,\hbar}^\pm\\
&\quad\in \C[z_1\inv,\dots,z_k\inv][[\hbar]] \(\(x_{i,\hbar}^\pm\)_0^kx_{j,\hbar}^\pm\).\nonumber
\end{align}
Moreover, we define $\bar Y_{ij,k}^\pm(z)$ to be
\begin{align*}
  Y_{i,\dots,i,j}^\pm\(z+r_i\big((k-1)a_{ii}+a_{ij}\big)\hbar,z+r_i\big((k-2)a_{ii}+a_{ij}\big)\hbar, \dots,z+r_ia_{ij}\hbar,z\).
\end{align*}
Then we have that
\begin{align}\label{eq:qSerre-2}
  Y_\tau\(\(x_{i,\hbar}^\pm\)_0^kx_{j,\hbar}^\pm,z\)=c_k\bar Y_{ij,k}^\pm(z)
  \quad\te{for some }c_k\in\C[[\hbar]]^\times.
\end{align}
\end{lem}

\begin{proof}
We prove the lemma by using induction on $k$. From equations \eqref{eq:local-h-3} and \eqref{eq:local-h-4}, we have that
\begin{align*}
  &\(\prod_{1\le a\le k}f_{ii,\hbar}\(z_1-z_2-r_i\big((a-1)a_{ii}+a_{ij}\big)\hbar\)\)
  f_{ij,\hbar}(z_1-z_2)Y_\tau\(x_{i,\hbar}^\pm,z_1\)\bar Y_{ij,k}^\pm(z_2)\\
  =&Y_{i,\dots,i,j}^\pm\(z_1,z_2+r_i((k-1)a_{ii}+a_{ij})\hbar,z_2+r_i((k-2)a_{ii}+a_{ij})\hbar,\dots,z_2+r_ia_{ij}\hbar,z_2\)
\end{align*}
lies in $\E_\hbar^{(2)}\(F_\tau(A,\ell)\)$. Notice that
\begin{align*}
  &f_{ij,\hbar}(z_1-z_2)\prod_{1\le a\le k}f_{ii,\hbar}\(z_1-z_2-r_i\big((a-1)a_{ii}+a_{ij}\big)\hbar\)\\
  &\quad=q_i^{-k-a_{ij}/2}e^{(z_1-z_2)/2}-q_i^{k+a_{ij}/2}e^{-(z_1-z_2)/2}.
\end{align*}
and that $\(q_i^{-k-a_{ij}/2}e^{(z_1-z_2)/2}-q_i^{k+a_{ij}/2}e^{-(z_1-z_2)/2}\)/(z_1-z_2-r_i(ka_{ii}+a_{ij})\hbar)\in \C[[z_1,z_2]]^\times$.
Then
\begin{align*}
  &\(z_1-z_2-r_i(ka_{ii}+a_{ij}\hbar)\)Y_\tau\(x_{i,\hbar}^\pm,z_1\)\bar Y_{ij,k}^\pm(z_2)\\
  =&d_k(z_1-z_2)Y_{i,\dots,i,j}^\pm\big(z_1,z_2+r_i((k-1)a_{ii}+a_{ij})\hbar,\\
  &\qquad z_2+r_i((k-2)a_{ii}+a_{ij})\hbar,\dots,z_2+r_ia_{ij}\hbar,z_2\big)
\end{align*}
lies in $\E_\hbar^{(2)}(F_\tau(A,\ell))$ for some $d_k(z)\in\C[[z,\hbar]]^\times$.
From induction assumptions \eqref{eq:qSerre-1}, \eqref{eq:qSerre-2} and Lemmas \ref{lem:tech-2}, \ref{lem:tech-3}, we complete the proof of \eqref{eq:qSerre-1} for $k+1$.

From induction assumption \eqref{eq:qSerre-2} and \eqref{eq:weak-asso}, we have that
\begin{align*}
  &(z-r_i(ka_{ii}+a_{ij})\hbar)Y\(Y\(x_{i,\hbar}^\pm,z\)\(\(x_{i,\hbar}^\pm\)_0^kx_{j,\hbar}^\pm\),z_2\)\\
  =&c_k(z-r_i(ka_{ii}+a_{ij})\hbar)Y_\E\(Y(x_{i,\hbar}^\pm,z_2),z\)\bar Y_{ij,k}^\pm(z_2)\\
  =&c_k
  \left.\(\(z_1-z-r_i(ka_{ii}+a_{ij}\hbar)\)Y_\tau\(x_{i,\hbar}^\pm,z_1\)\bar Y_{ij,k}^\pm(z_2)\)\right|_{z_1=z_2+z}\\
  =&c_kd_k(z)
  Y_{i,\dots,i,j}^\pm\big(z_2+z,z_2+r_i((k-1)a_{ii}+a_{ij})\hbar,\\
  &\qquad z_2+r_i((k-2)a_{ii}+a_{ij})\hbar,\dots,z_2+r_ia_{ij}\hbar,z_2\big).
\end{align*}
By multiplying $(z-r_i(ka_{ii}+a_{ij})\hbar)\inv$ and taking $\Sing_z$ on both hand sides, we get that
\begin{align}\label{eq:qSerre-temp}
  &Y\(\Sing_zY\(x_{i,\hbar}^\pm,z\)\(x_{i,\hbar}^\pm\)_0^kx_{j,\hbar}^\pm,z_2\)\\
  =&(z-r_i(ka_{ii}+a_{ij})\hbar)\inv
  c_{k+1}\bar Y_{ij,k+1}^\pm(z_2),\nonumber
\end{align}
where $c_{k+1}=c_kd_k(r_i(ka_{ii}+a_{ij})\hbar)$.
Then from the induction assumption \eqref{eq:qSerre-1}, we complete the proof of \eqref{eq:qSerre-1} for $k+1$.
Notice that both $c_k$ and $d_k(r_i(ka_{ii}+a_{ij})\hbar)$ are invertible in $\C[[\hbar]]$. We have that $c_{k+1}$ is invertible.
By taking $\Res_z$ on both hand sides of \eqref{eq:qSerre-temp}, we complete the proof of \eqref{eq:qSerre-2} for $k+1$.
Therefore, we complete the proof of the lemma.
\end{proof}

Similarly, we have that
\begin{lem}\label{lem:Mi}
Let $i\in I$ and $k\in\N$. Then in $F_\tau(A,\ell)$, we have that
\begin{align}
  &\Sing_{z_1,\dots,z_k}z_1\inv\cdots z_k\inv Y_\tau(x_{i,\hbar}^\pm,z_1)\cdots Y_\tau(x_{i,\hbar}^\pm,z_k)x_{i,\hbar}^\pm\\
  &\quad \in \C[z_1\inv,\dots,z_k\inv][[\hbar]]\(\(x_{i,\hbar}^\pm\)_{-1}^kx_{i,\hbar}^\pm\).\nonumber
\end{align}
Moreover, there is $c_k'\in\C[[\hbar]]^\times$, such that
\begin{align*}
  Y_\tau\(\(x_{i,\hbar}^\pm\)_{-1}^kx_{i,\hbar}^\pm,z\)=c_k'Y_{i,\dots,i}^\pm\( z+kr_ia_{ii}\hbar,z+(k-1)r_ia_{ii}\hbar,\dots,z \).
\end{align*}
\end{lem}

It is immediate from Lemma \ref{lem:qSerre} that
\begin{prop}\label{prop:qSerre}
For $i,j\in I$ with $a_{ij}<0$, we let
\begin{align*}
  &Q_{ij}^\pm(z_1,\dots,z_{m_{ij}})\\
  =&\Sing_{z_1,z_2,\dots,z_{m_{ij}}}Y_\tau(x_{i,\hbar}^\pm,z_1)Y_\tau(x_{i,\hbar}^\pm,z_2)\cdots Y_\tau(x_{i,\hbar}^\pm,z_{m_{ij}})x_{j,\hbar}^\pm.
\end{align*}
Then $Q_{ij}^\pm(z_1,\dots,z_{m_{ij}})=0$ in $V_{\hat\fg,\hbar}(\ell,0)$.
\end{prop}

And it is immediate from Lemma \ref{lem:Mi} that
\begin{prop}\label{prop:Mi}
For $\ell\in\Z_+$, we let
\begin{align*}
  &M_i^\pm(z_1,\dots,z_{r\ell/r_i})\\
  =&\Sing_{z_1,z_2,\dots,z_{r\ell/r_i}}z_1\inv \cdots z_{r\ell/r_i}\inv
  Y_\tau(x_{i,\hbar}^\pm,z_1)Y_\tau(x_{i,\hbar}^\pm,z_2)\cdots Y_\tau(x_{i,\hbar}^\pm,z_{r\ell/r_i})x_{i,\hbar}^\pm.
\end{align*}
Then $M_i^\pm(z_1,\dots,z_{r\ell/r_i})=0$ in $L_{\hat\fg,\hbar}(\ell,0)$.
\end{prop}

Combining Propositions \ref{prop:qCartan}, \ref{prop:qSerre} and \ref{prop:Mi}, we have that
\begin{prop}\label{prop:ideal-def-alt}
$R_1^\ell$ is the minimal closed ideal of $F_\tau(A,\ell)$ such that $[R_1^\ell]=R_1^\ell$, and contains all coefficients of $A_i(z)$ for $i\in I$ and all coefficients of $Q_{ij}^\pm(z_1,\dots,z_{m_{ij}})$ for $i,j\in I$ with $a_{ij}<0$.
Moreover, if $\ell\in\Z_+$, then $R_2^\ell$ is the minimal closed ideal of $V_{\hat\fg,\hbar}(\ell,0)$ such that $[R_2^\ell]=R_2^\ell$ and contains all coefficients of $M_i^\pm(z_1,\dots,z_{r\ell/r_i})$ for $i\in I$.
\end{prop}

Theorem \ref{thm:quotient-algs} is immediate from Lemma \ref{lem:S-quotient-alg}, Proposition \ref{prop:ideal-def-alt} and the following three technical results.

\begin{lem}\label{lem:S-tau-Aij-h}
For $i,j\in I$, we have that
\begin{align*}
  &S_\tau(z_1)(A_i(z_2)\ot h_{j,\hbar})\\
  =&A_i(z_2)\ot h_{j,\hbar}+\Sing_{z_2}\(A_i(z_2)\ot \vac\ot\(e^{z_2\pd{z_1}}-1\)\(\tau_{ij}^{1,+}(-z_1)+\tau_{ji}^{2,+}(z_1)\) \),\\
  &S_\tau(z_1)(h_{j,\hbar}\ot A_i(z_2))\\
  =&h_{j,\hbar}\ot A_i(z_2)+\Sing_{z_2}\(\vac\ot A_i(z_2)\ot\(1-e^{-z_2\pd{z_1}}\)\(\tau_{ji}^{2,+}(-z_1)+\tau_{ij}^{1,+}(z_1)\)\),\\
  &S_\tau(z_1)(A_i(z_2)\ot x_{j,\hbar}^\pm)\\
  =&\Sing_{z_2}\(A_i(z_2)\ot x_{j,\hbar}^\pm\ot
  \exp\(\(e^{z_2\pd{z_1}}-1\) \log\frac{\tau_{ij}^{+,\pm}(z_1)}{\tau_{ji}^{\pm,+}(-z_1)} \)\),\\
  &S_\tau(z_1)(x_{j,\hbar}^\pm\ot A_i(z_2))\\
  =&\Sing_{z_2}\(x_{j,\hbar}^\pm\ot A_i(z_2)\ot \exp\(\(1-e^{-z_2\pd{z_1}}\) \log\frac{\tau_{ij}^{+,\pm}(-z_1)}{\tau_{ji}^{\pm,+}(z_1)} \)\).
\end{align*}
\end{lem}

\begin{proof}
From the relations \eqref{eq:qyb-hex-id}, \eqref{eq:S-tau-2} and \eqref{eq:special-tau-tech0-2}, we have that
\begin{align*}
  &S_\tau(z_1)\(\Sing_{z_2}Y_\tau(x_{i,\hbar}^+,z_2)x_{i,\hbar}^-\ot h_{j,\hbar}\)\\
  =&\Sing_{z_2}Y_\tau^{12}(z_2)S_\tau^{23}(z_1)S_\tau^{13}(z_1+z_2)(x_{i,\hbar}^+\ot x_{i,\hbar}^-\ot h_{j,\hbar})\\
  =&\Sing_{z_2}Y_\tau^{12}(z_2)S_\tau^{23}(z_1)\Big(x_{i,\hbar}^+\ot x_{i,\hbar}^-\ot h_{j,\hbar}\\
  &\quad+x_{i,\hbar}^+\ot x_{i,\hbar}^-\ot\vac
  \ot\ot \(\tau_{ji}^{1,+}(-z_1-z_2)+\tau_{ij}^{2,+}(z_1+z_2)\)\Big)\\
  =&\Sing_{z_2}Y_\tau^{12}(z_2)\Big(
    x_{i,\hbar}^+\ot x_{i,\hbar}^-\ot h_{j,\hbar}\\
  &\quad+x_{i,\hbar}^+\ot x_{i,\hbar}^-\ot\vac\ot\(e^{z_2\pd{z_1}}-1\)\(\tau_{ji}^{1,+}(-z_1)+\tau_{ij}^{2,+}(z_1)\)
  \Big)\\
  =&\Sing_{z_2}Y_\tau(x_{i,\hbar}^+,z_2)x_{i,\hbar}^-\ot h_{j,\hbar}+\Sing_{z_2,z}\\
  &\quad\(Y_\tau(x_{i,\hbar}^+,z_2)x_{i,\hbar}^-\ot \vac\ot \(e^{z_2\pd{z_1}}-1\)\(\tau_{ji}^{1,+}(-z_1)+\tau_{ij}^{2,+}(z_1)\)\).
\end{align*}
From the relations \eqref{eq:qyb-shift-total1}, \eqref{eq:S-tau-1} and \eqref{eq:special-tau-tech1-3}, we have that
\begin{align*}
  &S_\tau(z)\(q^{-r\ell\partial}F(\partial)h_{i,\hbar}\ot h_{j,\hbar}\)\\
  =&q^{-r\ell\partial\ot 1-r\ell\pd{z}}F\(\partial\ot 1+\pd{z}\)S(z)(h_{i,\hbar}\ot h_{j,\hbar})\\
  =&q^{-r\ell\partial\ot 1-r\ell\pd{z}}F\(\partial\ot 1+\pd{z}\)\(h_{i,\hbar}\ot h_{j,\hbar}+\vac\ot \vac\ot\(\tau_{ji}(-z)-\tau_{ij}(z)\)\)\\
  =&q^{-r\ell\partial}F(\partial)h_{i,\hbar}\ot h_{j,\hbar}+\vac\ot\vac\ot q^{-r\ell\pd{z}}F\(\pd{z}\)\(\tau_{ji}(-z)-\tau_{ij}(z)\)\\
  =&q^{-r\ell\partial}F(\partial)h_{i,\hbar}\ot h_{j,\hbar}+\vac\ot\vac\ot \(1-q^{-2r\ell\pd{z}}\)\(\tau_{ji}^{1,+}(-z)+\tau_{ij}^{2,+}(z)\).
\end{align*}
Then by using Lemma \ref{lem:S-special-tech-gen2}, we get that
\begin{align*}
  &S_\tau(z)\(\exp\(\(-q^{-r\ell\partial}F(\partial)h_{i,\hbar}\)_{-1}\)\vac\ot h_{j,\hbar}\)\\
  =&\exp\(\(-q^{-r\ell\partial}F(\partial)h_{i,\hbar}\)_{-1}\)\vac\ot h_{j,\hbar}\\
  & +\exp\(\(-q^{-r\ell\partial}F(\partial)h_{i,\hbar}\)_{-1}\)\vac\ot \vac\ot
  \(q^{-2r\ell\pd{z}}-1\)\(\tau_{ji}^{1,+}(-z)+\tau_{ij}^{2,+}(z)\).
\end{align*}
Hence, we have that
\begin{align}\label{eq:Aij-h-temp3}
  &S_\tau(z_1)\(A_i(z_2)\ot h_{j,\hbar}\)=\Sing_{z_2}Y_\tau(x_{i,\hbar}^+,z_2)x_{i,\hbar}^-\ot h_{j,\hbar}+\Sing_{z_2}\nonumber\\
  &\quad\(\Sing_{z_2}Y_\tau(x_{i,\hbar}^+,z_2)x_{i,\hbar}^-\ot \vac\ot \(e^{z_2\pd{z_1}}-1\)\big(\tau_{ji}^{1,+}(-z_1)+\tau_{ij}^{2,+}(z_1)\big)\)\nonumber\\
  &\quad-\frac{1}{q_i-q_i\inv}\(\vac\ot h_{j,\hbar}z_2\inv-E(h_{i,\hbar})\ot h_{j,\hbar}(z_2+2r\ell\hbar)\inv\)\nonumber\\
  &\quad+\frac{(z_2+2r\ell\hbar)\inv}{q_i-q_i\inv}E(h_{i,\hbar})\ot \vac\ot
  \(q^{-2r\ell\pd{z_1}}-1\)\big(\tau_{ji}^{1,+}(-z_1)+\tau_{ij}^{2,+}(z_1)\big)\nonumber\\
  =&A_i(z_2)\ot h_{j,\hbar}+\Sing_{z_2}A_i(z_2)\ot \vac\ot \(e^{z_2\pd{z_1}}-1\)\big(\tau_{ji}^{1,+}(-z_1)+\tau_{ij}^{2,+}(z_1)\big)\\
  &\quad+\frac{1}{q_i-q_i\inv}\Sing_{z_2}\(\vac\ot\vac \ot z_2\inv\(e^{z_2\pd{z_1}}-1\)\big(\tau_{ji}^{1,+}(-z_1)+\tau_{ij}^{2,+}(z_1)\big)\)\nonumber\\
  &\quad-\frac{1}{q_i-q_i\inv}\Sing_{z_2}\(E(h_{i,\hbar})\ot \vac
  \ot\frac{e^{z_2\pd{z_1}}-1}{z_2+2r\ell\hbar}\big(\tau_{ji}^{1,+}(-z_1)+\tau_{ij}^{2,+}(z_1)\big)\)\nonumber\\
  &\quad+\frac{(z_2+2r\ell\hbar)\inv}{q_i-q_i\inv}E(h_{i,\hbar})\ot \vac\ot
  \(q^{-2r\ell\pd{z_1}}-1\)\big(\tau_{ji}^{1,+}(-z_1)+\tau_{ij}^{2,+}(z_1)\big).\nonumber
\end{align}
Notice that
\begin{align*}
  &\Sing_zz\inv\(e^{z\pd{x}}-1\)=0,\\
  &\Sing_z(z+2r\ell\hbar)\inv \(e^{z\pd{x}}-1\)=\(q^{-2r\ell\pd{x}}-1\)(z+2r\ell\hbar)\inv.
\end{align*}
Combining these equations and \eqref{eq:Aij-h-temp3}, we get that
\begin{align*}
  &S_\tau(z_1)\(A_i(z_2)\ot h_{j,\hbar}\)\\
  =&A_i(z_2)\ot h_{j,\hbar}+\Sing_{z_2} A_i(z_2)\ot \vac\ot \(e^{z_2\pd{z_1}}-1\)\(\tau_{ji}^{1,+}(-z_1)+\tau_{ij}^{2,+}(z_1)\).
\end{align*}
Therefore, we complete the proof of the first equation.
The proof of the rest equations are similar.
\end{proof}

Similar to the proof of Lemma \ref{lem:S-tau-Aij-h}, we have the following two results.

\begin{lem}\label{lem:S-tau-Qij-h}
For $i,j,k\in I$ such that $a_{ij}<0$, we have that
\begin{align*}
  &S_\tau(z)\(Q_{ij}^\pm(z_1,\dots,z_{m_{ij}})\ot h_{k,\hbar}\)\\
  =&Q_{ij}^\pm(z_1,\dots,z_{m_{ij}})\ot h_{k,\hbar}
  \pm \Sing_{z_1,\dots,z_{m_{ij}}}\Bigg(Q_{ij}^\pm(z_1,\dots,z_{m_{ij}})\ot\vac\\
  &\quad\ot
    \(\tau_{kj}^{1,\pm}(-z)+\tau_{jk}^{2,\pm}(z)+\sum_{a=1}^{m_{ij}}\(\tau_{ki}^{1,\pm}(-z-z_a)+\tau_{ik}^{2,\pm}(z+z_a)\)
  \)\Bigg),\\
  &S_\tau(z)\(h_{k,\hbar}\ot Q_{ij}^\pm(z_1,\dots,z_{m_{ij}})\)\\
  =&h_{k,\hbar}\ot Q_{ij}^\pm(z_1,\dots,z_{m_{ij}})
  \mp \Sing_{z_1,\dots,z_{m_{ij}}}\Bigg(\vac\ot Q_{ij}^\pm(z_1,\dots,z_{m_{ij}})\\
  &\quad\ot \(\tau_{jk}^{2,\pm}(-z)+\tau_{kj}^{1,\pm}(z)
  +\sum_{a=1}^{m_{ij}}\(\tau_{ik}^{2,\pm}(-z+z_a)+\tau_{ki}^{1,\pm}(z-z_a)\)\)\Bigg),\\
  &S_\tau(z)\(Q_{ij}^\pm(z_1,\dots,z_{m_{ij}})\ot x_{k,\hbar}^\epsilon\)
  =\Sing_{z_1,\dots,z_{m_{ij}}}\Bigg(Q_{ij}^\pm(z_1,\dots,z_{m_{ij}})\ot x_{k,\hbar}^\epsilon\\
  &\quad
  \ot \tau_{kj}^{\epsilon,\pm}(-z)\inv\tau_{jk}^{\pm,\epsilon}(z)
  \prod_{a=1}^{m_{ij}}\tau_{ki}^{\epsilon,\pm}(-z-z_a)\inv\tau_{ik}^{\pm,\epsilon}(z+z_a)\Bigg),\\
  &S_\tau(z)\(x_{k,\hbar}^\epsilon\ot Q_{ij}^\pm(z_1,\dots,z_{m_{ij}})\)=\Sing_{z_1,\dots,z_{m_{ij}}}\Bigg(x_{k,\hbar}^\epsilon\ot Q_{ij}^\pm(z_1,\dots,z_{m_{ij}})\\
  &\quad\ot \tau_{jk}^{\epsilon,\pm}(-z)\inv \tau_{kj}^{\pm,\epsilon}(z)
    \prod_{a=1}^{m_{ij}}\tau_{ik}^{\epsilon,\pm}(-z+z_a)\inv\tau_{ki}^{\pm,\epsilon}(z-z_a)\Bigg).
\end{align*}
\end{lem}

\begin{lem}\label{lem:S-tau-Mij-h}
For any $\ell\in \Z_+$ and $i,j\in I$, we have that
\begin{align*}
  &S_\tau(z)\(M_i^\pm(z_1,\dots,z_{r\ell/r_i})\ot h_{j,\hbar}\)
  =M_i^\pm(z_1,\dots,z_{r\ell/r_i})\ot h_{j,\hbar}\\
  &\quad\pm\Sing_{z_1,\dots,z_{r\ell/r_i}}z_1\inv\cdots z_{r\ell/r_i}\inv \\
  &\quad\Bigg(
  M_i^\pm(z_1,\dots,z_{r\ell/r_i})\ot \vac\ot
    \sum_{a=1}^{r\ell/r_i+1}\(\tau_{ji}^{1,\pm}(-z-z_a)+\tau_{ij}^{2,\pm}(z+z_a)\)
  \Bigg),\\
  &S_\tau(z)\(h_{j,\hbar}\ot M_i^\pm(z_1,\dots,z_{r\ell/r_i})\)
  =h_{j,\hbar}\ot M_i^\pm(z_1,\dots,z_{r\ell/r_i})\\
  &\quad\mp\Sing_{z_1,\dots,z_{r\ell/r_i}}z_1\inv \cdots z_{r\ell/r_i}\inv\\
  &\quad\(
    \vac\ot M_i^\pm(z_1,\dots,z_{r\ell/r_i})\ot
    \sum_{a=1}^{r\ell/r_i+1}\(\tau_{ij}^{1,\pm}(-z+z_a)+\tau_{ji}^{2,\pm}(z-z_a)\)
  \),\\
  &S_\tau(z)\(M_i^\pm(z_1,\dots,z_{r\ell/r_i})\ot x_{j,\hbar}^\epsilon\)
  =\Sing_{z_1,\dots,z_{r\ell/r_i}}z_1\inv\cdots z_{r\ell/r_i}\inv\\
  &\quad\(
    M_i^\pm(z_1,\dots,z_{r\ell/r_i})\ot x_{j,\hbar}^\epsilon\ot
    \prod_{a=1}^{r\ell/r_i+1}\tau_{ji}^{\epsilon,\pm}(-z-z_a)\inv \tau_{ij}^{\pm,\epsilon}(z+z_a)
  \),\\
  &S_\tau(z)\(x_{j,\hbar}^\epsilon\ot M_i^\pm(z_1,\dots,z_{r\ell/r_i})\)
  =\Sing_{z_1,\dots,z_{r\ell/r_i}}z_1\inv\cdots z_{r\ell/r_i}\inv\\
  &\quad\(
    x_{j,\hbar}^\epsilon\ot M_i^\pm(z_1,\dots,z_{r\ell/r_i})\ot
    \prod_{a=1}^{r\ell/r_i+1}\tau_{ij}^{\epsilon,\pm}(-z+z_a)\inv\tau_{ji}^{\pm,\epsilon}(z-z_a)
  \),
\end{align*}
where $z_{r\ell/r_i+1}=0$.
\end{lem}

\section{$\phi$-coordinated modules}\label{sec:phi-mods}

In this section, we recall the construction of $\hbar$-adic nonlocal vertex algebras and their $\phi$-coordinated modules introduced in \cite{Li-phi-coor}.

First, we fix an associate $\phi(z_1,z)=z_1e^z$, which is a particular associate of
the additive formal group $F_a(z_1,z_2)=z_1+z_2$.
Let $V$ be a nonlocal vertex algebra.
Recall from \cite{Li-phi-coor} that a \emph{$\phi$-coordinated $V$-module} is a vector space $W$ equipped with a linear map $Y_W^\phi(\cdot,z):V\to \E(W)$, satisfying the condition that $Y_W^\phi(\vac,z)=1_W$
and that for $u,v\in V$, there exists positive integer $k$ such that
\begin{align}
  &(1-z_2/z_1)^kY_W^\phi(u,z_1)Y_W^\phi(v,z_2)\in \E^{(2)}(W),\label{eq:phi-local}\\
  &\left.\((1-z_2/z_1)^kY_W^\phi(u,z_1)Y_W^\phi(v,z_2)\)\right|_{z_1=z_2e^{z_0}}\label{eq:phi-asso}\\
  &\quad=(1-e^{-z_0})^kY_W^\phi(Y(u,z_0)v,z_2).\nonumber
\end{align}
Then we have that

\begin{lem}\label{lem:valocal}
Let $V$ be a nonlocal vertex algebra, and let $(W,Y_W^\phi)$ be a $\phi$-coordinated $V$-module. Suppose
\begin{align*}
  \sum_{i\ge 1}a_i\ot b_i\ot f_i(z),\quad\sum_{j\ge 1}\al_j\ot \beta_j\ot g_j(z)\in V\ot V\ot \C(z),
\end{align*}
such that
\begin{align}\label{eq:valocal}
  &\sum_{i\ge 1}\iota_{z_1,z_2}f_i(e^{z_1-z_2})Y(a_i,z_1)Y(b_i,z_2)\\
  =&\sum_{j\ge 1}\iota_{z_2,z_1}g_j(e^{z_1-z_2})Y(\al_j,z_2)Y(\beta_j,z_1).\nonumber
\end{align}
Then we have that
\begin{align*}
  &\sum_{i\ge 1}\iota_{z_1,z_2}f_i(z_1/z_2)Y_W^\phi(a_i,z_1)Y_W^\phi(b_i,z_2)\\
  =&\sum_{j\ge 1}\iota_{z_2,z_1}g_j(z_1/z_2)Y_W^\phi(\al_j,z_2)Y_W^\phi(\beta_j,z_1).
\end{align*}
\end{lem}

\begin{proof}
By letting the both hand sides of \eqref{eq:valocal} act on $\vac$, we get that
\begin{align}\label{eq:valocal-temp1}
  &\sum_{i\ge 1}\iota_{z_1,z_2}f_i(e^{z_1-z_2})Y(a_i,z_1-z_2)b_i\\
  =&\sum_{j\ge 1}\iota_{z_2,z_1}g_j(e^{z_1-z_2})e^{(z_1-z_2)\partial}Y(\al_j,z_2-z_1)\beta_j.\nonumber
\end{align}
Let $k$ be a positive integer such that ($i,j\ge 1$)
\begin{align}
  &z^{k}f_i(e^z),\quad z^{k}g_j(e^z)\in \C[[z]], \quad
  z^{k}Y(a_i,z)b_i,\quad z^{k}Y(\al_j,z)\beta_j\in V[[z]].\label{eq:valocal-temp2-3}
\end{align}
By multiplying $(z_1-z_2)^{2k}$ on both hand sides of \eqref{eq:valocal-temp1}, we get that
\begin{align}\label{eq:valocal-temp4}
  &\sum_{i\ge 1}\((z_1-z_2)^{k}f_i(e^{z_1-z_2})\)\((z_1-z_2)^{k}Y(a_i,z_1-z_2)b_i\)\nonumber\\
  =&\sum_{j\ge 1}\((z_1-z_2)^{k}g_j(e^{z_1-z_2})\)e^{(z_1-z_2)\partial}\((z_1-z_2)^{k}Y(\al_j,z_2-z_1)\beta_j\).
\end{align}
From \eqref{eq:valocal-temp2-3}, we can take $z_2=0$ in \eqref{eq:valocal-temp4}:
\begin{align*}
  &z^{2k}\sum_{i\ge 1}f_i(e^z)Y(a_i,z)b_i
  =\sum_{j\ge 1}z^{k}g_j(e^z)e^{z\partial}z^{k}Y(\al_j,-z)\beta_j.
\end{align*}
Since $z^{k}$ is invertible, we get that
\begin{align}\label{eq:valocal-temp5}
  \sum_{i\ge 1}f_i(e^z)Y(a_i,z)b_i
  =\sum_{j\ge 1}g_j(e^z)e^{z\partial}Y(\al_j,-z)\beta_j.
\end{align}

From \eqref{eq:valocal-temp2-3}, we get that
\begin{align*}
  &\left.\((z_1-z_2)^{2k}\sum_{i\ge 1}f_i(e^{z_1-z_2})Y(a_i,z_1)Y(b_i,z_2)\)\right|_{z_1=z_2+z_0}\\
  =&\sum_{i\ge 1}z_0^kf_i(e^{z_0})z_0^kY_\E\(Y(a_i,z_2),z_0\)Y(b_i,z_2)\\
  =&\sum_{i\ge 1}z_0^kf_i(e^{z_0})z_0^kY\(Y(a_i,z_0)b_i,z_2\),
\end{align*}
where the last equation follows from \eqref{eq:weak-asso}.
We note that the condition \eqref{eq:valocal} shows that
\begin{align*}
  \sum_{i\ge 1}f_i(e^{z_1-z_2})Y(a_i,z_1)Y(b_i,z_2)\in\E^{(2)}(V).
\end{align*}
Combining these relations, we get that
\begin{align*}
  &z_0^{2k}\left.\(\sum_{i\ge 1}f_i(e^{z_1-z_2})Y(a_i,z_1)Y(b_i,z_2)\)\right|_{z_1=z_2+z_0}\\
  =&\sum_{i\ge 1}z_0^kf_i(e^{z_0})z_0^kY\(Y(a_i,z_0)b_i,z_2\).
\end{align*}
Then we have that
\begin{align*}
  &\sum_{i\ge 1}f_i(e^{z_0})Y\(Y(a_i,z_0)b_i,z_2\)\\
  =&\left.\(\sum_{i\ge 1}f_i(e^{z_1-z_2})Y(a_i,z_1)Y(b_i,z_2)\)\right|_{z_1=z_2+z_0}
  \in\E(V)[[z_0]].
\end{align*}
Acting on $\vac$ and taking $z_2\to 0$, we get that
\begin{align*}
  \sum_{i\ge 1}f_i(e^{z_0})Y(a_i,z_0)b_i\in V[[z_0]].
\end{align*}
Viewing $\log(1+z_0)$ as an element in $\C[[z_0]]$, we get that
\begin{align}\label{eq:valocal-temp5.5}
  \sum_{i\ge 1}f_i(1+z_0)Y(a_i,\log(1+z_0))b_i\in V[[z_0]].
\end{align}

We replace $k$ with a larger one if necessary so that ($i,j\ge 1$):
\begin{align*}
  (1-z_2/z_1)^{k}Y_W^\phi(a_i,z_1)Y_W^\phi(b_i,z_2),\,\,
  (1-z_2/z_1)^{k}Y_W^\phi(\al_j,z_2)Y_W^\phi(\beta_j,z_1)
  \in \E^{(2)}(W).
\end{align*}
From the weak associativity of $\phi$-coordinated modules \eqref{eq:phi-asso}, we get that
\begin{align}
  &\left.\((1-z/z_1)^{k}Y_W^\phi(a_i,z_1)Y_W^\phi(b_i,z)\)\right|_{z_1=\phi(z,z_0)}\label{eq:valocal-temp6a}\\
  =&(1-e^{-z_0})^{k}Y_W^\phi(Y(a_i,z_0)b_i,z),\nonumber\\
  &\left.\((1-z/z_1)^{k}Y_W^\phi(\al_j,z)Y_W^\phi(\beta_j,z_1)\)\right|_{z=\phi(z_1,-z_0)}\label{eq:valocal-temp6b}\\
  =&(1-e^{-z_0})^{k}Y_W^\phi(Y(\al_j,-z_0)\beta_j,z_1).\nonumber
\end{align}
From the equation \eqref{eq:valocal-temp6b} and \cite[Remark 2.8]{Li-phi-coor}, we have that
\begin{align}
  &\left.\((1-z/z_1)^{k}Y_W^\phi(\al_j,z)Y_W^\phi(\beta_j,z_1)\)\right|_{z_1=\phi(z,z_0)}\nonumber\\
  =&\left.\(\left.\((1-z/z_1)^{k}Y_W^\phi(\al_j,z)Y_W^\phi(\beta_j,z_1)\)\right|_{z=\phi(z_1,-z_0)}\)
  \right|_{z_1=\phi(z,z_0)}\nonumber\\
  =&\left.\((1-e^{-z_0})^{k}Y_W^\phi(Y(\al_j,-z_0)\beta_j,z_1)\)
  \right|_{z_1=\phi(z,z_0)}\nonumber\\
  =&(1-e^{-z_0})^{k}Y_W^\phi(Y(\al_j,-z_0)\beta_j,\phi(z,z_0))\nonumber\\
  =&(1-e^{-z_0})^{k}Y_W^\phi(e^{z_0\partial}Y(\al_j,-z_0)\beta_j,z),\label{eq:valocal-temp6}
\end{align}
where the \eqref{eq:valocal-temp6} follows from \cite[Lemma 3.7]{Li-phi-coor}.
Combining this with equations \eqref{eq:valocal-temp5} and \eqref{eq:valocal-temp6a}, we get
\begin{align}
  &\left.\((1-z/z_1)^{2k}\sum_{i\ge 1}f_i(z_1/z)Y_W^\phi(a_i,z_1)Y_W^\phi(b_i,z)\)\right|_{z_1=\phi(z,z_0)}\nonumber\\
  =&\sum_{i\ge 1}\left.\(\((1-z/z_1)^{k}f_i(z_1/z)\)\((1-z/z_1)^{k}Y_W^\phi(a_i,z_1)Y_W^\phi(b_i,z)\)\)\right|_{z_1=\phi(z,z_0)}\nonumber\\
  =&(1-e^{-z_0})^{2k}\sum_{i\ge 1}f_i(e^{z_0})Y_W^\phi(Y(a_i,z_0)b_i,z)\label{eq:valocal-temp6.5}\\
  =&(1-e^{-z_0})^{2k}\sum_{j\ge 1}g_j(e^{z_0})Y_W^\phi(e^{z_0\partial}Y(\al_j,-z_0)\beta_j,z)\nonumber\\
  =&\sum_{j\ge 1}\left.\(\((1-z/z_1)^{k}g_j(z_1/z)\)\((1-z/z_1)^{k}Y_W^\phi(\al_j,z)Y_W^\phi(\beta_j,z_1)\)\)\right|_{z_1=\phi(z,z_0)}\nonumber\\
  =&\left.\((1-z/z_1)^{2k}\sum_{j\ge 1}g_j(z_1/z)Y_W^\phi(\al_j,z)Y_W^\phi(\beta_j,z_1)\)\right|_{z_1=\phi(z,z_0)}.\nonumber
\end{align}
Notice that
\begin{align*}
  &(1-z_2/z_1)^{2k}\sum_{i\ge 1}f_i(z_1/z_2)Y_W^\phi(a_i,z_1)Y_W^\phi(b_i,z_2)\in \E^{(2)}(W),\\
  &(1-z_2/z_1)^{2k}\sum_{j\ge 1}g_j(z_1/z_2)Y_W^\phi(\al_j,z_2)Y_W^\phi(\beta_j,z_1)
  \in \E^{(2)}(W).
\end{align*}
Then we get from \cite[Remark 2.8]{Li-phi-coor} that
\begin{align*}
  &(1-z_2/z_1)^{2k}\sum_{i\ge 1}f_i(z_1/z_2)Y_W^\phi(a_i,z_1)Y_W^\phi(b_i,z_2)\\
  =&
  (1-z_2/z_1)^{2k}\sum_{j\ge 1}g_j(z_1/z_2)Y_W^\phi(\al_j,z_2)Y_W^\phi(\beta_j,z_1).
\end{align*}
Combining this with equation \eqref{eq:valocal-temp6.5} and \cite[Lemma 5.8]{Li-phi-coor}, we get that
\begin{align*}
  &(z_2z_0)\inv\delta\(\frac{z_1-z_2}{z_2 z_0}\)\sum_{i\ge 1}f_i(z_1/z_2)Y_W^\phi(a_i,z_1)Y_W^\phi(b_i,z_2)\\
  -&(z_2z_0)\inv\delta\(\frac{z_2-z_1}{-z_2 z_0}\)\sum_{j\ge 1}g_j(z_1/z_2)Y_W^\phi(\al_j,z_2)Y_W^\phi(\beta_j,z_1)\\
  =&z_1\inv\delta\(\frac{z_2(1+z_0)}{z_1}\)\sum_{i\ge 1}f_i(1+z_0)Y_W^\phi(Y(a_i,\log(1+z_0))b_i,z_2).
\end{align*}
Taking $\Res_{z_0}$ on both hand sides, we get that
\begin{align*}
  &\sum_{i\ge 1}f_i(z_1/z_2)Y_W^\phi(a_i,z_1)Y_W^\phi(b_i,z_2)
  -\sum_{j\ge 1}g_j(z_1/z_2)Y_W^\phi(\al_j,z_2)Y_W^\phi(\beta_j,z_1)\\
  =&\Res_{z_0} z_1\inv\delta\(\frac{z_2(1+z_0)}{z_1}\)\sum_{i\ge 1}f_i(1+z_0)Y_W^\phi(Y(a_i,\log(1+z_0))b_i,z_2)
  =0,
\end{align*}
where the last equation follows from \eqref{eq:valocal-temp5.5}.
We complete the proof.
\end{proof}

\begin{prop}\label{prop:vacom}
Let $(W,Y_W^\phi)$ be a $\phi$-coordinated module of a nonlocal vertex algebra $V$. Let
\begin{align*}
  &\sum_{i\ge 1}a_i\ot b_i\ot f_i(z),\quad\sum_{j\ge 1}\al_j\ot \beta_j\ot g_j(z)\in V\ot V\ot V\ot \C(x),\\
  &\te{and}\quad \sum_{k\ge 0}\gamma_k\in V
\end{align*}
be finite sums, such that
\begin{align}
  &\sum_{i\ge 1}\iota_{z_1,z_2}f_i(e^{z_1-z_2})Y(a_i,z_1)Y(b_i,z_2)\nonumber\\
  -&\sum_{j\ge 1}\iota_{z_2,z_1}g_j(e^{z_1-z_2})Y(\al_j,z_2)Y(\beta_j,z_1)\nonumber\\
  &\qquad=\sum_{k\ge 0}Y(\gamma_k,z_2)\frac{1}{k!}\pdiff{z_2}{k}z_1\inv\delta\(\frac{z_2}{z_1}\).\label{eq:vacom}
\end{align}
Then we have that
\begin{align}
  &\sum_{i\ge 1}\iota_{z_1,z_2}f_i(z_1/z_2)Y_W^\phi(a_i,z_1)Y_W^\phi(b_i,z_2)\nonumber\\
  -&\sum_{j\ge 1}\iota_{z_2,z_1}g_j(z_1/z_2)Y_W^\phi(\al_j,z_2)Y_W^\phi(\beta_j,z_1)\nonumber\\
  &\qquad=\sum_{k\ge 0}Y_W^\phi(\gamma_k,z_2)\frac{1}{k!}\(z_2\pd{z_2}\)^k\delta\(\frac{z_2}{z_1}\).\label{eq:phi-com}
\end{align}
\end{prop}

\begin{proof}
For any integer $i\ge 0$, we set $b_{-i}=\al_{-i}=\gamma_i$, $a_{-i}=\beta_{-i}=\vac$ and
\begin{align*}
  f_{-i}(z)=g_{-i}(z)=\frac{1}{2i!}\(-z\pd{z}\)^i\frac{z+1}{z-1}\in \C(z).
\end{align*}
It is straightforward to verify that
\begin{align*}
  &\iota_{z_1,z_2}f_{-i}(z_1/z_2)-\iota_{z_2,z_1}g_{-i}(z_1/z_2)
  =\frac{1}{i!}\(z_2\pd{z_2}\)^i\delta\(\frac{z_2}{z_1}\),\\
 &\iota_{z_1,z_2}f_{-i}(e^{z_1-z_2})-\iota_{z_2,z_1}g_{-i}(e^{z_1-z_2})
 =\frac{1}{i!}\pdiff{z_2}{i}z_1\inv\delta\(\frac{z_2}{z_1}\).
\end{align*}
Then the equation \eqref{eq:vacom} is equivalent to
\begin{align*}
  &\sum_{i\in\Z}\iota_{z_1,z_2}f_i(e^{z_1-z_2})Y(a_i,z_1)Y(b_i,z_2)\\
  =&\sum_{j\in\Z}\iota_{z_2,z_1}g_j(e^{z_1-z_2})Y(\al_j,z_2)Y(\beta_j,z_1),
\end{align*}
and the equation \eqref{eq:phi-com} is equivalent to
\begin{align*}
  &\sum_{i\in\Z}\iota_{z_1,z_2}f_i(z_1/z_2)Y_W^\phi(a_i,z_1)Y_W^\phi(b_i,z_2)\\
  =&\sum_{j\in\Z}\iota_{z_2,z_1}g_j(z_1/z_2)Y_W^\phi(\al_j,z_2)Y_W^\phi(\beta_j,z_1).
\end{align*}
Therefore, this proposition follows immediate from Lemma \ref{lem:valocal}.
\end{proof}

Let $V$ be an $\hbar$-adic nonlocal vertex algebra. A \emph{$\phi$-coordinated $V$-module} is a topologically free $\C[[\hbar]]$-module $W$ equipped with a $\C[[\hbar]]$-linear map $Y_W^\phi(\cdot,x):V\to \E_\hbar(W)$, such that $(W/\hbar^nW,Y_{W,n}^\phi)$ is a $\phi$-coordinated $V/\hbar^nV$-module, where $Y_{W,n}^\phi:V/\hbar^nV\to \E(W/\hbar^nW)$ is the $\C[[\hbar]]$-linear map induced from $Y_W^\phi$.
As an immediate $\hbar$-adic analogue of Proposition \ref{prop:vacom}, we have that

\begin{prop}\label{prop:vacom-h}
Let $V$ be an $\hbar$-adic nonlocal vertex algebra, and let $(W,Y_W^\phi)$ be a $\phi$-coordinated $V$-module. Suppose that
\begin{align*}
  &\sum_{i\ge 1}a_i\ot b_i\ot f_i(z),\quad\sum_{j\ge 1}\al_j\ot \beta_j\ot g_j(z)\in V\wh\ot V\wh\ot V\wh\ot \C(z)[[\hbar]],\\
  &\te{and}\quad
  \sum_{k\ge 0}\gamma_k\in V,
\end{align*}
such that
\begin{align*}
  &\sum_{i\ge 1}\iota_{z_1,z_2}f_i(e^{z_1-z_2})Y(a_i,z_1)Y(b_i,z_2)\\
  -&\sum_{j\ge 1}\iota_{z_2,z_1}g_j(e^{z_1-z_2})Y(\al_j,z_2)Y(\beta_j,z_1)\\
  &\qquad=\sum_{k\ge 0}Y(\gamma_k,z_2)\frac{1}{k!}\pdiff{z_2}{k}z_1\inv\delta\(\frac{z_2}{z_1}\).
\end{align*}
Then we have that
\begin{align*}
  &\sum_{i\ge 1}\iota_{z_1,z_2}f_i(z_1/z_2)Y_W^\phi(a_i,z_1)Y_W^\phi(b_i,z_2)\\
  -&\sum_{j\ge 1}\iota_{z_2,z_1}g_j(z_1/z_2)Y_W^\phi(\al_j,z_2)Y_W^\phi(\beta_j,z_1)\\
  &\qquad=\sum_{k\ge 0}Y_W^\phi(\gamma_k,z_2)\frac{1}{k!}\(z_2\pd{z_2}\)^k\delta\(\frac{z_2}{z_1}\).
\end{align*}
\end{prop}

Let $Z^\phi(z_1,z_2):\E_\hbar(W)\wh\ot \E_\hbar(W)\wh\ot \C(z)[[\hbar]]\to\End(W)[[z_1^{\pm 1},z_2^{\pm 1}]]$ be the $\C[[\hbar]]$-module map
defined by
\begin{align*}
  Z^\phi(z_1,z_2)(a(z)\ot b(z)\ot f(z))=f(z_1/z_2)a(z_1)b(z_2).
\end{align*}
Recall from \cite{Li-phi-coor} that a subset $U$ of $\E_\hbar(W)$ is said to be \emph{$\hbar$-adically $S_{trig}$-local} if for any $a(z),b(z)\in U$, there is
$A(z)\in\(\C U\ot \C U\ot \C(z)\)[[\hbar]]$, such that
\begin{align*}
  a(z_1)b(z_2)\sim Z^\phi(z_2,z_1)\(A(z)\).
\end{align*}

Let $U$ be an $\hbar$-adically $S_{trig}$-local subset of $\E_\hbar(W)$.
For $a(z),b(z)\in U$, the locality implies that for any positive integer $n$, there is positive integer $k_n$ such that
\begin{align*}
  (1-z_1/z_2)^{k_n}\pi_n(a(z_1))\pi_n(b(z_2))\in \E^{(2)}(W/\hbar^nW).
\end{align*}
The following is a partial $\hbar$-adic analogue of \cite[Definition 4.4]{Li-phi-coor}:
\begin{align*} 
  &Y_\E^\phi(a(z),z_0)b(z)=\sum_{n\in\Z}a(z)_n^\phi b(z)z_0^{-n-1}\\
  =&\varprojlim_{n>0} \(1-e^z_0\)^{-k_n}
    \left.\((1-z_1/z)^{k_n}a(z_1)b(z)\)\right|_{z_1=ze^{z_0}}.
\end{align*}
We have the partial $\hbar$-adic analogue of \cite[Theorem 4.8]{Li-phi-coor}:
\begin{thm} \label{thm:phi-construction}
Let $U$ be an $\hbar$-adic $S_{trig}$-local subset of $\E_\hbar(W)$. Then there is a minimal $\hbar$-adically $S_{trig}$-local subset $\<U\>_\phi\subset \E_\hbar(W)$ containing $U$ and $1_W$, such that
\begin{itemize}
  \item[(1)] $\<U\>_\phi$ is topologically free and $[\<U\>_\phi]=\<U\>_\phi$.
  \item[(2)] $\<U\>_\phi$ is $Y_\E^{\phi}$ closed, that is, for $a(x),b(x)\in \<U\>_\phi$, $a(x)_nb(x)\in \<U\>_\phi$.
\end{itemize}
Then $(\<U\>_\phi,Y_\E^{\phi},1_W)$ carries the structure of an $\hbar$-adic weak quantum vertex algebra and $W$ is a faithful $\phi$-coordinated $\<U\>_\phi$-module with $Y_W(a(z),z_0)=a(z_0)$ for $a(z)\in\<U\>_\phi$.
\end{thm}

The following result is a partial $\hbar$-adic analogue of \cite[Theorem 2.21]{JKLT-G-phi-mod}:
\begin{prop}\label{prop:YEcom-h}
Let $W$ be a topologically free $\C[[\hbar]]$-module, and let $V\subset \E_\hbar(W)$ be an $\hbar$-adic $S_{trig}$-local subset such that $V=\<V\>_\phi$.
Suppose that
\begin{align*}
  &\sum_{i\ge 1}a_i(z)\ot b_i(z)\ot f_i(z),\quad\sum_{j\ge 1}\al_j(z)\ot \beta_j(z)\ot g_j(z)\in V\wh\ot V\wh\ot \C(z)[[\hbar]],\\
  &\te{and}\quad
  \sum_{k\ge 1}\gamma_k(z)\in V,
\end{align*}
such that the following relation holds on $W$:
\begin{align*}
  &\sum_{i\ge 1}\iota_{z_1,z_2}f_i(z_1/z_2)a_i(z_1)b_i(z_2)
  -\sum_{j\ge 1}\iota_{z_2,z_1}g_j(z_1/z_2)\al_j(z_2)\beta_j(z_1)\\
  &\quad=\sum_{k\ge 1}\gamma_k(z_2)\frac{1}{k!}\(z_2\pd{z_2}\)^k\delta\(\frac{z_2}{z_1}\).
\end{align*}
Then we have that
\begin{align*}
  &\sum_{i\ge 1}\iota_{z_1,z_2}f_i(e^{z_1-z_2})Y_\E^\phi(a_i(z),z_1)Y_\E^\phi(b_i(z),z_2)\\
  -&\sum_{j\ge 1}\iota_{z_2,z_1}g_j(e^{z_1-z_2})Y_\E^\phi(\al_j(z),z_2)Y_\E^\phi(\beta_j(z),z_1)\\
  &\qquad=\sum_{k\ge 1}Y_\E^\phi(\gamma_k(z),z_2)\frac{1}{k!}\pdiff{z_2}{k}z_1\inv\delta\(\frac{z_2}{z_1}\).
\end{align*}
\end{prop}

Now we start to determine the category of $\phi$-coordinated $V_{\hat\fg,\hbar}(\ell,0)$-modules.

\begin{de}\label{de:cat-M-phi}
Let $\ell\in\C$. Define $\mathcal M_\ell^\phi$ to be the category consisting of topologically free $\C[[\hbar]]$-modules $W$, equipped with fields $\psi_{i,q}(z),y_{i,q}^\pm(z)\in\E_\hbar(W)$ that satisfy the following relations:
\begin{align}
  & [\psi_{i,q}(z_1),\psi_{j,q}(z_2)]\label{eq:cat-M-phi-1}\\
  =&[r_ia_{ij}]_{q^{z_2\pd{z_2}}}[r\ell]_{q^{z_2\pd{z_2}}}
  \(\iota_{z_1,z_2}q^{-r\ell z_2\pd{z_2}}-\iota_{z_2,z_1}q^{r\ell z_2\pd{z_2}}\)\frac{z_2/z_1}{(1-z_2/z_1)^2},\nonumber\\
  & [\psi_{i,q}(z_1),y_{j,q}^\pm(z_2)]\label{eq:cat-M-phi-2}\\
  =&\pm y_{j,q}^\pm(z_2)[r_ia_{ij}]_{q^{z_2\pd{z_2}}}
  \(\iota_{z_1,z_2}q^{-r\ell z_2\pd{z_2}}-\iota_{z_2,z_1}q^{r\ell z_2\pd{z_2}}\)\frac{1+z_2/z_1}{2-2z_2/z_1},\nonumber\\
  &(1-z_2/z_1)^{\delta_{ij}}(1-q^{-2r\ell}z_2/z_1)^{\delta_{ij}}\label{eq:cat-M-phi-3}\\
  &\quad\times\(y_{i,q}^+(z_1)y_{j,q}^-(z_2)
    -\iota_{z_2,z_1}g_{ij,q}(z_1/z_2) y_{j,q}^-(z_2)y_{i,q}^+(z_1)\)=0,\nonumber\\
  & \iota_{z_1,z_2}f_{ij,q}(z_1,z_2)y_{i,q}^\pm(z_1)y_{j,q}^\pm(z_2)
    =C_{ij}\iota_{z_2,z_1}f_{ji,q}(z_2,z_1)y_{j,q}^\pm(z_2)y_{i,q}^\pm(z_1),\label{eq:cat-M-phi-4}
\end{align}
where
\begin{align}\label{eq:def-fq}
  f_{ij,q}(z_1,z_2)=\iota_{z_1,z_2}(z_1-q_i^{ a_{ij}}z_2)(z_1-z_2)^{-\delta_{ij}},\quad g_{ij,q}(z)=\frac{q_i^{a_{ij}}-z}{1-q_i^{a_{ij}}z}.
\end{align}
\end{de}

For an object $(W,\psi_{i,q}(z),y_{i,q}^\pm(z))$ of $\mathcal M_\ell^\phi$, we set
\begin{align}\label{eq:def-UW}
  U_W=\set{\psi_{i,q}(z),y_{i,q}^\pm(z)}{i\in I}.
\end{align}
Then it is immediate from equations \eqref{eq:cat-M-phi-1}, \eqref{eq:cat-M-phi-2}, \eqref{eq:cat-M-phi-3} and \eqref{eq:cat-M-phi-4} that
$U_W$ is an $\hbar$-adically $S_{trig}$-local subset of $\E_\hbar(W)$.

\begin{de}
Let $\ell\in\C$. Define $\mathcal R_\ell^\phi$ to be the full subcategory of $\mathcal M_\ell^\phi$ consisting of objects $(W,\psi_{i,q}(z),y_{i,q}^\pm(z))$ such that
\begin{align}
  &y_{i,q}^+(z_1)y_{j,q}^-(z_2)-
  \iota_{z_2,z_1}g_{ij}(z_1/z_2)y_{j,q}^-(z_2)y_{i,q}^+(z_1)\nonumber\\
  &\qquad=\frac{\delta_{ij}}{q_i-q_i\inv}\(\delta\(\frac{z_1}{z_2}\)
    -E(\psi_{i,q}(z))\delta\(\frac{q^{-2r\ell}z_2}{z_1}\)\),\label{eq:cat-M-phi-5}\\
  &\(\(y_{i,q}^\pm(z)\)_0^\phi\)^{m_{ij}}y_{j,q}^\pm(z)=0,\quad\te{if }a_{ij}<0.\label{eq:cat-M-phi-6}
\end{align}
\end{de}

\begin{prop}\label{prop:V-tau-phi}
Let $(W,\psi_{i,q}(z),y_{i,q}^\pm(z))$ be an object of $\mathcal R_\ell^\phi$.
Then $W$ becomes a $\phi$-coordinated $V_{\hat\fg,\hbar}(\ell,0)$-module such that
\begin{align*}
  Y_W^\phi(h_{i,\hbar},z)=\psi_{i,q}(z),\quad Y_W^\phi(x_{i,\hbar}^\pm,z)=y_{i,q}^\pm(z),\quad i\in I.
\end{align*}
On the other hand, let $(W,Y_W^\phi)$ be a $\phi$-coordinated $V_{\hat\fg,\hbar}(\ell,0)$-module.
Then $$(W,Y_W^\phi(h_{i,\hbar},z),Y_W^\phi(x_{i,\hbar}^\pm,z))$$ is an object of $\mathcal R_\ell^\phi$.
\end{prop}

\begin{proof}
Let $(W,\psi_{i,q}(z),y_{i,q}^\pm(z))$ be an object of $\mathcal R_\ell^\phi$. Then it is an object of $\mathcal M_\ell^\phi$.
Recall the $\hbar$-adically $S_{trig}$-local subset $U_W$ (see \eqref{eq:def-UW} for details).
From Theorem \ref{thm:phi-construction}, we get an $\hbar$-adic nonlocal vertex algebra $\<U_W\>_\phi$, and $W$ becomes a $\phi$-coordinated $\<U_W\>_\phi$-module with module action $Y_W^\phi(a(z),z_0)=a(z_0)$ for $a(z)\in \<U_W\>_\phi$.
From Proposition \ref{prop:YEcom-h} and equations \eqref{eq:cat-M-phi-1}, \eqref{eq:cat-M-phi-2}, \eqref{eq:cat-M-phi-3}, \eqref{eq:cat-M-phi-4},
we get that $(W,Y_\E^\phi(\psi_{i,q}(z_1),z), Y_\E^\phi(y_{i,q}^\pm(z_1),z))$ is an object of $\mathcal M_\tau$ (see Definition \ref{de:cat-M-tau}).
By using Proposition \ref{prop:universal-nonlocal-va}, we get an $\hbar$-adic nonlocal vertex algebra homomorphism
$\varphi:F_\tau(A,\ell)\to \<U_W\>_\phi$ such that
\begin{align*}
  \varphi(h_{i,\hbar})=\psi_{i,q}(z),\quad \varphi(x_{i,\hbar}^\pm)=y_{i,q}^\pm(z),\quad i\in I.
\end{align*}
Since $(W,\psi_{i,q}(z),y_{i,q}^\pm(z))$ be an object of $\mathcal R_\ell^\phi$.
We apply Proposition \ref{prop:YEcom-h} to \eqref{eq:cat-M-phi-5} and get that
\begin{align*}
  &Y_\E^\phi(y_{i,q}^+(z),z_1)Y_\E^\phi(y_{j,q}^-(z),z_2)
  -\iota_{z_2,z_1}g_{ij,\hbar}(z_1-z_2)
  Y_\E^\phi(y_{j,q}^-(z),z_2)Y_\E^\phi(y_{i,q}^+(z),z_1)
  \\
  &\quad=\delta_{ij}(q_i-q_i\inv)\inv
  \(z_1\inv\delta\(\frac{z_2}{z_1}\)-
  Y_\E^\phi\( E(\psi_{i,q}(z)),  z_2\)z_1\inv\delta\(\frac{z_2-2r\ell\hbar}{z_1}\)\).
\end{align*}
Let the both hand sides act on $1_W$, and take $\Sing_{z_1}\Res_{z_2}z_2\inv$, we get that
\begin{align*}
  &\Sing_{z_1}Y_\E^\phi(y_{i,q}^+(z),z_1)y_{j,q}^-
  =\delta_{ij}(q_i-q_i\inv)\inv\(1_W z_1\inv-
  \theta_{i,q}(z)(z_1+2r\ell\hbar)\inv\).
\end{align*}
Combining this with \eqref{eq:cat-M-phi-6} and Definition \ref{de:V-tau}, we get that $\varphi$ factor through $V_{\hat\fg,\hbar}(\ell,0)$.
Therefore, $W$ becomes a $\phi$-coordinated $V_{\hat\fg,\hbar}(\ell,0)$-module such that ($i\in I$):
\begin{align*}
  &Y_W^\phi(h_{i,\hbar},z_0)=Y_W^\phi(\varphi(h_{i,\hbar}),z_0)=Y_W^\phi(\psi_{i,q}(z),z_0)=\psi_{i,q}(z_0),\\
  &Y_W^\phi(x_{i,\hbar}^\pm,z_0)=Y_W^\phi(\varphi(x_{i,\hbar}^\pm),z_0)=Y_W^\phi(y_{i,q}(z)^\pm,z_0)=y_{i,q}^\pm(z_0).
\end{align*}

On the other hand, let $(W,Y_W^\phi)$ be a $\phi$-coordinated $V_{\hat\fg,\hbar}(\ell,0)$-module.
From Proposition \ref{prop:vacom-h} and equations \eqref{eq:local-h-1}, \eqref{eq:local-h-2}, \eqref{eq:local-h-3}, \eqref{eq:local-h-4}, \eqref{eq:local-h-5}, we have that
\begin{align*}
  & [Y_W^\phi(h_{i,\hbar},z_1),Y_W^\phi(h_{j,\hbar},z_2)]\\
  &\quad=[r_ia_{ij}]_{q^{z_2\pd{z_2}}}[r\ell]_{q^{z_2\pd{z_2}}}
  \(\iota_{z_1,z_2}q^{-r\ell z_2\pd{z_2}}-\iota_{z_2,z_1}q^{r\ell z_2\pd{z_2}}\)\frac{z_2/z_1}{(1-z_2/z_1)^2},\\
  & [Y_W^\phi(h_{i,\hbar},z_1),Y_W^\phi(x_{j,\hbar}^\pm,z_2)]\\
  &\quad=\pm Y_W^\phi(x_{j,\hbar}^\pm,z_2)[r_ia_{ij}]_{q^{z_2\pd{z_2}}}
  \(\iota_{z_1,z_2}q^{-r\ell z_2\pd{z_2}}-\iota_{z_2,z_1}q^{r\ell z_2\pd{z_2}}\)\frac{1+z_2/z_1}{2-2z_2/z_1},\\
  &(1-z_2/z_1)^{\delta_{ij}}(1-q^{-2r\ell}z_2/z_1)^{\delta_{ij}}\\
  &\quad\times\(y_{i,q}^+(z_1)y_{j,q}^-(z_2)
    -\iota_{z_2,z_1}g_{ij,q}(z_1/z_2) y_{j,q}^-(z_2)y_{i,q}^+(z_1)\)=0,\\
  & \iota_{z_1,z_2}f_{ij,q}(z_1,z_2)Y_W^\phi(x_{i,\hbar}^\pm,z_1)Y_W^\phi(x_{j,\hbar}^\pm,z_2)\\
  &\quad  =C_{ij}\iota_{z_2,z_1}f_{ji,q}(z_2,z_1)Y_W^\phi(x_{j,\hbar}^\pm,z_2)Y_W^\phi(x_{i,\hbar}^\pm,z_1),
\end{align*}
and
\begin{align*}
  &Y_W^\phi(x_{i,\hbar}^+,z_1)Y_W^\phi(x_{j,\hbar}^-,z_2)
  -\iota_{z_2,z_1}g_{ij,q}(z_1/z_2)Y_W^\phi(x_{j,\hbar}^-,z_2)Y_W^\phi(x_{i,\hbar}^+,z_1)\\
  =&\frac{\delta_{ij}}{q_i-q_i\inv}\(\delta\(\frac{z_1}{z_2}\)
    -Y_W^\phi(E(h_{i,\hbar}),z)\delta\(\frac{q^{-2r\ell}z_2}{z_1}\)\)\\
  =&\frac{\delta_{ij}}{q_i-q_i\inv}\Bigg(\delta\(\frac{z_1}{z_2}\)-\(\frac{F(r_i+r\ell)}{F(r_i-r\ell)}\)^\half \\
   &\quad\times Y_W^\phi\(
    \exp\(\(-q^{-r\ell\partial} F\(\partial\)h_{i,\hbar}\)_{-1}\)\vac,z\)\delta\(\frac{q^{-2r\ell}z_2}{z_1}\)\Bigg)\\
  =&\frac{\delta_{ij}}{q_i-q_i\inv}\(\delta\(\frac{z_1}{z_2}\)
    -E\(Y_W^\phi(h_{i,\hbar},z)\)\delta\(\frac{q^{-2r\ell}z_2}{z_1}\)\).
\end{align*}
From Definition \ref{de:V-tau}, we also have that
\begin{align*}
  &0=Y_W^\phi\(\(x_{i,\hbar}^\pm\)_0^{m_{ij}}x_{j,\hbar}^\pm,z\)
  =\(\(Y_W^\phi(x_{i,\hbar}^\pm,z)\)_0^\phi\)^{m_{ij}}Y_W^\phi(x_{j,\hbar}^\pm,z)
\end{align*}
for $i,j\in I$ with $a_{ij}<0$.
Therefore, $(W,Y_W^\phi(h_{i,\hbar},z),Y_W^\phi(x_{i,\hbar}^\pm,z))$ is an object of $\mathcal R_\ell^\phi$.
\end{proof}

\section{Quantum affinization algebras}\label{sec:qaff}

Let $\U_\hbar^l(\hat\fg)$ be the unital associative algebra over $\C[[\hbar]]$ topologically generated by the elements \eqref{eq:tqagenerators} subject to relations (Q1)-(Q6). Then $\U_\hbar(\hat\fg)$ is naturally a quotient algebra of $\U_\hbar^l(\hat\fg)$, and $\U_\hbar^l(\hat\fg)$ is naturally a quotient algebra of $\U_\hbar^f(\hat\fg)$.
We call a $\U_\hbar^f(\hat\fg)$-module $W$ a \emph{restricted module} if $W$ is topologically free and
\begin{align*}
  \phi_{i,q}^\pm(z),\quad x_{i,q}^\pm(z)\in\E_\hbar(W).
\end{align*}
In addition, a $\U_\hbar(\hat\fg)$-module (resp. $\U_\hbar^l(\hat\fg)$-module) is said to be \emph{restricted}, if it is restricted as a $\U_\hbar^f(\hat\fg)$-module.
For $i_1,\dots, i_m\in I$, we set
\begin{align}\label{eq:def-x-no}
  &x_{i_1,\dots,i_m,q}^\pm(z_1,\dots,z_m)\\
  =&\(\prod_{1\le r<s\le m}f_{i_r,i_s,q}^\pm(z_r,z_s)\)x_{i_1,q}^\pm(z_1)\cdots x_{i_m,q}^\pm(z_m),\nonumber
\end{align}
where $f_{ij,q}^\pm(z_1,z_2)=(z_1-q_i^{\pm a_{ij}}z_2)(z_1-z_2)^{-\delta_{ij}}$ is defined in \eqref{eq:def-fq} and $C_{ij}$ is defined in \eqref{eq:def-C}.
The following result is proved in \cite[Proposition 5.8]{CJKT-qeala-II-twisted-qaffinization}:
\begin{lem}
Let $W$ be a restricted $\U_\hbar^l(\hat\fg)$-module.
Then for positive integer $m$, and $i_1,\dots,i_m\in I$,
\begin{align}
  x_{i_1,\dots,i_m,q}^\pm(z_1,\dots,z_m)\in\E_\hbar^{(m)}(W).
\end{align}
Moreover, for $\sigma\in S_m$, we have
\begin{align*}
  &x_{i_{\sigma(1)},\dots,i_{\sigma(m)},q}^\pm(z_{\sigma(1)},\dots,z_{\sigma(m)})\\
  =&\(\prod_{\substack{1\le r<s\le m\\ \sigma(r)>\sigma(s)}}C_{i_r,i_s}\)
  x_{i_1,\dots,i_m,q}^\pm(z_1,\dots,z_m).
\end{align*}
\end{lem}

We need the following special case of \cite[Theorem 5.17]{CJKT-qeala-II-twisted-qaffinization}:
\begin{prop}
A restricted $\U_\hbar^l(\hat\fg)$-module is a restricted $\U_\hbar(\hat\fg)$-module if and only if
\begin{align*}
  x_{i,\dots,i,j,q}^\pm(q_i^{\pm a_{ij}}z,q_i^{\pm a_{ij}\pm 2}z,\dots,q_i^{\mp a_{ij}}z,z)=0\quad\te{for }i,j\in I\,\,\te{with}\,\,a_{ij}<0.
\end{align*}
\end{prop}

From the definition of $\U_\hbar(\hat\fg)$, we immediately get that
\begin{lem}\label{lem:sl2-inj}
For each $i\in I$, there is a $\C$-algebra homomorphism from $\jmath_i:\U_\hbar(\wh\ssl_2)\to\U_\hbar(\hat\fg)$ given by $\hbar\mapsto r_i\hbar$, $c\mapsto {rc}/{r_i}$, and
\begin{align}
  h_{1,q}(0)\mapsto \frac{h_{i,q}(0)}{r_i},\quad h_{1,q}(m)\mapsto \frac{h_{i,q}(m)}{[r_i]_q},\quad
  x_{i,q}^\pm(n) \mapsto x_{i,q}^\pm(n).
\end{align}
\end{lem}

Let $d$ be the derivation of $\U_\hbar(\hat\fg)$ defined by
\begin{align}
  [d,\phi_{i,q}^\pm(z)]=-z\pd{z}\phi_{i,q}^\pm(z),\quad [d,x_{i,q}^\pm(z)]=-z\pd{z}x_{i,q}^\pm(z)\quad \te{for } i\in I,
\end{align}
and let $\h:=\oplus_{i\in I}\C h_{i,q}(0)\oplus\C c\oplus \C d$.
For a $\U_\hbar(\hat\fg)$-module $W$ and $\lambda\in \h^\ast$, we denote
\begin{align*}
W_\lambda=\set{v\in W}{h.v=\lambda(h)v,\,h\in\h}.
\end{align*}
A $\U_\hbar(\hat\fg)$-module $W$ is called a \emph{weight module} if $W=\oplus_{\lambda\in\h^\ast}W_\lambda$.
And a weight module $W$ for $\U_\hbar(\hat\fg)$ is said to be \emph{integrable}, if $x_{i,q}^\pm(m)$ acts locally nilpotently for any $i\in I$, $m\in\Z$.
It is immediate to see that $W$ is integrable if and only if $W$ is integrable as a $\U_\hbar(\wh\ssl_2)$-module through $\jmath_i$ for all $i\in I$ (see Lemma \ref{lem:sl2-inj}).
Then we get from \cite[Theorem 7, Theorem 8]{DM-int-rep-qaff} that

\begin{prop}\label{prop:int}
A nontrivial restricted weight $\U_\hbar(\hat\fg)$-module $W$ of level $\ell$ is integrable if and only if
\begin{align*}
  \ell\in\Z_+,\quad\te{and}\quad
  x_{i,\dots,i,q}^\pm(q^{\pm 2r\ell}z,q^{\pm 2r\ell\mp 2r_i}z,\dots,q^{\pm 2r_i}z,z)=0
  \quad\te{for }i\in I.
\end{align*}
\end{prop}

We need another presentation of restricted $\U_\hbar(\hat\fg)$-modules.
A straightforward calculation shows that
\begin{lem}\label{lem:q-rel-alt-pre}
Let $\ell\in\C$ and let $W$ be a restricted $\U_\hbar^f(\hat\fg)$-module of level $\ell$.
Define $\hat x_{i,q}^+(z)=x_{i,q}^+(z)$ and
\begin{align*}
  &\hat\phi_{i,q}(z)=F\(z\pd{z}\)\inv\log\frac{\phi_{i,q}^+(zq^{\half r\ell})}{\phi_{i,q}^-(zq^{-\half r\ell})},\quad
  \hat x_{i,q}^-(z)=x_{i,q}^-(zq^{-r\ell})\phi_{i,q}^+(zq^{-\half r\ell})\inv,
\end{align*}
where $F(z)$ is defined in \eqref{eq:def-op-F}.
Then $(W,\hat\phi_{i,q}(z),\hat x_{i,q}^\pm(z))$ is an object of $\mathcal M_\ell^\phi$.

On the other hand, let $(W,\psi_{i,q}(z),y_{i,q}^\pm(z))$ be an object of $\mathcal M_\ell^\phi$.
Define $\check y_{i,q}^+(z)=y_{i,q}^+(z)$ and
\begin{align}
  &\psi_{i,q}^\pm(z)=\sum_{\pm n>0}\psi_{i,q}(n)z^{-n}+\half \psi_{i,q}(0),\label{eq:def-psi+-}\\
  &\qquad\te{where }
  \psi_{i,q}(z)=\sum_{n\in\Z}\psi_{i,q}(n)z^{-n},\nonumber\\
  &\check\psi_{i,q}^\pm(z)=\exp\(\pm F\(z\pd{z}\)\psi_{i,q}^\pm(zq^{\mp\half r\ell})\),\label{eq:def-check-op1}\\
  &\check y_{i,q}^-(z)=y_{i,q}^-(zq^{r\ell})\check\psi_{i,q}^+(zq^{\half r\ell}).\label{eq:def-check-op2}
\end{align}
Then $W$ becomes a restricted $\U_\hbar^f(\hat\fg)$-module of level $\ell$ with module action
\begin{align*}
  \phi_{i,q}^\pm(z)=\check\psi_{i,q}^\pm(z),\quad x_{i,q}^\pm(z)=\check y_{i,q}^\pm(z),\quad i\in I.
\end{align*}
\end{lem}

We need the following technical result given in \cite{JKLT-Quantum-lattice-va}.

\begin{lem}\label{lem:tech-4}
Let $W$ be a topologically free $\C[[\hbar]]$-module, and
\begin{align*}
  &\al(z)\in \Hom(W,W\wh\ot\C[z,z\inv][[\hbar]]),\quad \beta(z)\in \E_\hbar(W),\\
  &\qquad
  \gamma(z)\in\Hom(W,W\wh\ot\C(z)[[\hbar]]),
\end{align*}
be such that
\begin{align*}
  &[\al(z_1),\al(z_2)]=0=[\beta(z_1),\beta(z_2)]=[\al(z_1),\gamma(z_2)]=[\beta(z_1),\gamma(z_2)],\\
  &[\al(z_1),\beta(z_2)]=\iota_{z_1,z_2}\gamma(z_2/z_1).
\end{align*}
Suppose that
\begin{align*}
  \exp\(A(z)\)=\sum_{n\ge 0}\frac{A(z)^n}{n!}
\end{align*}
is well defined in $\Hom(W,W[[z,z\inv]])$ for $A(z)=\al(z),\beta(z),\gamma(z)$ or $\al(z)+\beta(z)$.
Then we have that
\begin{align*}
  &\exp\(\(\al(z)+\beta(z)\)_{-1}^\phi\)1_W\\
  =&\exp\(\beta(z)\)\exp\(\al(z)\)\exp\(\Res_zz\inv\gamma(e^{-z})/2\).
\end{align*}
\end{lem}

Let $U_W=\set{\hat \phi_{i,q}(z),\hat x_{i,q}^\pm(z)}{i\in I}$. From Lemma \ref{lem:q-rel-alt-pre} , we have that $U_W$ is a $\hbar$-adically $S_{trig}$-subset.
Then by using Theorem \ref{thm:phi-construction}, we get an $\hbar$-adic nonlocal vertex algebra $\<U_W\>_\phi$.

\begin{lem}\label{lem:exp}
For $i\in I$, we have that
\begin{align*}
  E\(\hat\phi_{i,q}(z)\)
  =\phi_{i,q}^-(zq^{-\frac 32 r\ell})\phi_{i,q}^+(zq^{-\half r\ell})\inv.
\end{align*}
\end{lem}

\begin{proof}
Recall the definition of $\hat\phi_{i,q}^\pm(z)$ in \eqref{eq:def-psi+-}.
From the equation \eqref{eq:cat-M-phi-1}, we get that
\begin{align}\label{eq:Q1+-}
  [\hat\phi_{i,q}^+(z_1),\hat\phi_{i,q}^-(z_2)]=[r_ia_{ii}]_{q^{z_2\pd{z_2}}}[r\ell]_{q^{z_2\pd{z_2}}}\iota_{z_1,z_2}q^{-r\ell z_2\pd{z_2}}\frac{z_2/z_1}{(1-z_2/z_1)^2}.
\end{align}
Let
\begin{align*}
  \wt\phi_{i,q}^\pm(z)=-q^{-r\ell z\pd{z}}F\(z\pd{z}\)\hat\phi_{i,q}^\pm(z).
\end{align*}
We note from \eqref{eq:def-check-op1} and \eqref{eq:def-check-op2} that
\begin{align}\label{eq:exp=check-psi}
  \exp\(\wt\phi_{i,q}^+(z)\)=\phi_{i,q}^\pm(zq^{- r\ell\pm\half r\ell})^{\mp 1}.
\end{align}
By using the relation \eqref{eq:Q1+-}, we have that
\begin{align*}
  &-z_1\pd{z_1}z_2\pd{z_2}[\wt\phi_{i,q}^+(z_1),\wt\phi_{i,q}^-(z_2)]\\
  =&\(z_2\pd{z_2}\)^2F\(z_2\pd{z_2}\)^2[r_ia_{ii}]_{q^{z_2\pd{z_2}}}[r\ell]_{q^{z_2\pd{z_2}}}\iota_{z_1,z_2}q^{-r\ell z_2\pd{z_2}}
  \frac{z_2/z_1}{(1-z_2/z_1)^2}\\
  =&\(q^{z_2\pd{z_2}}-q^{-z_2\pd{z_2}}\)^2[r_ia_{ii}]_{q^{z_2\pd{z_2}}}[r\ell]_{q^{z_2\pd{z_2}}}\iota_{z_1,z_2}q^{-r\ell z_2\pd{z_2}}
  \frac{z_2/z_1}{(1-z_2/z_1)^2}\\
  =&\(q^{r_ia_{ii}z_2\pd{z_2}}-q^{-r_ia_{ii}z_2\pd{z_2}}\)\(1-q^{-2r\ell z_2\pd{z_2}}\)\iota_{z_1,z_2}\frac{z_2/z_1}{(1-z_2/z_1)^2}\\
  =&\(q^{r_ia_{ii}z_2\pd{z_2}}-q^{-r_ia_{ii}z_2\pd{z_2}}\)\(1-q^{-2r\ell z_2\pd{z_2}}\)\iota_{z_1,z_2}z_1\pd{z_1}z_2\pd{z_2}\log\(1-z_2/z_1\).
\end{align*}
Notice that $\Res_{z_2}z_2\inv\circ \(z_2\pd{z_2}\)=0$. Then we have that
\begin{align*}
  &\Res_{z_2}z_2\inv[\wt\phi_{i,q}^+(z_1),\wt\phi_{i,q}^-(z_2)]\\
  =&q^{-r\ell z_1\pd{z_1}}F\(z_1\pd{z_1}\)\Res_{z_2}z_2\inv [\hat\phi_{i,q}^+(z_1),\hat\phi_{i,q}^-(z_2)]\\
  =&q^{-r\ell z_1\pd{z_1}}F\(z_1\pd{z_1}\)\Res_{z_2}z_2\inv [r_ia_{ii}]_{q^{z_2\pd{z_2}}}[r\ell]_{q^{z_2\pd{z_2}}}\iota_{z_1,z_2}q^{-r\ell z_2\pd{z_2}}\frac{z_2/z_1}{(1-z_2/z_1)^2}\\
  =&q^{-r\ell z_1\pd{z_1}}F\(z_1\pd{z_1}\)\Res_{z_2}z_2\inv r_ia_{ii}r\ell\iota_{z_1,z_2}\frac{z_2/z_1}{(1-z_2/z_1)^2}=0.
\end{align*}
Hence, we get that
\begin{align*}
  &[\wt\phi_{i,q}^+(z_1),\wt\phi_{i,q}^-(z_2)]\\
  =&\(q^{r_ia_{ii}z_2\pd{z_2}}-q^{-r_ia_{ii}z_2\pd{z_2}}\)\(q^{-2r\ell z_2\pd{z_2}}-1\)\iota_{z_1,z_2}\log\(1-z_2/z_1\)\\
  =&\(q^{r_ia_{ii}z_2/z_1\pd{z_2/z_1}}-q^{-r_ia_{ii}z_2/z_1\pd{z_2/z_1}}\)\(q^{-2r\ell z_2/z_1\pd{z_2/z_1}}-1\)\iota_{z_1,z_2}\log\(1-z_2/z_1\)\\
  =&\iota_{z_1,z_2}\gamma(z_2/z_1),
\end{align*}
where
\begin{align*}
  \gamma(z)=\(q^{r_ia_{ii}z\pd{z}}-q^{-r_ia_{ii}z\pd{z}}\)\(q^{-2r\ell z\pd{z}}-1\)\log\(1-z\).
\end{align*}
Since $\Res_z z\inv \pdiff{z}{2}\log z=0=\Res_z z\inv \pdiff{z}{2}z=\Res_z z\inv \pdiff{z}{2}1$, we get that
\begin{align*}
  &\Res_z z\inv\gamma(e^{-z})
  =\Res_z z\inv
  \(q^{-2r_i\pd{z}}-q^{2r_i\pd{z}}\)\(q^{2r\ell \pd{z}}-1\)\log\(1-e^{-z}\)\\
  =&\Res_z z\inv
  \(q^{-2r_i\pd{z}}-q^{2r_i\pd{z}}\)\(q^{2r\ell \pd{z}}-1\)\log\frac{e^{\half z}-e^{-\half z}}{z/2}\\
  =&\Res_z\log\frac{e^{\half z}-e^{-\half z}}{z/2}  \(q^{2r_i\pd{z}}-q^{-2r_i\pd{z}}\)\(q^{-2r\ell \pd{z}}-1\)z\inv\\
  =&\Res_z\(\frac{1}{z+2r_i\hbar-2r\ell\hbar}-\frac{1}{z-2r_i\hbar-2r\ell\hbar}-\frac{1}{z+2r_i\hbar}+\frac{1}{z-2r_i\hbar}\)
  \log\frac{e^{\half z}-e^{-\half z}}{z/2}\\
  =&\log \frac{F(r_i)F(-r_i+r\ell)}{F(-r_i)F(r_i+r\ell)}=\log \frac{F(r_i-r\ell)}{F(r_i+r\ell)}.
\end{align*}
By using Lemma \ref{lem:tech-4}, we get that
\begin{align*}
  &E\(\hat\phi_{i,q}(z)\)
  =\(\frac{F(r_i+r\ell)}{F(r_i-r\ell)}\)^\half
  \exp\(\(\wt\phi_{i,q}^+(z)+\wt\phi_{i,q}^-(z)\)_{-1}^\phi\)1_W\\
  =&\exp\(\wt\phi_{i,q}^-(z)\)\exp\(\wt\phi_{i,q}^+(z)\)
  =\phi_{i,q}^-(zq^{-\frac 32 r\ell})\phi_{i,q}^+(zq^{-\half r\ell})\inv,
\end{align*}
where the last equation follows from \eqref{eq:exp=check-psi}.
We complete the proof.
\end{proof}

For any positive integer $m$ and $i_1,\dots,i_m\in I$, we define
\begin{align}
  &\hat x_{i_1,\dots,i_m,q}^\pm(z_1,\dots,z_m)\\
  =&\(\prod_{1\le a<b\le m}f_{i_a,i_b,q}^+(z_a,z_b)\)\hat x_{i_1,q}^\pm(z_1)\cdots \hat x_{i_m,q}^\pm(z_m).
\end{align}
From the relation \eqref{eq:cat-M-phi-4}, we have that
\begin{align*}
  \hat x_{i_1,\dots,i_m,q}^\pm(z_1,\dots,z_m)\in\E_\hbar^{(m)}(W).
\end{align*}
Then from a similar argument to the proof of Lemma \ref{lem:qSerre}, we get the following two results.

\begin{lem}\label{lem:qSerre-va}
For $i,j\in I$ with $a_{ij}<0$, we have that
\begin{align*}
  \(\(\hat x_{i,q}^\pm(z)\)_0^\phi\)^{m_{ij}}\hat x_{j,q}^\pm(z)=c \hat x_{i,\dots,i,j,q}^\pm(q_i^{a_{ij}}z,q_i^{a_{ij}+2}z,\dots,q_i^{-a_{ij}}z,z)
\end{align*}
for some invertible element $c\in \C[[\hbar]]$.
\end{lem}

\begin{lem}\label{lem:qInt-va}
Suppose that $\ell\in\Z_+$.
For any $i\in I$, we have that
\begin{align*}
  \(\(\hat x_{i,q}^\pm(z)\)_{-1}^\phi\)^{r\ell/r_i}\hat x_{i,q}^\pm(z)=c \hat x_{i,\dots,i,q}^\pm(q^{2r\ell}z,q^{2r\ell-2r_i}z,\dots,q^{2r_i}z,z)
\end{align*}
for some invertible element $c\in \C[[\hbar]]$.
\end{lem}

Then we have that
\begin{prop}\label{prop:q-rel-alt}
Let $W$ be a restricted $\U_\hbar(\hat\fg)$-module of level $\ell$. Then $(W,\hat\phi_{i,q}(z),\hat x_{i,q}^\pm(z))$ is an object of $\mathcal R_\ell^\phi$. On the other hand, let $(W,\psi_{i,q}(z),y_{i,q}^\pm(z))$ be an object of $\mathcal R_\ell^\phi$.
Then $W$ is a restricted $\U_\hbar(\hat\fg)$-module of level $\ell$ with the module action defined by
\begin{align*}
  \phi_{i,q}^\pm(z)=\check\psi_{i,q}^\pm(z),\quad x_{i,q}^\pm(z)=\check y_{i,q}^\pm(z)\quad\te{for } i\in I.
\end{align*}
\end{prop}

\begin{proof}
From Lemma \ref{lem:q-rel-alt-pre}, we get that $W$ is a restricted $\U_\hbar^f(\hat\fg)$-module of level $\ell$ if and only if $(W,\hat\phi_{i,q}(z),\hat x_{i,q}^\pm(z))$ is an object of $\mathcal M_\ell^\phi$.
Then by a straightforward calculation, we get that $W$ is a restricted $\U_\hbar(\hat\fg)$-module of level $\ell$ if and only if
\begin{align*}
  &\hat x_{i,q}^+(z_1)\hat x_{j,q}^-(z_2)-g_{ji,q}(z_1/z_2)\hat x_{j,q}^-(z_2)\hat x_{i,q}^+(z_1)\\
  &\qquad=\frac{\delta_{ij}}{q_i-q_i\inv}
  \(\delta\(\frac{z_2}{z_1}\)-\phi_{i,q}^-(zq^{-\frac 32 r\ell})\phi_{i,q}^+(zq^{-\half r\ell})\inv
    \delta\(\frac{z_2 q^{-2r\ell}}{z_1}\)\),\\
  &\hat x_{i,\dots,i,j,q}^\pm(q_i^{a_{ij}}z,q_i^{a_{ij}+2}z,\dots,q_i^{-a_{ij}}z,z)=0,\quad \te{if }a_{ij}<0.
\end{align*}
Therefore, the proposition follows immediate from Lemmas \ref{lem:exp} and \ref{lem:qSerre-va}.
\end{proof}

\begin{prop}\label{prop:int-alt}
Let $\ell\in\Z_+$,
and let $W$ be a restricted weight $\U_\hbar(\hat\fg)$-module of level $\ell$.
Then $W$ is integrable if and only if
\begin{align}\label{eq:cat-M-phi-7}
  &\(\(\hat x_{i,q}^\pm(z)\)_{-1}^\phi\)^{r\ell/r_i}\hat x_{i,q}^\pm(z)=0,\quad i\in I.
\end{align}
\end{prop}

\begin{proof}
From Proposition \ref{prop:int} and the definition of $\hat x_{i,q}^\pm(z)$ (see Lemma \ref{lem:q-rel-alt-pre}), we have that $W$ is integrable if and only if
\begin{align*}
  \hat x_{i,\dots,i,q}^\pm(q^{2r\ell}z,q^{2r\ell- 2r_i}z,\dots,q^{2r_i}z,z)=0
  \quad\te{for }i\in I.
\end{align*}
Then the proposition follows from Lemma \ref{lem:qInt-va} and Proposition \ref{prop:q-rel-alt}.
\end{proof}

Combining Propositions \ref{prop:V-tau-phi} and \ref{prop:q-rel-alt}, we get that
\begin{thm}\label{thm:res}
Let $\ell\in\C$.
Then the category of restricted $\U_\hbar(\hat\fg)$-modules of level $\ell$ is isomorphic to the category of $\phi$-coordinated $V_{\hat\fg,\hbar}(\ell,0)$-modules.
To be more precise, let $W$ be a restricted $\U_\hbar(\hat\fg)$-module of level $\ell$.
Then $W$ becomes a $\phi$-coordinated $V_{\hat\fg,\hbar}(\ell,0)$-module such that
\begin{align*}
  Y_W^\phi(h_{i,\hbar},z)=\hat\phi_{i,q}(z),\quad Y_W^\phi(x_{i,\hbar}^\pm,z)=\hat x_{i,q}^\pm(z),\quad i\in I.
\end{align*}
On the other hand, let $(W,Y_W^\phi)$ be a $\phi$-coordinated $V_{\hat\fg,\hbar}(\ell,0)$-module. Then $W$ becomes a restricted $\U_\hbar(\hat\fg)$-module of level $\ell$ such that
\begin{align*}
  \hat\phi_{i,q}(z)=Y_W^\phi(h_{i,\hbar},z),\quad \hat x_{i,q}^\pm(z)=Y_W^\phi(x_{i,\hbar}^\pm,z),\quad i\in I.
\end{align*}
\end{thm}

A $\phi$-coordinated $V_{\hat\fg,\hbar}(\ell,0)$-module $(W,Y_W^\phi)$ is called a \emph{weight module} if $\(h_{i,\hbar}\)_0^\phi$ acts semi-simply for all $i\in I$, and there exists
a $\C[[\hbar]]$-module map $\partial:W\to W$, such that
\begin{align}
  [\partial, Y_W^\phi(u,z)]=z\pd{z}Y_W^\phi(u,z)\quad \te{for }u\in V_{\hat\fg,\hbar}(\ell,0).
\end{align}
A $\phi$-coordinated $L_{\hat\fg,\hbar}(\ell,0)$ is called a \emph{weight module}
if it is a weight module viewed as a $\phi$-coordinated $V_{\hat\fg,\hbar}(\ell,0)$-module.
Combining Proposition \ref{prop:int-alt} with Theorem \ref{thm:res}, we get that
\begin{thm}\label{thm:int}
Let $\ell\in \Z_+$. Then the category isomorphism obtained in Theorem \ref{thm:res} induces a category isomorphism between the category of restricted integrable $\U_\hbar(\hat\fg)$-modules of level $\ell$ and the category of $\phi$-coordinated weight modules of $L_{\hat\fg,\hbar}(\ell,0)$.
\end{thm}

\begin{thm}
Suppose that the GCM $A$ is of finite type.
Let $\ell\in \Z_+$. Then there is a vertex algebra isomorphism $L_{\hat\fg}(\ell,0)\to L_{\hat\fg,\hbar}(\ell,0)/\hbar L_{\hat\fg,\hbar}(\ell,0)$ defined by
\begin{align*}
  h_i\mapsto h_{i,\hbar},\quad x_i^\pm\mapsto x_{i,\hbar}^\pm,\quad i\in I.
\end{align*}
\end{thm}

\begin{proof}
Recall from Proposition \ref{prop:classical-limit-V-L-pre} that there is a surjective vertex algebra homomorphism $\varphi:L_{\hat\fg}(\ell,0)\to L_{\hat\fg,\hbar}(\ell,0)/\hbar L_{\hat\fg,\hbar}(\ell,0)$
determined by $h_i\mapsto h_{i,\hbar}$ and $x_i^\pm\mapsto x_{i,\hbar}^\pm$ for $i\in I$.
It is straightforward to verify that both $L_{\hat\fg}(\ell,0)$ and $L_{\hat\fg,\hbar}(\ell,0)/\hbar L_{\hat\fg,\hbar}(\ell,0)$ are $\hat\fg$-modules such that
\begin{align*}
  &a_i(z).u=Y(a_i,z)u\,\,\te{for }u\in L_{\hat\fg}(\ell,0)\quad\te{and}\quad\\
  &a_i(z).v=Y_\tau(a_{i,\hbar},z)v\,\,\te{for }v\in L_{\hat\fg,\hbar}(\ell,0)/\hbar L_{\hat\fg,\hbar}(\ell,0),
\end{align*}
where $a=h$ or $x^\pm$.
Then $\varphi$ is also a surjective $\hat\fg$-module homomorphism.
Since the GCM $A$ is of finite type, we get from Remark \ref{rem:aff-Lie-alg} that $L_{\hat\fg}(\ell,0)$ is a simple $\hat\fg$-module.
Then $\varphi$ is either an isomorphism or $0$.

Suppose that $\varphi=0$. Since $L_{\hat\fg,\hbar}(\ell,0)$ is topologically free, we get that
$L_{\hat\fg,\hbar}(\ell,0)=0$.
Then $L_{\hat\fg,\hbar}(\ell,0)$ has only trivial $\phi$-coordinated modules.
From Theorem \ref{thm:int}, we get that the $\U_\hbar(\hat\fg)$ has only trivial restricted integrable modules of level $\ell$,
which contradict to \cite[Section 4.13]{Lusztig-qdeform-simple-mod}.
Therefore, $\varphi$ must be an isomorphism.
\end{proof}


\begin{thebibliography}{HJKOS}

\bibitem[BJK]{BJK-qva-BCD}
M.~Butorac, N.~Jing, and S.~Ko\v{z}i\'{c}.
\newblock $\hbar$-adic quantum vertex algebras associated with rational
  ${R}$-matrix in types ${B}$, ${C}$ and ${D}$.
\newblock {\em Lett. Math. Phys.}, 109:2439--2471, 2019.

\bibitem[CJKT]{CJKT-qeala-II-twisted-qaffinization}
F.~Chen, N.~Jing, F.~Kong, and S.~Tan.
\newblock {{T}wisted quantum affinization and quantization of extended affine
  {L}ie algebras}.
\newblock {\em Trans. Amer. Math. Soc.},
\newblock 2022, DIO: 10.1090/tran/8706.

\bibitem[DF]{DF-qaff-RTT-Dr}
J.~Ding and I.~Frenkel.
\newblock Isomorphism of two realizations of quantum affine algebra
  $\mathcal{U}_q(\widehat{\mathfrak{gl}(n)})$.
\newblock {\em Comm. Math. Phys.}, 156:277--300, 1993.

\bibitem[DL]{DL}
C.~Dong and J.~Lepowsky.
\newblock {\em {Generalized Vertex Algebras and Relative Vertex Operators}},
  volume 112 of {\em Prog. Math.}
\newblock Birkh\"{a}user, Boston, 1993.

\bibitem[DLM1]{DLM}
C.~Dong, H.~Li, and G.~Mason.
\newblock {Regularity of rational vertex operator algebras}.
\newblock {\em Adv. Math.}, 132:148--166, 1997.

\bibitem[DLM2]{DLM-vertex-Lie-vertex-Poisson-VA}
C.~Dong, H.~Li, and G.~Mason.
\newblock Vertex {L}ie algebras, vertex poisson algebras and vertex algebras.
\newblock {\em Contemp. Math.}, 297:69--96, 2002.

\bibitem[DM]{DM-int-rep-qaff}
J.~Ding and T.~Miwa.
\newblock Quantum current operators - {I}. {Z}eros and poles of quantum current
  operators and the condition of quantum integrability.
\newblock {\em Publ. RIMS, Kyoto Univ.}, 33:277--284, 1997.

\bibitem[Dr1]{Dr-hopf-alg}
V.~Drinfeld.
\newblock Hopf algebras and quantum yang-baxter equation.
\newblock {\em Soviet Math. Dokl.}, 283:1060--1064, 1985.

\bibitem[Dr2]{Dr-new}
V.~{D}rinfeld.
\newblock A new realization of {Y}angians and quantized affine algebras.
\newblock In {\em Soviet Math. Dokl}, volume~36, pages 212--216, 1988.

\bibitem[EK]{EK-qva}
P.~Etingof and D.~Kazhdan.
\newblock Quantization of {L}ie bialgebras, {P}art {V}: {Q}uantum vertex
  operator algebras.
\newblock {\em Selecta Math.}, 6:105, 2000.

\bibitem[FHL]{fhl}
 I. B. Frenkel, Y.-Z. Huang, and J. Lepowsky, {\em On axiomatic approaches to
vertex operator algebras and modules}, Memoirs Amer. Math. Soc. 104,
1993.

\bibitem[FJ]{FJ-vr-qaffine}
I.~Frenkel and N.~Jing.
\newblock Vertex representations of quantum affine algebras.
\newblock {\em Proc. Nat. Acad. Sci. U.S.A.}, 85:9373--9377, 1988.

\bibitem[FZ]{FZ}
I.~Frenkel and Y.~Zhu.
\newblock Vertex operator algebras associated to representations of affine and
  {V}irasoro algebras.
\newblock {\em Duke Math. J.}, 66:123--168, 1992.

\bibitem[Gar]{Gar-loop-alg}
Howard Garland.
\newblock The arithmetic theory of loop algebras.
\newblock {\em J. Algebra}, 53:480 -- 551, 1978.

\bibitem[GKV]{GKV}
V.~Ginzburg, M.~Kapranov, and E.~Vasserot.
\newblock Langlands reciprocity for algebraic surfaces.
\newblock {\em Math. Res. Lett.}, 2:147--160, 1995.

\bibitem[Jim]{jim-start}
M.~Jimbo.
\newblock A q-difference analogue of ${U}(\mathfrak g)$ and the {Y}ang-{B}axter
  equation.
\newblock In {\em Yang-{B}axter {E}quation {I}n {I}ntegrable {S}ystems}, pages
  292--298. World Scientific, 1990.

\bibitem[Jin]{J-KM}
N.~Jing.
\newblock Quantum {K}ac-{M}oody algebras and vertex representations.
\newblock {\em Lett. Math. Phys.}, 44:261--271, 1998.

\bibitem[JKLT1]{JKLT-G-phi-mod}
N.~Jing, F.~Kong, H.~Li, and S.~Tan.
\newblock $({G},\chi_{\phi})$-equivariant $\phi$-coordinated quasi modules for
  nonlocal vertex algebras.
\newblock {\em J Algebra}, 570:24--74, 2021.

\bibitem[JKLT2]{JKLT-Defom-va}
N.~Jing, F.~Kong, H.~Li, and S.~Tan.
\newblock Deforming vertex algebras by vertex bialgebras.
\newblock {\em Comm. Cont. Math.}, 2022, DIO: 10.1142/S0219199722500675.

\bibitem[JKLT3]{JKLT-Quantum-lattice-va}
N.~Jing, F.~Kong, H.~Li, and S.~Tan.
\newblock Twisted quantum affine algebras and equivariant $\phi$-coordinated
  modules for quantum vertex algebras.
\newblock arXiv:2212.01895.

\bibitem[JLM1]{JLM-qaff-RTT-Dr-BD}
N.~Jing, M.~Liu, and A.~Molev.
\newblock Isomorphism between the ${R}$-matrix and {D}rinfeld presentations of
  quantum affine algebra: {T}ype ${B}$ and ${D}$.
\newblock {\em SIGMA}, 16:043, 2020.

\bibitem[JLM2]{JLM-qaff-RTT-Dr-C}
N.~Jing, M.~Liu, and A.~Molev.
\newblock Isomorphism between the ${R}$-matrix and {D}rinfeld presentations of
  quantum affine algebra: {T}ype ${C}$.
\newblock {\em J. Math. Phys.}, 61:031701, 2020.

\bibitem[Kac]{Kac-VA}
V.~Kac.
\newblock {\em Vertex algebras for beginners}.
\newblock Number~10. Amer. Math. Soc., 1998.

\bibitem[Kas]{Kassel-topologically-free}
C.~Kassel.
\newblock {\em Quantum groups, volume 155 of Graduate Texts in Mathematics}.
\newblock Springer-Verlag, New York, 1995.

\bibitem[Ko1]{K-qva-phi-mod-BCD}
S.~Ko\v{z}i\'{c}.
\newblock $\hbar$-adic quantum vertex algebras in types ${B}$, ${C}$, ${D}$ and
  their $\phi$-coordinated modules.
\newblock {\em J. Phys. A: Math. Theor.}, 54:485202, 2021.

\bibitem[Ko2]{Kozic-qva-tri-A}
S.~Ko\v{z}i\'{c}.
\newblock On the quantum affine vertex algebra associated with trigonometric
  ${R}$-matrix.
\newblock {\em Selecta Math. (N. S.)}, 27:45, 2021.

\bibitem[Li1]{Li-local}
H.~Li.
\newblock {Local systems of vertex operators, vertex superalgebras and
  modules}.
\newblock {\em J. Pure Appl. Algebra}, 109:143--195, 1996.

\bibitem[Li2]{li-g1}
H.~Li.
\newblock Axiomatic ${G}_{1}$-vertex algebras.
\newblock {\em Comm. Cont. Math.}, 5:1--47, 2003.

\bibitem[Li3]{Li-pseudo}
H.~Li.
\newblock Pseudoderivations, pseudoautomorphisms and simple current modules for
  vertex algebras.
\newblock In {\em Infinite-dimensional Aspects of Representation Theory and
  Applications: International Conference on Infinite-Dimensional Aspects of
  Representation Theory and Applications, May 18-22, 2004, University of
  Virginia, Charlottesville, Virginia}, volume 392, page~55. Amer. Math. Soc.,
  2005.

\bibitem[Li4]{Li-nonlocal}
H.~Li.
\newblock Nonlocal vertex algebras generated by formal vertex operators.
\newblock {\em Selecta Math.}, 11:349, 2006.

\bibitem[Li5]{Li-smash}
H.~Li.
\newblock A smash product construction of nonlocal vertex algebras.
\newblock {\em Comm. Cont. Math.}, 9:605--637, 2007.

\bibitem[Li6]{Li-h-adic}
H.~Li.
\newblock {$\hbar$-adic quantum vertex algebras and their modules}.
\newblock {\em Comm. Math. Phys.}, 296:475--523, 2010.

\bibitem[Li7]{Li-phi-coor}
H.~Li.
\newblock {$\phi$-coordinated quasi-modules for quantum vertex algebras}.
\newblock {\em Comm. Math. Phys.}, 308:703--741, 2011.

\bibitem[Li8]{Li-G-phi}
H.~Li.
\newblock G-equivariant $\phi$-coordinated quasi modules for quantum vertex
  algebras.
\newblock {\em J. Math. Phys.}, 54:051704, 2013.

\bibitem[LL]{LL}
J.~Lepowsky and H.~Li.
\newblock {\em {Introduction to vertex operator algebras and their
  representations}}, volume 227.
\newblock Birkh\"{a}user Boston Incoporation, 2004.

\bibitem[LS]{LS-twisted-tensor}
H~Li and J.~Sun.
\newblock Twisted tensor products of nonlocal vertex algebras.
\newblock {\em Journal of Algebra}, 345:266 -- 294, 2011.

\bibitem[Lu]{Lusztig-qdeform-simple-mod}
G.~Lusztig.
\newblock Quantum deformations of certain simple modules over enveloping
  algebras.
\newblock {\em Adv. Math.}, 70:237--249, 1988.

\bibitem[MERY]{MRY}
R.~Moody, S.~E.~Rao, and T.~Yokonuma.
\newblock {Toroidal {L}ie algebras and vertex representations}.
\newblock {\em Geom. Dedicata}, 35:283--307, 1990.

\bibitem[MP1]{MP1}
A.~Meurman and M.~Primc.
\newblock {Vertex Operator Algebras and Representations of Affine {L}ie
  Algebras}.
\newblock {\em Acta Appl. Math.}, 44:207--215, 1996.

\bibitem[MP2]{MP2}
A.~Meurman and M.~Primc.
\newblock {Annihilating Fields of Standard Modules of
  $\widetilde{\mathfrak{sl}(2,{\mathbb C})}$ and Combinatorial Identities}.
\newblock {\em Mem. Amer. Math. Soc.}, 652, 1999.

\bibitem[N]{Naka-quiver}
H.~Nakajima.
\newblock Quiver varieties and finite dimensional representations of quantum
  affine algebras.
\newblock {\em J. Amer. Math. Soc.}, 14:145--238, 2001.

\bibitem[P]{Primc-VA-gen-by-Lie}
M.~Primc.
\newblock Vertex algebras generated by {L}ie algebras.
\newblock {\em J. Pure Appl. Algebra}, 135:253--293, 1999.

\bibitem[R]{R-free-conformal-free-va}
M.~Roitman.
\newblock On free conformal and vertex algebras.
\newblock {\em J. Algebra}, 217:496--527, 1999.

\bibitem[RSTS]{RS-RTT}
Y.~Reshetikhin and A.~Semenov-{T}ian {S}hansky.
\newblock Central extensions of quantum current groups.
\newblock {\em Lett. Math. Phys.}, 19:133--142, 1990.

\end{thebibliography}
\end{document}